\documentclass[12pt]{article}
\usepackage{amsmath,amssymb,latexsym, amsfonts, amscd, amsthm}
\usepackage{graphicx,epsfig}

\theoremstyle{plain}
\newtheorem{theorem}{Theorem}[section]

\newtheorem{proposition}[theorem]{Proposition}
\newtheorem{corollary}[theorem]{Corollary}

\newtheorem{lemma}[theorem]{Lemma}

\newtheorem{example}[theorem]{Example}

\theoremstyle{definition}
\newtheorem{definition}[theorem]{Definition}
\newtheorem{remark}[theorem]{Remark}

\newcommand{\proket}{{\mathrm{pke}}}
\newcommand{\profket}{{\mathrm{profket}}}

\newcommand{\ket}{{\mathrm{ket}}}
\newcommand{\fket}{{\mathrm{fket}}}

\newcommand{\lra}{\longrightarrow}

\newcommand{\Spec}{\mbox{Spec}}

\newcommand{\Fil}{\mbox{\rm Fil}}
\newcommand{\Iw}{\mbox{Iw}}

\newcommand{\GL}{{\rm \mathbf{GL}}}

\newcommand{\Z}{{\mathbb Z}}
\newcommand{\Q}{{\mathbb Q}}
\newcommand{\C}{{\mathbb C}}

\newcommand{\F}{{\mathbb F}}
\newcommand{\N}{{\mathbb N}}
\newcommand{\V}{{\mathbb V}}
\newcommand{\WW}{{\mathbb W}}

\newcommand{\cM}{{\cal M}}

\newcommand{\cC}{{\cal C}}
\newcommand{\cI}{{\cal I}}

\newcommand{\cU}{{\cal U}}

\newcommand{\cT}{{\cal T}}
\newcommand{\cF}{{\cal F}}
\newcommand{\cG}{{\cal G}}
\newcommand{\cE}{{\cal E}}
\newcommand{\cV}{{\cal V}}
\newcommand{\cA}{{\cal A}}
\newcommand{\PP}{{\mathbb P}}

\newcommand{\bD}{{\mathbb D}}

\newcommand{\bV}{{\mathbb V}}

\newcommand{\cB}{\mathcal{B}}

\newcommand{\cD}{{\cal D}}
\newcommand{\cO}{{\cal O}}

\newcommand{\cW}{{\cal W}}
\newcommand{\cZ}{{\cal Z}}

\newcommand{\fD}{{\mathfrak{D}}}
\newcommand{\fFil}{{\mathfrak{Fil}}}

\newcommand{\Symm}{{\rm Sym}}

\newcommand{\cY}{\mathcal{Y}}
\newcommand{\cX}{\mathcal{X}}
\newcommand{\cQ}{\mathcal{Q}}

\begin{document}

\title{Overconvergent de Rham Eichler-Shimura morphisms.}
\author{Fabrizio Andreatta \\ Adrian Iovita}
\maketitle

\tableofcontents \pagebreak

\section{Introduction.}

This article represents our attempt to improve the previous results on defining and understanding 
overconvergent Eichler-Shimura maps in \cite{EichlerShimura} and \cite{half}.
 
We fix a prime integer $p>2$.
We recall that an (overconvergent) Eichler-Shimura morphism is a comparison map describing weight $k$ overconvergent modular symbols, seen as pro-Kummer \'etale cohomology classes of a sheaf of weight $k$ distributions $\bD_k$ (where $k\colon \Z_p^\ast\rightarrow B^\ast$ is a $B$-valued weight as in Definition \ref{def:weights}) tensored some period ring, in terms of overconvergent modular forms of weight $k+2$, tensored with the same period ring. In \cite{EichlerShimura} and in \cite{half} we defined and studied Hodge-Tate Eichler-Shimura maps while in this article  we'll have Hodge-Tate, de Rham and crystalline variants.

To really explain what are the main issues we deal with in this article, let us observe that there has been remarkable recent progress in $p$-adic Hodge theory and especially in integral $p$-adic Hodge theory, let us just mention \cite{bhatt_morrow_scholze}, \cite{bhatt_scholze},  \cite{cesnavicius_koshikawa}, \cite{colmez_dospinescu_niziol}. The articles quoted here deal with various cohomology theories on formal schemes or adic spaces with constant coefficients. On the other hand it has been clear for some time that for applications to $p$-adic automorphic forms one needs to work with cohomology with very large coefficients. In this article we try to understand $p$-adic Hodge theory (comparison morphisms really) with large coeffcients and therefore, unfortunately, we cannot use the recent results quoted above.

More precisely, let $\cX:=\cX_0(p^m, N)$ be the log adic space defined by  the modular curve over $\Q_p$ associated  to the congruence subgroup $\Gamma_0(p^m) \cap \Gamma_1(N)$. For any $r\ge 0$ we have open subspaces $\cX\bigl(p/\mathrm{Ha}^{p^r}\bigr)$ where $\mathrm{Ha}$ is a (any) local lift of the Hasse invariant. We fix $k, B$ as above and let $b\in \N$. Then, Theorem \ref{thm:hodgetate} states that there is such a neighborhood and a canonical $\C_p$-linear, Galois and Hecke equivariant map:
$$
\Psi_{\rm HT, b}\colon {\rm H}^1\bigl(\cX_\proket, \bD_k(1)\bigr)^{(b)}\widehat{\otimes}\C_p\lra {\rm H}^0\bigl(
\cX\bigl(p/\mathrm{Ha}^{p^r}\bigr), \omega^{k+2}  \bigr)^{(b)}\widehat{\otimes}\C_p.
$$
Here the exponent $(b)$ on a module indicates the submodule of slope $\le b$ for the action of the $U_p$-operator. 
In a slightly different way and only for analytic weights, the map $\Psi_{\rm HT,b}$ was constructed
in \cite{EichlerShimura}. There we also proved that the map is generically surjective. The new result in this paper is: 

\begin{theorem}If $\prod_{i=0}^{b-1}(u_k-i)\in \bigl(B[1/p]\bigr)^\ast$ then $\Psi_{\rm HT,b}$ is surjective, for all $b\ge 1$ and it is surjective if $b=0$.
\end{theorem}
Our next result in this article is a de Rham overconvergent Eichler-Shimura map, i.e. we prove, assuming the hypothesis and notations above:

\begin{theorem} a) There is a natural, Galois and Hecke equivariant $B\widehat{\otimes}B_{\rm dR}^+$-semilinear map
$$(\ast)\quad \rho_k\colon \mathrm{H}^1\bigl(\cX_{\overline{K},\proket}, \bD_k^o(T_0^\vee)[n] \bigr)^{(b)}\widehat{\otimes} B_{\rm dR}^+\lra \mathrm{H}^1_{\rm dR}\bigl(\cX\bigl(p/\mathrm{Ha}^{p^r}\bigr), \mathbf{W}_{k, \rm dR}\bigr)^{(b)}\widehat{\otimes} \Fil^{-1}B_{\rm dR},$$
where $\mathbf{W}_{k, dR}$ is a modular sheaf with connection and a filtration which interpolates 
$p$-adically the family of sheaves $\{{\rm Sym}^r{\rm H}^1_{\rm dR}\bigl(E/\cX  \bigr)\}_{r\in \N}$ with their filtrations and connections.  

b) If $\prod_{i=0}^{b-1}(u_k-i)\in \bigl(B[1/p]\bigr)^\ast$ then the display $(\ast)$ above becomes:
$$\rho_k\colon \mathrm{H}^1\bigl(\cX_{\overline{K},\proket}, \bD_k^o(T_0^\vee)[n] \bigr)^{(b)}\widehat{\otimes} B_{\rm dR}^+\lra \mathrm{H}^0\bigl(\cX\bigl(p/\mathrm{Ha}^{p^r}\bigr), \omega^{k+2}\bigr)^{(b)}\widehat{\otimes} \Fil^{-1}B_{\rm dR}$$ 
and it is surjective.
\end{theorem}

In order to make it clear what improvements we were able to produce in this article we now list the new ideas.
 
\

1){\it  Neighbourhoods of the ordinary loci in modular curves.}

\

Both in \cite{EichlerShimura} and 
\cite{half} we worked on the (log) adic modular curves $\cX_1(N)$ and $\cX_0(p, N)$; these are the (log) adic spaces associated to the modular curves over $\Q_p$ of level $\Gamma_1(N)$ and respectively $\Gamma_1(N)\cap \Gamma_0(p)$, which have a connected respectively two connected components of ordinary loci. We worked with strict neighbourhoods of these ordinary loci of deapth $n\in \N$ defined as the points $x$ with the property $v_x({\rm Ha})\le 1/n$ These neighbourhoods are defined over ${\rm Spa}(L, \cO_L)$
over which the point $x$ is defined, where $L$ is some complete extension of $\Q_p$ and these neighbourhoods are also used in this very article for the de Rham Eichler-Shimura maps.

For the Hodge-Tate comparison maps in this article we use a better technology, inspired by the work of 
\cite{CHJ}.
Namely by considering the perfectoid adic space $\cX(p^\infty, N)$ over ${\rm Spa}(\C_p, \cO_{\C_p})$ associated to the projective limit of adic modular curves $\displaystyle \lim_{\leftarrow, m}\cX(p^m, N)$  and the Hodge-Tate period map
$$\pi_{\rm HT}\colon \cX(p^\infty, N)\lra \PP^1_{\Q_p},$$ we define interesting opens $U_\#^{(n)}\subset \PP^1_{\Q_p}$, for $n\ge 1$ and the symbol $\#\in \{0, \infty\}$, which are invariant under the action of the $m$-th Iwahori subgroup ${\rm Iw}_m\subset \GL_2(\Z_p)$ as follows: if $\#=\infty$ then $m\ge n$ if $\#=0$ then $m\ge 1$ and  on which we understand the dynamic of the $U_p$-operator. Then by the properties of $\pi_{\rm HT}$, for every $n\ge 1$ there are: an $m\ge 1$ as above and neighbourhoods of the ordinary 
loci in $\cX_0(p^m, N)$ denoted $\cX_0^{(n)}, \cX_\infty^{(n)}$ such that if  $\pi_m\colon \cX(p^\infty, N)\rightarrow \cX_0(p^m, N)$ is the natural projection, then $\pi_{\rm HT}^{-1}(U_\#^{(n)})=\pi_m^{-1}(\cX_\#^{(n)})$, where $\#\in \{0, \infty  \}$, and we understand very well the dynamic of the $U_p$-operator on sections of modular sheaves on $\cX_0(p^m, N)$. We remark that $\cX_0(p^m, N)$ has many connected components of the ordinary locus if $m$ is large  and a complicated semistable integral model, therefore it would have been difficult to apply the previous method, i.e. defining neighbourhoods of the ordinary loci using ${\rm Ha}$,
in $\cX_0(p^m, N)$ for $m>1$.

\

2) {\it Payman Kassaei's method for the cohomology of pro-Kummer \'etale sheaves.}

\

Let us now explain our new take on the overconvergent Hodge-Tate Eichler-Shimura morphism. We fix a 
slope $b\in \N$ and a weight $k\colon \Z_p^\ast\lra B^\ast$ as in Definition \ref{def:weights}. This weight is $N$-analytic, for some $N\in \N$, i.e. there is 
$u_k\in B[1/p]$ such that $k(t)={\rm exp}\bigl(u_k\log(t)\bigr)$ for all $t\in 1+p^N\Z_p$.
This data determines integers $n, r, m$ such that on $\cX:=\cX_0(p^m, N)$ we have our neighbourhoods 
$\cX_\infty^{(\mu)}$ for $\mu\le m$ and $\cX_0^{(n)}, \cX_0^{(n+1)}$. We base-change $\cX, \cX_\infty^{(\mu)}, \cX_0^{(n)}, \cX_0^{(n+1)}$ over ${\rm Spa}(\Q_p, \Z_p)$ to ${\rm Spa}(B[1/p], B)$ and still denote them   $\cX, \cX_\infty^{(\mu)}, \cX_0^{(n)}, \cX_0^{(n+1)}$.
We let $\bD_k^o$ be the integral sheaf of weight $k$-distributions, seen as a pro-Kummer \'etale sheaf on $\cX$ and denote by $\bD_k:=\bD_k^o\otimes_{\Z_p}\Q_p$.

First let us recall that the map $\Psi_{\rm HT,b}$ appears, after passing to the the open subspaces defined in (1), as the following composition:

$$
{\rm H}^1\bigl(\cX_\proket, \bD_k(1)  \bigr)^{(b)}\widehat{\otimes}\C_p\cong \Bigl({\rm H}^1\bigl(\cX_\proket, \fD_k (1)  \bigr)[1/p]\Bigr)^{(b)}\stackrel{\mathcal{R}}{\lra}\Bigl({\rm H}^1\bigl((\cX_\infty^{(u)})_\proket, \fD_k(1)\bigr)[1/p]\Bigr)^{(b)}\stackrel{\Phi}{\lra} 
$$

$$
\stackrel{\Phi}{\lra} {\rm H}^0\bigl(\cX_\infty^{(u)}, \omega^{k+2}   \bigr)^{(b)}\widehat{\otimes} \C_p,
$$
where $\fD_k:=\bD_k\widehat{\otimes}\cO_{\cX_\proket}$, $\mathcal{R}$ is the restriction map and $\Phi$ was defined in \cite{half} and in \ref{prop:HeckeWkinfty} and it was proved in loc.~cit.~that it is an isomorphism. 
Therefore in order to prove Theorem \ref{thm:hodgetate} we need to show that restriction from 
$\cX_\proket$ to $(\cX_\infty^{(u)})_\proket$ induces a surjective map on the ${\rm H}^1$'s, if 
$\prod_{i-0}^{b-1}(u_k-i)$ is a unit in $B[1/p]$.

To do this we use Payman Kassaei's idea of proving classicity of overconvergent modular forms of integral weight and small slope. More precisely given $x\in {\rm H}^1\bigl((\cX_\infty^{(u)})_\proket, 
\fD_k\bigr)^{(b)}$, we may see it as an element of ${\rm H}^1\bigl((\cX_\infty^{(u)})_\proket, \fD_k^o\  \bigr)$ which is annihilated by $Q(U_p)$, where $Q(T)\in (B\widehat{\otimes}\cO_{\C_p})[T]$ is a polynomial all of whose roots have valuations $\le b$. We write $Q(T)=P(T)-\alpha$, with $P(0)=0$
and denote by $a$ the valuation of $\alpha$.
Then by applying to $x$ the operator $\bigl(P(U_p)\bigr)^{n+u+1}$ we can see it as a class $\tilde{x}$ in
${\rm H}^1\bigl((\cX\backslash \cX_0^{(n+1)})_\proket, \fD_k^o\bigr)$. On the other hand, following Kassaei, we can define a new operator $\bigl(P(U_p)^{n}  \bigr)^{\rm good}$ by choosing all the 
isogenies defining the correspondence $U_p^{n}$ which map $\cX_0^{(n)}$ to $\cX_\infty^{(1)}$.
Let $\mathcal{P}(x):=\bigl(P(U_p)^{n}\bigr)^{\rm good}\bigl(P(U_p)^{u+1}(x)\bigr)\in {\rm H}^1\bigl(\cX_0^{(n)}, \fD_k^o\bigr)$. As the family $\{\cX\backslash \cX_0^{(n+1)}, \cX_0^{(n)}\}$ is an open covering of $\cX$
one can use a Mayer-Vietoris sequence in order to glue  $p^s\tilde{x}, p^s\mathcal{P}(x)$ for a certain fixed power of $p$, $s$ modulo $p^r$, where $r$ was chosen in the beginning large enough so that
$r\ge 2(s+d+1+(u+n+1)a)$, for a certain constant $d$ (see Section \S 5). We obtain a class $z\in {\rm H}^1\bigl(\cX_\proket, \fD_k^0   \bigr)$
annihilated by $Q(U_p)$ and such that its restriction to $\cX_\infty^{(n)}$ is 
congruent to $p^{s+d+1}\alpha^{u+n+1}x$ modulo $p^r$, i.e. there is $x_1\in {\rm H}^1\bigl((\cX_\infty^{(u)})_\proket, \fD_k^o\bigr)$ annihilated by $Q(U_p)$ such that $\mathcal{R}(z)=p^{s+d+1}\alpha^{n+u+1}(x-p^{r/2}x_1)$. Now we iterate the process for $x_1$ and in the end obtain an element $y\in {\rm H}^1\bigl(\cX_\proket, \fD_k\bigr)^{(b)}$ such that $\mathcal{R}(y)=x$.

\

2) {\it On the de Rham comparison.}

\

What we think is interesting to note about the de Rham Eichler-Shimura map is the "decalage" between the filtrations on
$B_{\rm dR}$ that appear. This decalage is explained as follows: on the pro-Kummer \'etale site of
$\cX\bigl(p/\mathrm{Ha}^{p^r}\bigr)$ we have the sheaves with filtrations and connections:
$\nabla'\colon \mathbf{W}_{k, \rm dR}\widehat{\otimes}\cO\mathbb{B}_{\rm dR}\lra  \mathbf{W}_{k+2, \rm dR}\widehat{\otimes}\cO\mathbb{B}_{\rm dR}$ (see Section \S\ref{sec:derhamcomp} for the details), where
$\nabla'=\nabla_k\widehat{\otimes}1 + 1\widehat{\otimes}\nabla_{\rm dR}$ and $\mathbf{W}_{k, \rm dR}$ has an increasing, infinite filtration, while $\cO\mathbb{B}_{\rm dR}$ has a decreasing, infinite filtration. Both $\nabla_k$ and $\nabla_{\rm dR}$ satisfy the Griffith transversality property with respect to the respective filtrations but on the tensor product we don't have a natural filtration. We have however 
the following:
$$
\nabla'\colon \Fil_m\mathbf{W}_{k, \rm dR}\widehat{\otimes}\Fil^0\cO\mathbb{B}_{\rm dR}\lra  \Fil_{m+1}\mathbf{W}_{k+2, \rm dR}\widehat{\otimes}\Fil^{-1}\cO\mathbb{B}_{\rm dR}.
$$
This explains the decalage.

As an immediate consequence of the above theorem we have a "Big Exponential map", i.e. we have a Hecke equivariant, $B$-linear map
$$\mathrm{Exp}^\ast_k\colon \mathrm{H}^1\left(G,\mathrm{H}^1\bigl(\cX_{\overline{K},\proket}, \bD_k^o(T_0^\vee)[n] (1)\bigr)^{(b)}\right)\lra \mathrm{H}^1_{\rm dR}\bigl(\cX\bigl(p/\mathrm{Ha}^{p^r}\bigr), \mathbf{W}_{k, \rm dR,\bullet}\bigr)^{ (b)},$$ which has the property that  for every classical weight specialization  $k_0$ it is compatible with the classical dual exponentail map as follows:

a) If $k_0>b-1$, i.e. $k_0$ is a non-crtical weight for the slope $b$, then we have the following commutative diagram with horizontal isomorphisms. Here we denoted by $\exp^\ast_{k_0}$ the Kato dual exponential map associated to weight $k_0$ modular forms.
$$
\begin{array}{ccccccccc}
\Bigl(\mathrm{H}^1\left(G,\mathrm{H}^1\bigl(\cX_{\overline{K},\proket}, \bD_k^o(T_0^\vee)[n](1) \bigr)^{(b)}\right)\Bigr)_{k_0}&\stackrel{\bigl(\mathrm{Exp}^\ast_k\bigr)_{k_0}}{\lra} &\Bigl(\mathrm{H}^1_{\rm dR}\bigl(\cX\bigl(p/\mathrm{Ha}^{p^r}\bigr), \mathbf{W}_{k, \rm dR,\bullet}\bigr)^{ (b)}\Bigr)_{k_0}\\
\downarrow \cong&&\downarrow\cong\\
\mathrm{H}^1\left(G,\mathrm{H}^1\bigl(\cX_{\overline{K},\proket},  \Symm^{k_0}(T_p(E)^\vee)(1)\bigr)^{(b)}\right)&\stackrel{\exp^\ast_{k_0}}{\lra} &\Fil^0\mathrm{H}^1_{\rm dR}\bigl(\cX, \Symm^{k_0}(\mathrm{H}_E)\bigr)^{(b)}
\end{array}
$$

b) If \ $0\le k_0\le b+1$, i.e. $k_0$ is critical with respect to $b$, we only have a commutaive diagram of the form

 $$
\begin{array}{ccccccccc}
\Bigl(\mathrm{H}^1\left(G,\mathrm{H}^1\bigl(\cX_{\overline{K},\proket}, \bD_k^o(T_0^\vee)[n](1) \bigr)^{(b)}\right)\Bigr)_{k_0}&\stackrel{\bigl(\mathrm{Exp}^\ast_k\bigr)_{k_0}}{\lra} &\Bigl(\mathrm{H}^1_{\rm dR}\bigl(\cX\bigl(p/\mathrm{Ha}^{p^r}\bigr), \mathbf{W}_{k, \rm dR,\bullet}\bigr)^{ (b)}\Bigr)_{k_0}\\
\downarrow &&\uparrow\\
\mathrm{H}^1\left(G,\mathrm{H}^1\bigl(\cX_{\overline{K},\proket},  \Symm^{k_0}(T_p(E)^\vee)(1)\bigr)^{(b)}\right)&\stackrel{\exp^\ast_{k_0}}{\lra} &\Fil^0\mathrm{H}^1_{\rm dR}\bigl(\cX, \Symm^{k_0}(\mathrm{H}_E)\bigr)^{(b)}
\end{array}
$$
where the right vertical arrow is induced by restriction.

\section{Preliminaries.}
\label{sec:setup}

We will denote by $X$, $Y$, $Z$, $\ldots$   log schemes and by caligraphic letters $\cX$, $\cY$, $\cZ$, $\ldots$ log adic spaces. We refer to \cite{Diao} for generalities on those.

\subsection{Pro-Kummer \'etale site.} Given a finite saturated (for short "fs") locally noetherian log scheme $X$ (resp.~a fs locally noetherian log adic space $\cX$)  we denote by $X_\ket$, $X_\fket$ (resp.~$\cX_\ket$, $\cX_\fket$) the Kummer \'etale site, respectively the finite Kummer \'etale site (see \cite[Def.~2.1]{Illusie}, \cite[Def.~4.1.2]{Diao}). Following Scholze \cite{ScholzeHodge}, we denote  by $X_\proket$, resp.~$X_\profket$ (resp.~$\cX_\proket$, $\cX_\profket$) the pro-Kummer \'etale site, resp.~the pro-finite Kummer  \'etale site (see \cite[Def. 5.1.2 \& 5.1.9]{Diao}) of $X$ respectivelly $\cX$.

As a category it is the full subcategory of pro-$X_\ket$, resp.~pro-$X_\fket$ (resp.~pro-$\cX_\ket$, pro-$\cX_\fket$) of pro-objects that are pro-Kummer \'etale over $X$, resp.~pro-finite Kummer \'etale over $X$ (resp.~$\cX$), i.e., objects that are equivalent to cofiltered systems $\displaystyle{\lim_\leftarrow}  Z_i$ such that $Z_i\to X$ is  Kummer \'etale, resp.~finite Kummer \'etale,  for every $i$ and there exists an index $i_0$ such that $Z_j\to Z_i$ is finite Kummer \'etale and surjective for $i$ and $j\geq i_0$ (and similarly for $\cX$). For the covering families we refer to loc.~cit. 

We have a natural projection $\nu\colon X_\proket\to X_\ket$ (resp.~$\nu\colon \cX_\proket\to \cX_\ket$) sending $U\in X_\ket$ (resp.~in $\cX_\ket$) to the constant inverse system defined by $U$. Then, by \cite[Prop.~5.1.6 \& 5.1.7]{Diao} for every sheaf of abelian groups $\cF$ on $X_\ket$ (resp.~in $\cX_\ket$) and any quasi-compact and quasi-separated object $U={\displaystyle{\lim_\leftarrow}} U_j$ in $X_\proket$ (resp.~in $\cX_\proket$) we have natural isomorphisms of $\delta$-functors  $$\mathrm{H}^i\bigl(U_\proket,\nu^{-1}(\cF)\bigr)\cong \lim_\to \mathrm{H}^i\bigl(U_{j,\ket},\cF\bigr), \qquad \cF\to {\rm R}\nu_\ast\nu^{-1}\bigl(\cF\bigr).$$

\subsection{Sheaves on the pro-Kummer \'etale site.} \label{sec:periodsheaves}

We then have the following sheaves on $\cX_{\proket}$ defined in \cite[Def.~5.4.1]{Diao} and in \cite[Def.~2.2.3]{DLLZ} following \cite[Def.~6.1]{ScholzeHodge}:

\begin{itemize} 

\item[i.] The structure sheaf $\cO_{\cX_{\proket}}:=\nu^{-1}\bigl(\cO_{\cX_\ket}\bigr)$ and its subsheaf of integral elements  $\cO_{\cX_{\proket}}^+:=\nu^{-1}\bigl(\cO_{\cX_\ket}^+\bigr)$. It comes endowed with a moprhism of sheaves of multiplicative monoids $\alpha\colon \cM \to  \cO_{\cX_{\proket}}$ defined by taking $\nu^{-1}$ of the morphism of sheaves of multiplicative monoids  defining the log structure on $\cX$. 

\item[ii.] The completed sheaf $\widehat{\cO}_{\cX_{\proket}}^+:=\displaystyle{\lim_{\infty \leftarrow n}} \cO_{\cX_{\proket}}^+/p^n \cO_{\cX_{\proket}}^+$ and the completed structure sheaf $ \widehat{\cO}_{\cX_{\proket}}:=\widehat{\cO}_{\cX_{\proket}}^+\bigl[ \frac{1}{p}\bigr]$.

\item[iii.] The tilted integral structure sheaf  $\widehat{\cO}_{\cX^\flat_{\proket}}^+:= \displaystyle{\lim_{\leftarrow \varphi}} \cO_{\cX_{\proket}}^+/p \cO_{\cX_{\proket}}^+$ and the tilted  structure sheaf $\widehat{\cO}_{\cX_{\proket}^\flat}:=\widehat{\cO}_{\cX_{\proket}^\flat}^+\otimes_{K^{\flat+}} K^\flat$.  It comes endowed with a morphism of monoids $\alpha^\flat \colon \cM^\flat \to \widehat{\cO}_{\cX_{\proket}^\flat}$ where $\cM^\flat$ is the inverse limit $\displaystyle{\lim_{\leftarrow}} \cM$ indexed by $\N$ with transition maps given by rasing to the  $p$-th power, $\widehat{\cO}_{\cX^\flat}$ is identified  as a sheaf of mutiplicative monoids with  the inverse limit $\displaystyle{\lim_\leftarrow} \widehat{ \cO}_{\cX_{\proket}}$ indexed by $\N$ with transition maps given by rasing to the  $p$-th power
 and the map $\alpha^\flat$ is the inverse limit of the maps $\alpha$ composed with the natural maps ${ \cO}_{\cX_{\proket}}\to \widehat{ \cO}_{\cX_{\proket}}$.

\item[iv.] The period sheaf $\mathbb{A}_{\rm inf}:=\mathrm{W}\bigl(\widehat{\cO}_{\cX^\flat_{\proket}}^+ \bigr)$ and the period map $\vartheta\colon \mathbb{A}_{\rm inf,\cX}\to \widehat{\cO}_{\cX_{\proket}}^+$. 
\end{itemize}

\

\

\subsection{Log affinoid pefectoid opens.}\label{sec:logaffinoid} Consider a  locally noetherian fs log adic space $\cX$ over $\mathrm{Spa}(\Q_p,\Z_p)$. Following \cite[Def.~5.3.1 \& Rmk.~5.3.2]{Diao}, an object $U=\lim_{i\in I} U_i$, with $U_i=\bigl(\mathrm{Spa}(R_i,R_i^+),\cM_i\bigr)$ in~$\cX_\proket$ is called \emph{log affinoid perfectoid} if:
 
\begin{itemize}

\item[a.] there is an initial object $0 \in I$.
\item[b.]  each $U_i$ admits a global sharp finite saturated chart $P_i$ such that each transition map
$U_j\to U_i$ is modeled on the Kummer chart $P_i \to P_j$;
\item[c.]  $\bigl(\mathrm{Spa}(R_i,R_i^+)\bigr)_i$ is affinoid perfectoid, i.e.,  the $p$-adic completion $(R,R^+)$ of $\lim_i (R_i,R_i^+)$ is a perfectoid affinoid $\Q_p$-algebra;
\item[d.] the monoid $P = \lim_i P_i$ is $n$-divisible for all $n$.

\end{itemize}

Given a log affinoid pefectoid $U$ as above we denote by $\widehat{U}:=\mathrm{Spa}\bigl(R,R^+\bigr)$ the associated perfectoid affinoid space. By \cite[Lemma 5.3.6]{Diao} it has the same underlying topological space as $U$ (which is defined as the inverse limit of topological spaces $\displaystyle \lim_{\leftarrow\  i} \vert U_i\vert$). Moreover by \cite[Thm.~5.4.3]{Diao} and \cite[Thm.~6.5]{ScholzeHodge} we have that

$$\widehat{\cO}_{\cX_\proket}^+(U)=R^+,\quad \widehat{\cO}_{\cX_\proket}(U)=R, \quad \widehat{\cO}_{\cX^\flat_\proket}^+(U)=R^{\flat+}, \quad \widehat{\cO}_{\cX^\flat_\proket}(U)=R^{\flat},    \quad \mathbb{A}_{\rm inf}(U)=W\bigl(R^{\flat+}\bigr)$$and the cohomology groups $$\mathrm{H}^i\bigl(U,\widehat{\cO}_{\cX_\proket}^+\bigr) \sim 0 ,\quad \mathrm{H}^i\bigl(U,\widehat{\cO}_{\cX^\flat_\proket}^+\bigr) \sim 0 , \quad \mathrm{H}^i\bigl(U, \mathbb{A}_{\rm inf}\bigr) \sim 0 \quad \forall i\geq 1$$ (where $\sim$ means almost $0$).

Thanks to \cite[Prop.~5.3.12 \& Prop.~5.3.13]{Diao} there exists a basis $\mathcal{B}$ for the site $\cX_\proket$ given by log affinoid perfectoid objects such that for every locally constant 
$p$-torsion sheaf $\mathbb{L}$ on $\cX_\ket$ and every $U\in \mathcal{B}$ we have $\mathrm{H}^i\bigl( \cX_{\proket}\vert_U,\mathbb{L}\bigr)=0$ for $i\geq 1$. In case $X$ is a fs log scheme over $\Q_p$ there is an analogous notion of log affinoid perfectoid opens of $X_\proket$ and  it follows from the arguments in loc.~cit.~that there exists a basis of $X_\proket$  with the same property. 

\

Let $K$ be a perfectoid  field of characteristic $0$ with an open and bounded subring $K^+$. Assume that $\cX$ is defined over $\mathrm{Spa}(K,K^+)$.  In this case $\mathbb{A}_{\rm inf}$, resp.~$\widehat{\cO}_\cX^+$ is a sheaf of algebras over the classical period ring ${\rm A}_{\rm inf}:=\mathrm{W}\bigl(K^{\flat+}\bigr)$, resp.~over $K^+$ and given a generator $\zeta\in {\rm A}_{\rm inf}$ for the kernel of the canonical ring homomorphism ${\rm A}_{\rm inf}\to K^+$ it follows from \cite[Lemma 6.3]{ScholzeHodge} that we have an exact sequence $$0 \longrightarrow \mathbb{A}_{\rm inf} \stackrel{\cdot \zeta}{\longrightarrow} \mathbb{A}_{\rm inf} \stackrel{\vartheta}{\longrightarrow}  \widehat{\cO}_{\cX_\proket}^+ \longrightarrow 0.$$

\subsection{Comparison results.}\label{sec:comparecoeff} Assume that $\cX$ is a finite saturated locally noetherian log adic space over an affinoid field $\mathrm{Spa}(K,K^+)$ with $K$ algebraically closed of characteristic $0$.  The main result of \cite{Diao}, namely Theorem 6.2.1, states that if the underlying adic space to $\cX$ is log smooth and proper and $\mathbb{L}$ is an $\mathbb{F}_p$-local system on $\cX_\ket$, then the cohomology groups $\mathrm{H}^i(\cX_\ket,\mathbb{L}\bigr)$ are finite for all $i$, they vanish for $i\gg 0$ and the natural map $$\mathrm{H}^i(\cX_\ket,\mathbb{L}\bigr) \otimes K^+/p \to \mathrm{H}^i(\cX_\ket,\mathbb{L}\otimes \cO_{\cX_\ket}^+/p\bigr) $$is an almost isomorphism for every $i\geq 0$. As $\cF\cong {\rm R}\nu_\ast\nu^{-1}\bigl(\cF\bigr)$ for any sheaf of abelian goups, we obtain the same cohomology groups replacing $\cX_\ket$ with $\cX_\proket$ in the isomorphisms above.
Here we denoted $\nu\colon \cX_\proket\lra \cX_\ket$ the natural morphism of sites.

\

A second issue is a GAGA type comparison in the case $X$ is finite separated locally noetherian log scheme, proper and log smooth over $K$. Let $\cX$ be the associated log adic space over $\mathrm{Spa}(K,K^+)$. We have a natural morphism of sites $\gamma\colon \cX_\ket \to X_\ket$. Let $\mathbb{L}$ be an $\mathbb{F}_p$-local system on $X_\ket$. Then

\begin{proposition}\label{prop:GAGA} For every $i\geq 0$ the natural morphism $\mathrm{H}^i(X_\ket,\mathbb{L}\bigr) \longrightarrow \mathrm{H}^i(\cX_\ket,\gamma^\ast\bigl(\mathbb{L}\bigr)\bigr)$ is an isomorphism.

\end{proposition}
\begin{proof} Let $X^o$ and $\cX^o$ be the scheme, resp.~the adic space defined by $X$ and $\cX$ forgetting the log structures. In this case the morphism of sites $\gamma^o\colon \cX_{\rm et}^o \to X_{\rm et}^o$ induces the map $\mathrm{H}^i(X_{\rm et}^o,F\bigr) \longrightarrow \mathrm{H}^i(\cX_{\rm et}^o,\gamma^{o,\ast}\bigl(F\bigr)\bigr)$  for every sheaf of torsion abelian groups $F$. It is an isomorphism due to \cite[Thm. 3.2.10]{Huber}.
Consider the commutative diagram of sites $$\begin{array}{ccc}
\cX_{\rm et} & \stackrel{\gamma}{\longrightarrow} &  X_{\rm et} \cr 
\alpha \big\downarrow & & \big\downarrow \beta \cr 
\cX_{\rm et}^o & \stackrel {\gamma^o}{\longrightarrow} &  X_{\rm et}^o \cr 
\end{array}$$
 
Using the compatibility of the Leray spectral sequences $\mathrm{H}^i(\cX_{\rm et}^o,\mathrm{R}^j \alpha_\ast \gamma^\ast(\mathbb{L})\bigr) \Longrightarrow   \mathrm{H}^{i+j}(\cX_\ket,\gamma^\ast\bigl(\mathbb{L}\bigr)\bigr)$ and $\mathrm{H}^i(X_{\rm et}^o,\mathrm{R}^j \beta_\ast (\mathbb{L})\bigr) \Longrightarrow   \mathrm{H}^{i+j}(X_\ket,\mathbb{L}\bigr)$ and the result of Huber,  it suffices to prove that the natural morphism $$\gamma^{o,\ast}\bigl(\mathrm{R}^j \beta_\ast (\mathbb{L}) \bigr) \longrightarrow \mathrm{R}^j \alpha_\ast\bigl( \gamma^\ast(\mathbb{L}\bigr) $$is an isomorphism of sheaves for every $j$.   It suffice to prove that we get an isomorphism after passing to stalks at  geometric points $\zeta=\mathrm{Spa}\bigl(l,l^+\bigr) \to \cX^o$, as those form a conservative family by \cite[Prop.~2.5.5]{Huber}. Recall that $\zeta$ might consist of more than one point but it has a unique closed point $\zeta_0$. Taking the stalk at $\zeta$ is equivalent to take global sections over the associated strictly local adic space $\cX^o(\zeta)$  (see \cite[Lemma 2.5.12]{Huber}. Let $X^o(\zeta_0)$ be the spectrum of the strict henselization of $X$ at $\zeta_0$. Taking the stalk at $\zeta$ of $\gamma^{o,\ast}$ of a sheaf is equivalent to take the sections of that sheaf over 
 $X^o(\zeta_0)$. We have a natural map of sites  $\cX^o(\zeta)_\ket\to X^o(\zeta_0)_\ket$, considering on $X^o(\zeta)$ and on $ \cX^o(\zeta)$ the log structures coming from  $X$ and $\cX$. We need to show that it is an equivalence. In both cases the Kummer \'etale sites are the same as the finite Kummer \'etale sites; indeed by definition the Kummer \'etale topology is generated in both cases by finite Kummer \'etale covers and classical \'etale morphisms and  a Kummer cover of  $X^o(\zeta)$, resp.~$ \cX^o(\zeta)$ is still strictly local and hence does not admit any non-trivial classical \'etale cover (see \cite[Lemma 2.5.6]{Huber} in the adic setting). Both in the schematic and in the adic setting the finite Kummer \'etale sites are  equivalent to the category of finite sets with continuous action of the group $\mathrm{Hom}(\overline{M}^{\rm gp},\widehat{\Z}\bigr)$ with $\overline{M}$ the stalk of the log structure at $\zeta$, modulo $l^\ast$. See \cite[Ex.~4.7(a)]{Illusie} in the schematic case and \cite[Prop.~4.4.7]{Diao} in the adic case. As such quotient is the same in the schematic and adic cases, the conclusion follows.

\end{proof}

\section{VBMS and dual VBMS.}

\subsection{VBMS i.e. vector bundles with marked sections.}
\label{sec:VBMS}

We recall the main constructions of \cite{andreatta_iovita_triple} and \cite{ICM_AIP}.
Let $\cX$ denote an adic analytic space over $\mathrm{Spa}(\Q_p, \Z_p)$ and let $(\cE, \cE^+)$ denote a pair consisting of  a locally free $\cO_{\cX}$-module $\cE$ of rank $2$ and a sub-sheaf $\cE^+$ of $\cE$ 
which is a locally free  $\cO_{\cX}^+$-module of rank $2$ such that $\cE=\cE^+\otimes_{\cO_{\cX}^+}\cO_{\cX}$. Let $\cI\subset \cO_{\cX}^+$ be an invertible ideal such that $\cI$ gives the topology on $\cO_{\cX}^+$ and let $r\ge 0$ be an integer such that $  \cI \subset p^r \cO_{\cX}^+$.

We suppose that there is a section $s\in \mathrm{H}^0(\cX, \cE^+/\cI\cE^+)$ such that the submodule $\bigl(\cO_{\cX}^+/\cI\bigr) s$ is a direct summand of $\cE^+/\cI\cE^+$.
We have the following

\begin{theorem}[\cite{ICM_AIP}] 
\label{thm:VBMS}
a) The functor attaching to every adic space $\gamma\colon \cZ\rightarrow \cX$, such that
$t^\ast(\cI)$ is an invertible ideal in $\cO_{\cZ}^+$, the set (group in fact):
$$
\bV(\cE, \cE^+)\bigl(\gamma\colon \cZ\rightarrow \cX  \bigr):=\mathrm{Hom}_{\cO_{\cZ}^+}\bigl(\gamma^\ast(\cE^+), \cO_{\cZ}^+  \bigr)=\mathrm{H}^0\bigl(\cZ, \gamma^\ast(\cE^+)^\vee   \bigr),
$$
is represented by the adic vector bundle $\bV(\cE,\cE^+):=\mathrm{Spa}_{\cX}\bigl(\mathrm{Sym}(\cE), \mathrm{Sym}(\cE^+)  \bigr)\rightarrow\cX$.

b) The subfunctor of  $\bV(\cE, \cE^+)$ denoted $\bV_0(\cE^+, s)$ which associates to every adic space
$\gamma\colon \cZ\rightarrow \cX$ as above, the set:
$$
\bV_0(\cE^+,s)\bigl(\gamma\colon \cZ\rightarrow \cX\bigr):=\Bigl\{h\in\bV(\cE,\cE^+)\bigl(\gamma\colon \cZ\rightarrow \cX\bigr)
\ \vert \ h\bigl(\mathrm{mod}\ \gamma^\ast(\cI) \bigr)(\gamma^\ast(s))=1\Bigr\},
$$
is represented by the open adic subspace of $\bV(\cE,\cE^+)$, also denoted $\bV_0(\cE^+,s)$,  consisting of the points $x$ such that $\vert\tilde{s}-1\vert_x\le \vert\alpha\vert_x$, where $\tilde{s}$ is a (local) lift of $s$ to $\cE^+$ and $\alpha$ is a (local) generator of $\cI$ at $x$. 

c) Suppose that we have sections $s$ and $t\in \mathrm{H}^0(\cX, \cE^+/\cI\cE^+)$ which form an $\bigl(\cO_{\cX}^+/\cI\bigr)$-basis of $\cE^+/\cI\cE^+$. Then, the subfunctor $\bV_0(\cE^+,s,t)$ of $\bV_0(\cE^+, s)$ which associates to every adic space
$\gamma\colon \cZ\rightarrow \cX$ the set:
$$
\bV_0(\cE^+,s,t)\bigl(\gamma\colon \cZ\rightarrow \cX\bigr):=\Bigl\{h\in\bV_0(\cE^+,s)\bigl(\gamma\colon \cZ\rightarrow \cX\bigr)
\ \vert \ h\bigl(\mathrm{mod}\ \gamma^\ast(\cI) \bigr)(\gamma^\ast(t))=0\Bigr\},
$$
is represented by the open adic subspace $\bV_0(\cE^+,s)$ of $\bV_0(\cE^+)$ consisting of the points $x$ such that $\vert\tilde{t}\vert_x\le \vert\alpha\vert_x$ for a (any) lift $\tilde{t}$ of $t$ to $\cE^+$ and $\alpha$  a (local) generator of $\cI$ at $x$.

\end{theorem}
\begin{proof}
The proof is local on $\cX$. Assume that $U\subset \cX$ is an affinoid open $U=\mathrm{Spa}(R,R^+)$ such that $\cI\vert_U$ is principal generated by $\alpha\in R^+$ and $\cE^+\vert_U$ is free with basis $f_0$, $f_1$ with $f_0({\rm mod}\ \alpha)=s\vert_U$. Then  $f_1({\rm mod}\ \alpha) $ generates $\Bigl(\bigl(\cE^+/\cI\cE^+\bigr)/s\bigl(\cO_{\cX}/\cI  \bigr)\Bigr)\vert_U$ and we assume in case (c) that $f_1({\rm mod}\ \alpha)=t\vert_U$. Then by \cite[\S 2]{andreatta_iovita_triple} we have
$
\bV(\cE,\cE^+)\vert_U={\rm Spa}\bigl(R\langle X,Y\rangle , R^+\langle X,Y\rangle  \bigr)$ and 
$$
\bV_0(\cE^+,s)\vert_U={\rm Spa}\bigl(R\langle X,Y\rangle\langle\frac{X-1}{\alpha}\rangle , R^+\langle X,Y\rangle\langle \frac{X-1}{\alpha}\rangle  \bigr)={\rm Spa}\bigl(R\langle Z,Y\rangle, R^+\langle Z,Y\rangle  \bigr)
$$
where $\displaystyle X=1+\alpha Z$ giving also the map to $\bV(\cE,\cE^+)\vert_U$ . Similarly $$
\bV_0(\cE^+,s,t)\vert_U={\rm Spa}\bigl(R\langle Z,W\rangle, R^+\langle Z,W\rangle  \bigr)
$$
with $\displaystyle Y=\alpha W$. We have the tautological map over $\bV(\cE,\cE^+)\vert_U$ given by $$ \cE^+\otimes_{R^+} R^+\langle X,Y\rangle \to R^+\langle X,Y\rangle,\qquad  f_0\mapsto X, f_1\mapsto Y$$from which we deduce similarly the tautological maps over $\bV_0(\cE^+,s)\vert_U$ and $\bV_0(\cE^+,s,t)\vert_U$ providing the claimed representability and concluding the proof.

\end{proof}

\subsection{Dual VBMS.}\label{sec:dualVBMS}

In this article we'll need a variant of the construction in Section \S \ref{sec:VBMS} which we now present. 
Suppose that $\cX$, $\cI$, $(\cE, \cE^+)$ are as in Section \S \ref{sec:VBMS}. Moreover we assume that there is an exact sequence of locally free $\cO_{\cX}^+/\cI$-modules
$$
0\lra \cQ\lra \cE^+/\cI\cE^+\lra \cF\lra 0,
$$
and a section $s\in \mathrm{H}^0\bigl(\cX,  \cF \bigr)$ such that $\bigl(\cO_{\cX}^+/\cI  \bigr)s=\cF$.
We have

\begin{theorem}
\label{thm:VBMSD}
The subfunctor $\bV_0^D(\cE^+, \cQ, s)$ of $\bV(\cE,\cE^+)$ defined by associating to every 
adic space $t\colon \cZ\rightarrow \cX$ as in Section \S \ref{sec:VBMS} the set
$$
\bV_0(\cE^+,\cQ, s)\bigl(\gamma\colon \cZ\rightarrow \cX\bigr):=
$$
$$:=\Bigl\{h\in\bV(\cE,\cE^+)\bigl(\gamma\colon \cZ\rightarrow \cX\bigr)
\ |\ h\bigl(\mathrm{mod}\ \gamma^\ast(\cI) \bigr)\bigl(\cQ\bigr)=0\ \mathrm{ and }\ h\bigl(\mathrm{mod}\ \gamma^\ast(\cI) \bigr)(t^\ast(s))=1\Bigr\},
$$
is represented by the the open adic subspace of $\bV(\cE, \cE^+)$ denoted $\bV_0^D(\cE^+, \cQ,s)$ and consisting of the points $x$ such that $\vert q\vert_x\le \vert\alpha\vert_x$ and $\vert\tilde{s}-1\vert_x\le \vert\alpha\vert_x$, where
$q$ is a (local) lift to $\cE^+$ of a local generator of $\cQ$ at $x$, $\alpha$ is a (local) generator of $\cI$ at $x$ and $\tilde{s}$ is a (local) lift of $s$ to $\cE^+$.

\end{theorem}

\subsection{The sheaves $\WW_k$ and $\WW^D_k$.}
\label{sec:WkWDk}

Let the hypothesis be as in Section \S \ref{sec:VBMS} and \S \ref{sec:dualVBMS}. We assume that we have a morphism of
adic spaces $\cX\to \cW$ where let us recall that $\cW$ is the adic weight space for $\GL_{2,/\Q}$. For every adic space $\cZ$ the morphisms $\cZ\to \cW$ classify continuous homomorphisms $\Z_p^\ast \to \Gamma\bigl(\cZ,\cO_{\cZ}\bigr)$. 
We denote by $k^{\rm univ}\colon \Z_p^\ast \to  \Gamma\bigl(\cX,\cO_{\cX}\bigr)$ the continuous homomorphism defined by $\cX\to \cW$. We assume that  $k^{\rm univ}$ satisfies the following   analyticity assumption:    there exists a section $u_{\rm univ}$ of $\cO_{\cX}$ such that  
$\vert  u_{\rm univ} \vert_x <  \vert p^{\frac{1}{p-1} -r}  \vert_x$ for every $x\in \cX$   and  $$k^{\rm univ}(t)=\exp u_{\rm univ} \log(t), \quad \forall t\in 1+ p^r \Z_p.$$
We recall that the integer $r\ge 0$ is such that $\cI\subset p^r \cO_{\cX}^+$.  Let us denote by $\cT$ the adic torus representing the functor which associates to an adic space $\gamma\colon \cZ\rightarrow \cX$
the group
$$
\cT(\gamma\colon \cZ\rightarrow \cX):=1+\gamma^\ast(\cI).
$$
Then $k^{\rm univ}$ defines a character $k^{\rm univ}\colon \cT \to \mathbb{G}_m$, i.e., a morphism of adic spaces and group functors, using the fomula above. 

We have natural actions of $\cT$ on $\bV_0(\cE^+, s)$, $\bV_0(\cE^+, s,t)$  and $\cV_0^D\bigl(\cE^+,\cQ, s \bigr)$ defined on $\gamma\colon \cZ\rightarrow \cX$ points by: $u\ast h:=uh$ and $u\ast h':=uh'$
where $u\in \cT(\gamma\colon \cZ\rightarrow \cX)$, $h\in \bV_0(\cE^+, s)(\gamma\colon \cZ\rightarrow \cX)$, resp. $h\in \bV_0(\cE^+, s,t)(\gamma\colon \cZ\rightarrow \cX)$, resp.  $h'\in\bV_0^D\bigl( \cE^+\vee,\cQ, s\bigr)(\gamma\colon\cZ\rightarrow \cX)$.   Let us denote by $f\colon \bV_0(\cE^+,s)\lra \cX$, $g\colon \bV_0(\cE^+,s,t)\lra \cX$  and by $f^D:\bV_0^D\bigl(\cE^+,\cQ, s\bigr)\lra \cX$ the structural morphisms.

\begin{definition}
We denote:
$$
\WW_{k^{\rm univ}}(\cE^+, s):=f_\ast\bigl(\cO_{\bV_0(\cE^+,s)^+ } \bigr)[k^{\rm univ}],\ \ \WW_{k^{\rm univ}}(\cE^+, s,t):=g_\ast\bigl(\cO_{\bV_0(\cE^+,s) }^+\bigr)[k^{\rm univ}]$$
and
$$
\WW^D_{k^{\rm univ}}\bigl(\cE^+,\cQ, s  \bigr):=f^D_\ast\Bigl(\cO_{\bV_0^D\bigl(\cE^+,\cQ, s  \bigr)}^+  \Bigr)[k^{\rm univ}],
$$
where if $\cG$ is an $\cO_{\cX}^+$-module on $\cX$ with an action by the torus $\cT:=1+\cI$, we denote $\cG[k^{\rm univ}]$ the subsheaf of
$\cG$ of sections $x$ such that $u\ast x=k^{\rm univ}(u)x$, for all corresponding sections $u$ of $\cT$. 

\end{definition}

\subsubsection{Local descriptions of the sheaves $\WW_k$ and $\WW_k^D$.}\label{sec:explicit}

We assume all notations and assumptions of the proof of Theorem \ref{thm:VBMS} and of Section \S \ref{sec:WkWDk}. Let $U={\rm Spa}(R,R^+)\subset 
\cX$ be an affinoid open such that $\cE^+\vert_U=f_0\cO_U^++f_1\cO_U^+$, $\cI\vert_U=\alpha\cO_U^+$
and such that $f_0({\rm mod}\ \alpha)=s\vert_U$ and $f_1({\rm mod}\ \alpha) $ generates $\Bigl(\bigl(\cE^+/\cI\cE^+\bigr)/s\bigl(\cO_{\cX}/\cI  \bigr)\Bigr)\vert_U$ (resp. $f_1({\rm mod}\ \alpha)=t\vert_U$.

In view of the next section we also consider the dual situation, i.e. recall that $s$ is a global marked section of $\cE^+$ modulo $\cI$ and denote $\cF:=(\cE^+/\cI\cE^+)/(s\cO_\cX)$. By dualizing we obtain the exact sequence
$$
0\lra \cQ\lra  (\cE^+)^\vee/\cI(\cE^+)^\vee \lra (\cO_\cX/\cI)s^\vee\lra 0,
$$
where $\cQ:=\cF^\vee$,
which defines data for $\bV_0^D$.
Then $(\cE^+)^\vee\vert_U=f_0^\vee\cO_{U}^++
f_1^\vee\cO_{U}^+$, $\cQ\vert_U=f_1^\vee\bigl(\cO_U^+/\alpha\cO_U^+\bigr)$ and 
$f_0^\vee({\rm mod}\ \alpha)=s^\vee\vert_U$. 
Therefore
$$
\bV_0(\cE^+,s)\vert_U={\rm Spa}\bigl(R\langle X,Y\rangle\langle\frac{X-1}{\alpha}\rangle , R^+\langle X,Y\rangle\langle \frac{X-1}{\alpha}\rangle  \bigr)={\rm Spa}\bigl(R\langle Z,Y\rangle, R^+\langle Z,Y\rangle  \bigr)
$$
where $\displaystyle X=1+\alpha Z$ and 
$$
\bV_0(\cE^+,s,t)\vert_U={\rm Spa}\bigl(R\langle Z,W\rangle, R^+\langle Z,W\rangle\bigr)
$$with $Y=\alpha W$. Similarly we have
$$
\bV_0^D\bigl(\cE^\vee,\cQ, s^\vee\bigr)|_U={\rm Spa}\Bigl(R\langle A,B\rangle\langle \frac{A-1}{\alpha}, \frac{B}{\alpha} \rangle,  R^+\langle A,B\rangle\langle \frac{A-1}{\alpha}, \frac{B}{\alpha} \rangle \Bigr)=
$$
$$
={\rm Spa}\bigl(R\langle C, D\rangle, R^+\langle C, D\rangle  \bigr),
$$
where $A=1+\alpha C$ and $B=\alpha D$.

Therefore we have
$$
\WW_{k^{\rm univ}}(\cE^+, s)(U)=R^+\langle Z,Y\rangle[k^{\rm univ}]=R^+\langle \frac{Y}{1+\alpha Z}\rangle k^{\rm univ}(1+\alpha Z),
$$

$$
\WW_{k^{\rm univ}}(\cE^+, s,t)(U)=R^+\langle Z,W\rangle[k^{\rm univ}]=R^+\langle \frac{W}{1+\alpha Z}\rangle k^{\rm univ}(1+\alpha Z)
$$

and
$$
\WW^D_{k^{\rm univ}}\bigl((\cE^+)^\vee,\cQ, s^\vee  \bigr)(U)=R^+\langle C,D\rangle[k^{\rm univ}]=R^+\langle \frac{D}{1+\alpha C}\rangle
k^{\rm univ}(1+\alpha C).
$$

\subsection{The duality.}
\label{sec:duality}

Suppose that $\cX$, $\cI$, $ (\cE, \cE^+)$, $s$ are as in  Section \S \ref{sec:VBMS}  and let us denote by
$$\iota \colon \cF:=s\bigl(\cO_{\cX}^+/\cI\bigr)\hookrightarrow  \cE^+/\cI\cE^+.$$Let $\bigl(\cE', (\cE^+)^\vee\bigr)$ 
denote the pair where $(\cE^+)^\vee$ is the $\cO_{\cX}^+$-dual of $\cE^+$
and $\cE':=(\cE^+)^\vee\otimes_{\cO_{\cX}^+}\cO_{\cX}$. Moreover let us consider 
$\cQ:=\mathrm{Ker}\bigl(\cE^+/\cI\cE^+\stackrel{\iota^\vee}{\lra}\cF^\vee  \bigr)$, where $\cF^\vee$ is the $\cO_{\cX}^+/\cI$-dual of $\cF$. 
As explained in Theorems \ref{thm:VBMS} and \ref{thm:VBMSD} we have 
adic spaces $\bV_0(\cE^+,s)$ and $\bV_0^D\bigl((\cE^+)^\vee, \cQ, s^\vee\bigr)$ over $\cX$.  Consider the morphism of adic spaces
$$
\langle \ , \ \rangle\colon \bV_0(\cE^+, s)\times_{\cX}\bV_0^D\bigl((\cE^+)^\vee, \cQ, s^\vee\bigr)\lra 
\bV_0(\cO_{\cX}^+, 1), 
$$defined on points as follows. Fix a morphism of adic spaces $\gamma\colon \cZ\rightarrow \cX$. Consider  $h\in \bV_0(\cE^+,s)\bigl(\gamma\colon \cZ\rightarrow \cX\bigr)$ and $h'\in \bV_0^D\bigl((\cE^+)^\vee, \cQ, s^\vee\bigr)\bigl(\gamma\colon \cZ\rightarrow \cX\bigr)$. By definition this is equivalent to give morphisms of $\cO_{\cZ}^+$-modules  $h\colon\gamma^\ast(\cE^+)\rightarrow \cO_{\cZ}^+$ with $h(\mathrm{mod}\ \cI)(\gamma^\ast(s))=1$  and  
$h'\colon\gamma^\ast(\cE^+)^\vee\rightarrow \cO_{\cZ}^+$ with $h'(\mathrm{mod}\ \gamma^\ast(\cI))(\cQ)=0$ 
and $h'(\mathrm{mod}\ \gamma^\ast(\cI))(\gamma^\ast(s^\vee))=1$. Then $h'(h)\in \mathrm{H}^0(\cZ, \cO_{\cZ}^+)$ and $h'(h)(\mathrm{mod}\ \gamma^\ast(\cI))=1$, i.e., $h'(h)\in \bV_0(\cO_{\cX}^+, 1)\bigl(\gamma\colon\cZ\rightarrow \cX  \bigr)$.
We  define
$$\langle h, h'\rangle:=h'(h)\in  \bV_0(\cO_{\cX}^+, 1)\bigl(\gamma\colon\cZ\rightarrow \cX  \bigr).$$Notice that $\bV(\cO_{\cX},\cO_{\cX}^+, 1)$ is the affine one dimensional space  $\mathbb{A}^1_\cX$ over $\cX$, with standard coordinate $T$.  The  torus $\cT \times \cT$  acts componentwisly on  $\bV_0(\cE^+, s)\times_{\cX}\bV_0^D\bigl((\cE^+)^\vee, \cQ, s^\vee\bigr)$. Given sections $(h,h')$ of the latter and $(u,u')$ of $\cT\times \cT$ we have $\langle u \ast h, u' \ast h'\rangle =(u \cdot u') \ast \langle h, h'\rangle$.

\begin{lemma} There is a section $T^{k^{\rm univ}}$ of $\WW_{k^{\rm univ}}(\cO_{\cX}^+, 1)$ over $\cX$ such that  $\WW_{k^{\rm univ}}(\cO_{\cX}^+, 1)= T^{k^{\rm univ}} \cdot \cO_{\cX}^+$. Moroever,  $\langle \ , \ \rangle^\ast\bigl(T^{k^{\rm univ}}\bigr)\in  \WW_{k^{\rm univ}}(\cE^+, s) \widehat{\otimes} \WW^D_{k^{\rm univ}}\bigl((\cE^+)^\vee,\cQ, s^\vee  \bigr)$.

\end{lemma} 
\begin{proof} The statament is local on $\cX$. We use the explicit coordinates of \S \ref{sec:explicit}. Then, $T=1+\alpha V$ and $T^{k^{\rm univ}}$ is the section $ k^{\rm univ}(1+\alpha V)$. By loc.~cit., we have $\WW_{k^{\rm univ}}(\cO_{\cX}^+, 1)(U)= T^{k^{\rm univ}} \cdot R^+$.

As    $\langle \ , \ \rangle^\ast\bigl( T \bigr)=X\otimes A + Y \otimes B=\bigl(1+\alpha Z )\big) \cdot \bigl(1+\alpha C \bigr)+ \alpha Y \otimes D$, we conclude that
\begin{equation}\label{eq:dualk} \langle \ , \ \rangle^\ast\bigl( T^{k^{\rm univ}} \bigr)=k^{\rm univ}(1+\alpha Z)  k^{\rm univ}(1+\alpha C)   \cdot  k^{\rm univ}\left(1+\alpha \frac{Y}{1+\alpha Z} \otimes \frac{D}{1+\alpha C}\right).\end{equation}This concludes the proof.
\end{proof}

\begin{definition}\label{def:dualityWkWk^D} Write $\WW_{k^{\rm univ}}(\cE^+, s)^\vee$ for the $\cO_{\cX}^+$-dual of $\WW_{k^{\rm univ}}(\cE^+, s)$.  Define the map of $\cO_{\cX}^+$-modules $$\xi_{k^{\rm univ}}\colon \WW_{k^{\rm univ}}(\cE^+, s)^\vee \longrightarrow \WW^D_{k^{\rm univ}}\bigl((\cE^+)^\vee,\cQ, s^\vee  \bigr),\quad \gamma \mapsto (\gamma \otimes 1) \bigl(\langle \ , \ \rangle^\ast\bigl(T^{k^{\rm univ}}\bigr)\bigr).$$

\end{definition}

\subsubsection{Local descriptions of the duality between $\WW_k$ and $\WW_k^D$.}\label{sec:explicitduality}  

We put ourselves in the setting of \S \ref{sec:explicit} and compute explicitly the pairing $\xi_{k^{\rm univ}}$ on the affinoid $U$ in terms of the local coordintaes of loc.~cit..  As $k^{\rm univ}$ is supposed to be $r$-analytic on $\cX$, we can write $k^{\rm univ}(t)=\exp u_{\rm univ} \log (t)$ for every $t\in 1+p^r \Z_p$. We then claim that $$\xi_{k^{\rm univ}}\left( k^{\rm univ}(1+\alpha Z)  \bigl(\frac{Y}{1+\alpha Z} \bigr)^n\right)^\vee=\alpha^n \left( \begin{matrix}u_{\rm univ} \cr n \cr \end{matrix}\right)  k^{\rm univ}(1+\alpha C)  \bigl(\frac{D}{1+\alpha C}\bigr)^n ,$$where $\displaystyle \left( \begin{matrix}u_{\rm univ} \cr n \cr \end{matrix}\right) :=\frac{ u_{\rm univ} (u_{\rm univ}-1) \cdots (u_{\rm univ}-n+1)}{n!}$
if $n\ge 1$ and $\displaystyle \left( \begin{matrix}u_{\rm univ} \cr 0 \cr \end{matrix}\right)=1$.

First of all, one computes that, if we write $f(X):=\exp u_{\rm univ} \log (1+X)=\sum_{n=0}^\infty a_n X^n$ as a formal power series in $X$, then $\displaystyle a_n= \left( \begin{matrix}u_{\rm univ} \cr n \cr \end{matrix}\right)$. Using (\ref{eq:dualk}), we deduce that  $$\frac{\langle \ , \ \rangle^\ast\bigl( T^{k^{\rm univ}} \bigr)}{k^{\rm univ}(1+\alpha Z)  k^{\rm univ}(1+\alpha C)}= \sum_ {n=0}^\infty \frac{\alpha^n u_{\rm univ} (u_{\rm univ}-1) \cdots (u_{\rm univ}-n+1)}{n!}  \left(\frac{Y}{1+\alpha Z} \right)^n\otimes \left(\frac{D}{1+\alpha C} \right)^n$$and the claim follows.

\subsection{An example: locally analytic functions and distributions.}\label{sec:anal}

We consider in this article weights defined as follows. Let $W$ denote the weight space seen as an adic space.

\begin{definition}
\label{def:weights}
Let $U\subset W$ be an open disk which is not a connected component of $W$ and $\Lambda_U$ denote the $\Z_p$-subalgebra of sections in $\cO_U^+(U)$.
which are bounded, see Section \S4 of \cite{EichlerShimura}. We recall that $\Lambda_U$ is a complete noetherian local $\Z_p$-algebra and we call "weak topology" the $m_{\Lambda_U}$-adic topology of $\Lambda_U$, where $m_{\Lambda_U}$ is its maximal ideal.  Let $B$ denote either $\Lambda_U$ for some open disk $U\subset W$ or $\cO_K$ for a finite extension $K$ of $\Q_p$.
In this article we will work with $B$-valued weights $k\colon\Z_p^\ast\lra B^\ast$, which, if
$B=\Lambda_U$ is the universal weight associated to $U$. 
\end{definition}

Let   $k\colon \Z_p^\ast\lra B^\ast$ be a weight as in definition \ref{def:weights} and suppose  it is an $r$-analytic character. Following \cite[Def. 3.1]{BarreraShan} we set:

\begin{definition}\label{def:anal}  For every integer $n\geq r$ let $A_{k}^o(T_0)[n]$ denote the space of functions $f\colon T_0\lra B$ such that 
\smallskip

(1) for every  $a\in \Z_p^\ast$, $t\in T_0$ we have $f(at) = k(a) f(t)$

\smallskip
(2) the function $z\rightarrow f(1,z)$ extends to an $r$-analytic  function, i.e., for every $i\in \Z/p^n \Z$ the function $f(1,i+p^n z)$ for $z\in \Z_p$ is given by the values of a convergent power series $\sum_{m=0}^\infty a_{m,i} z^m $.  

\

Define $D_k^o(T_0)[n]:={\rm Hom}_{B}\bigl(A_{k}^o(T_0)[n], B\bigr)$, the continuous dual of $A_{k}^o(T_0)[n]$ with respect to the weak topology of $B$. We write $A_{k}(T_0)[n]:=A_{k}^o(T_0)[n]\otimes_\Z \Q$ and 
$D_k(T_0)[n]:=D_k^o(T_0)[n]\otimes_\Z \Q$.

\end{definition}

By \cite[Lemma 3.1]{BarreraShan} the action of  $\Delta_1$ on $T_0$ induces actions of $\Delta_1$ on $A_k^o(T_0)[n]$, $D_k^o(T_0)[n]$, $A_k(T_0)[n]$ and $D_k(T_0)[n]$.  Morever,  \cite[Def. 3.3 \& Prop. 3.3]{BarreraShan}  the $B$ module $D_k^o(T_0)[n]$ admits a  decreasing filtration $\Fil^\bullet D_k^o(T_0)[n]$ of $B$-modules, stable under the action of $\Iw_1$, such that the graded pieces are finite and  $D_k^o(T_0)[n]$ is the inverse limit $\displaystyle{\lim_{\infty\leftarrow m}} D_k^o(T_0)[n]/\Fil^m D_k^o(T_0)[n]$.

\subsubsection{An alternative description.} For later purposes we end this section by describing $A_k(T_0)[n]$ and $D_k(T_0)[n]$ using the formalism of VBMS.  Let $T=\Z_p\oplus \Z_p$, with marked sections $s=(1,0)$ and $t=(0,1)$ in $T/p^n T$. Write $f_0=(1,0)$ and $f_1=(0,1)\in T$. Denote by $T^\vee$ the $\Z_p$-dual of $T$ and  let $T_0^\vee \subset T^\vee$ be the subset of elements $\Z_p^\ast e_0 \times \Z_p e_1 $ with $e_0=f_0^\vee$ and $e_1=f_1^\vee$. For every $\lambda=0,\ldots, p^n-1$ denote by  $\WW_k(T,s,t+\lambda s ,p^n)$, or simply $\WW_k(T,s,t+\lambda s)$ if the power of $p$ we are working with is clear from the context,  the sections $\WW_{k}(T, s,t+\lambda s)(U)$ over the adic space $U=\mathrm{Spa}(B[1/p],B)$. Let $\Iw_n$ be the subgroup of matrices $M=\left(\begin{matrix}\alpha & \beta \cr \gamma & \delta\end{matrix} \right)\in \mathrm{GL}_2(\Z_p)$ such that $\gamma\equiv 0$ modulo $p^n$. Then:

\begin{proposition}\label{prop:AkWkveee} There is an  $\Iw_n$-equivariant isomorphism of $B$-modules $$\nu_n\colon \oplus_{\lambda=0}^{p^n-1}  \WW_{k}(T, s,t-\lambda s) \longrightarrow A_{k,\lambda}^o(T_0^\vee)[n].$$Then, taking duals with respect to the weak topology on $B$ we get a decomposition into  a direct sum and a $\Iw_n$-equivariant isomorphism$$\nu_n^\vee\colon D_k^o(T_0^\vee)[n]=\oplus_{\lambda=0}^{p^n} D_{k,\lambda}^o(T_0^\vee)\cong \oplus_{\lambda=0}^{p^n-1}  \WW_{k}(T, s,t-\lambda s)^\vee.$$

\end{proposition}
\begin{proof} We describe the isomorphism explicitly. First of all notice that  $A_k^o(T_0^\vee)[n]$ decomposes as a direct sum $$A_k^o(T_0^\vee)[n]=\oplus_{\lambda=0}^{p^n-1} A_{k,\lambda}^o(T_0^\vee)$$according to residue classes: we say that $f\in A_k^o(T_0^\vee)[n]$ lies in $A_{k,\lambda}^o(T_0^\vee)$ if an only if $f(1,w)$ is zero if $w\not\in \lambda+ p^n \Z_p$. In particular, $$A_{k,\lambda}^o(T_0^\vee)=B\langle w_\lambda \rangle \cdot u^k$$where $\sum_{m=0}^\infty a_m w^m \cdot u^k\colon T_0^\vee \to B$ sends $u e_0 + v e_1 \mapsto k(u) \cdot \sum_m a_m \bigl(\frac{v/u -\lambda }{p^m}\bigr)^m  $ if $v/u\in \lambda + p^n\Z_p$ and to $0$ otherwise. The standard left action of $\Iw_n$ on $T$ is described as follows: given $M=\left(\begin{matrix}\alpha & \beta \cr \gamma & \delta\end{matrix} \right)\in \Iw_n$ we have $M (f_0) =\alpha f_0 + \gamma f_1$, $M(f_1)=\beta f_0 + \delta f_1$.  This induces a right action given by  $e_0 \cdot M =\alpha e_0 + \beta e_1$, $e_1 \cdot M=\gamma e_0 + \delta e_1$. We finally obtain the left action of $\Iw_n$ on $A_k^o(T_0^\vee)[n]$. Explicitly as $(u e_0 + v e_1)\cdot M= (\alpha u +\gamma v) e_0 + (\beta u+ \delta v) e_1$ then    $$M(w_0^m \cdot u^k ) =  k(\alpha + \gamma w_0)  \left(\frac{\beta  + \delta w_0 }{\alpha +\gamma w_0}\right)^m \cdot u^k.$$As $w_\lambda^m \cdot u^k = \left(\begin{matrix} 1 & -\lambda \cr 0 & 1 \end{matrix}\right) (w_0^m \cdot u^k) $ we get the sought for action of $\Iw_n$ on $\oplus_{\lambda=0}^{p^n-1} A_{k,\lambda}^o(T_0^\vee)$.

On the other hand consider the subfunctors $\amalg_{\lambda=0}^{p^n-1}  \bV_0(T,s,t-\lambda s) \to \bV(T)$ over the adic space $U=\mathrm{Spa}(B[1/p],B)$. The action of $\Iw_n$ on $T$ restricts to an action on this subfunctors and induces an action on $\oplus_{\lambda=0}^{p^n-1}  \WW_{k}(T, s,t-\lambda s)$. Explicitly  $$\WW_{k}(T, s,t-\lambda s)=B\langle \frac{W_\lambda}{1+p^n Z}\rangle k(1+p^n Z)$$ according to \S\ref{sec:explicit}. It contains the $B$-submodule of  the space of integral functions $B\langle X,Y\rangle$ of $\V(T)$ where $X=1+p^n Z$ and $Y= p^n W_\lambda+\lambda X$. Recall that we have a universal map $T \to B\langle X,Y\rangle$ defined by sending $m f_0 + r f_1\mapsto m X + r Y$. The left action of $\Iw_n$ on 
$T$ defines by universality an action on $B\langle X,Y\rangle$: if $M=\left(\begin{matrix} \alpha & \beta \cr \gamma & \delta \cr \end{matrix}\right)$ then $M (f_0) =\alpha f_0 + \gamma f_1$, $M(f_1)=\beta f_0 + \delta f_1$ and $M(X)=\alpha X + \gamma Y$, $M(Y)=\beta X + \delta Y$. Denote by $\Iw_n^1\subset \Iw_n$ the subgroup of matrices  with $\alpha =1 $ modulo $p^n$. We then get an action of $\Iw_n^1$  on the integral functions on $\amalg_{\lambda=0}^{p^n-1}  \bV_0(T,s,t-\lambda s)$, and hence on  $\oplus_\lambda W(T,s,t-\lambda s )$, determined on the variables $Z$ and $W_\lambda$'s by the formulas  $$M(Z)=\frac{(\alpha X-1)}{p^n} + \gamma W_0,\quad M(W_0)=\frac{\beta}{p^n} X + \delta W_0, \quad \left(\begin{matrix} 1 &- \lambda \cr 0 & 1 \end{matrix}\right)(W_0)=W_{\lambda}.$$Notice that if $M=\left(\begin{matrix} \alpha & 0 \cr 0 & \delta \cr \end{matrix}\right) \in \Iw_n$ then  $M\bigl(k(1+p^n Z)\bigr)=k(\alpha) k(1+p^n Z)$ and $M(W_0)=\delta W_0$ giving explicitly the action of diagonal matrices on each  $W_{k}(T, s,t-\lambda s)$. For every $\lambda=0,\ldots, p^n-1$  define the map $$\nu_\lambda\colon W_{k}(T, s,t-\lambda s )\longrightarrow A_{k,\lambda}^o(T_0^\vee)[n] , \quad k(1+p^n Z) \sum_i a_i \left(\frac{W_\lambda}{1+p^nZ}\right)^i\mapsto \sum_i a_i w_\lambda^i \cdot u^k.$$It is clearly an isomorphism of $B$-modules. We are left to show that $$\nu_n:=\sum_{\lambda=0}^{p^n-1} \nu_\lambda\colon \oplus_{\lambda=0}^{p^n-1}  \WW_{k}(T, s,t-\lambda s)\to \oplus_{\lambda=0}^{p^n-1} A_{k,\lambda}^o(T_0^\vee)[n]=A_k^o(T_0^\vee)[n] $$is $\Iw_n$-equivariant.

Consider the $B$-linear map $\xi\colon B\langle X,Y\rangle \to A(T^\vee)$, where $A(T\vee)$ is the ring of analytic functions from $T^\vee$ to $B$, sending $f(X,Y)=\sum_{h,m} a_{h,m} X^h Y^m$ to the function $\xi\bigl(f(X,Y)\bigr)\colon T^\vee \to B$, $u e_0 + v e_1 \to \sum_{h,m} a_{h,m} u^hv^m$. This map is $\Delta_1$-equivariant. Indeed,  given $M\in \Delta_1$ such that $M (f_0) =\alpha f_0 + \gamma f_1$, $M(f_1)=\beta f_0 + \delta f_1$ then $M (e_0) =\alpha e_0 + \beta e_1$, $M(e_1)=\gamma e_0 + \delta e_1$
so that $M(u e_0 + v e_1)= (u \alpha + v \gamma) e_0 + (u\beta + \delta v) e_1$. Hence $M\bigl(\xi(f(X,Y)) \bigr)=\xi\bigl(f\bigl(M(X),M(Y)\bigr)\bigr)$. As $\nu$ is determined by $\xi$ using that $X=1+p^nZ$ and $Y= p^n W_\lambda+\lambda X$ this implies that $\nu$ is $\Iw_n$-equivariant as well. 

\end{proof}

Note that we have a $\Iw_n$-equivariant map of functors, and hence of representing objects, $\amalg_{\lambda\in \Z/p^n\Z}\bV_0(T,s,t-\lambda s) \to \bV_0(T,s)$. This provides a $\Iw_n$-equivariant map $$\WW_{k}(T, s) \to \oplus_{\lambda=0}^{p^n-1}  \WW_{k}(T, s,t-\lambda s).$$Let $Q_n\subset T^\vee/p^n T^\vee$ be the $(\Z/p^n\Z)$-dual of the quotient $\bigl(T/p^n T\bigr)/ (\Z/p^n\Z) s$.   The duality between $$\zeta_k\colon \WW_{k}(T, s)^\vee \lra \WW_{k}^D(T^\vee, s^\vee,Q_n)$$of Definition \ref{def:dualityWkWk^D} composed with the $\Iw_n$-equivariant isomorphism   $$\nu_n^\vee\colon D_k^o(T_0^\vee)[n] \cong \oplus_{\lambda=0}^{p^n-1}  \WW_{k}(T, s,t-\lambda s)^\vee $$of  Proposition \ref{prop:AkWkveee} give the following:

\begin{corollary}\label{cor:DkWkD} We have a $\Iw_n$-equivariant, $B$-linear map  $$D_k^o(T_0^\vee)[n] \cong \oplus_{\lambda=0}^{p^n-1}  W_{k}(T, s,t-\lambda s)^\vee\lra  \WW_{k}^D\bigl(T^\vee, s^\vee,Q_n\bigr).$$
\end{corollary}

\subsubsection{The $U_p$ operator.}\label{sec:HeckeD} For $\rho=0,\ldots,p-1$ let $\pi_\rho\colon T \to T$ be the map defined by $\left(\begin{matrix} 1 & \rho \cr 0 & p \cr \end{matrix}\right) $, i.e.,
$f_0\mapsto f_0$, $f_1\mapsto p f_1+ \rho f_0$. It defines the map $\pi_\rho^\vee\colon T^\vee \to T^\vee$ that sends $u e_0+ v e_1 \mapsto u e_0+(pv+\rho) e_1$. In particular taking $(\pi_\rho^\vee)^\ast$ it induces a map $A_{k}^o(T_0^\vee)[n+1] \to A_{k}^o(T_0^\vee)[n]$ that is $0$ on $A_{k,\lambda}^o(T_0^\vee)[n+1] $ for $\lambda\not\equiv \rho$ modulo $p$ and it induces a map $A_{k,\lambda}^o(T_0^\vee)[n+1] \to A_{k,\lambda_0}^o(T_0^\vee)[n] $ if $\lambda\equiv \rho+ p \lambda_0$. Taking the sum over the $\rho$'s we get a map $\pi_n=\sum_{\rho=0}^{p-1} (\pi_\rho^\vee)^\ast$, where  $$\pi_n \colon A_{k}^o(T_0^\vee)[n+1]=\oplus_{\rho=0}^{p-1} \oplus_{\lambda_0=0}^{p^n}  A_{k,\rho+p\lambda_0}^o(T_0^\vee)[n+1] \to \oplus_{\lambda_0\in\ Z/p^n \Z} A_{k,\lambda_0}^o(T_0^\vee)[n]=A_{k}^o(T_0^\vee)[n].$$

Notice that $\pi_\rho$ defines by functoriality a map $\bV_0(T,s,t-\lambda s,p^n) \to \bV_0(T,s,t-(\rho+p\lambda) s,p^{n+1})$ (we have added the dependence on the power of $p$ in the definition of $\bV_0$ to avoid confusion). This gives a map $\mu_\rho\colon \WW_{k}(T, s,t-(\rho+p\lambda) s,p^{n+1})\to \WW_k( s, t-\lambda s,p^{n}) $ and, summing over all $\rho$'s, 
$$\mu_n=\sum_{\rho=0}^{p-1}\mu_\rho\colon\oplus_{\rho=0}^{p-1}\oplus_{\lambda=0}^{p^n}  \WW_{k}(T, s,t-(\rho+p\lambda) s,p^{n+1})\to \oplus_{\lambda\in \Z/p^n\Z}  \WW_{k}(T, s,t-\lambda s,p^{n}) $$

\begin{lemma} With the notation of Proposition \ref{prop:AkWkveee} we have $\pi_n \circ \nu_{n+1}=\nu_n \circ \mu_n$ and similarly taking strong duals  $ \nu_{n+1}^\vee \circ \pi_n^\vee=\mu_n^\vee \circ \nu_n^\vee$.

\end{lemma}
\begin{proof} This is an explicit computation using the notation of the proof of Proposition \ref{prop:AkWkveee} and follows from the fact that $ (\pi_\rho^\vee)^\ast$ sends $w_{\rho+p\lambda_0} \mapsto w_{\lambda_0}$ and $\mu_\lambda$ sends $W_{\rho+p\lambda}\mapsto W_\lambda$. 

\end{proof}

\section{The modular curve setting.}

Let $p>0$ be a prime integer. We fix once for all
the $p$-adic completion $\C_p$of an algebraic closure of $\Q_p$.   We denote by $v$ the valuation on $\C_p$, normalized such that
$v(p)=1$.

Let $N\geq 5$ and $r\geq 0$ be integers with $N$ prime to $p$ and let $X_0(p^r,N)$, resp.~$X_1(p^r,N)$, resp.~$X(p^r,N)$ be the modular curves over $\C_p$ of level $\Gamma_1(N) \cap \Gamma_0(p^r)$, resp.~$\Gamma_1(N)\cap \Gamma_1(p^r) $,  resp.~$\Gamma_1(N)\cap \Gamma(p^r) $. Over the complement of the cusps of the modular curve $X_0(p^s,N,p) $ we have a universal elliptic curve $E$, a cyclic subgroup $H_s\subset E[p^s]$ of order $p^s$ and an embedding $\Psi_N\colon \mu_N\hookrightarrow E[N]$. For $X_1(p^s,N,p) $ we further have a generator of $H_s$.

We denote the associated adic space over $\mathrm{Spa}(\C_p,\cO_{\C_p}\bigr)$ by $\cX_0(p^r,N)$, resp.~$\cX_1(p^r,N)$, resp.~$\cX(p^r,N)$
considered as adic spaces with logarithmic structures given by the cusps, with reduced structure, as in \cite[Ex.~2.3.17]{Diao}. We simply write $X$, resp.~$\cX$ for $X_0(p^0,N)$, resp.~$\cX(p^0,N)$. Notice that the $p^r$-torsion of the universal elliptic curve $E$ over the complement of the cusps in $X$ defines a locally constant sheaf for the finite Kummer \'etale topology that we denote by $E[p^r]$  and $\cX(p^r,N)\to \cX$ is the finite Kummer \'etale Galois cover, with group $\GL_2(\Z/p^r\Z)$, defined by trivializing it. We let $T_p(E)$ be the sheaf on $X_\profket$, resp.~$\cX_\profket$ defined by the inverse limt $\lim E[p^r]$.  Thanks to \cite[Thm.~III.3.3.17]{ScholzeTorsion}  we have

\begin{itemize} 

\item[(i)] a unique perfectoid space $\cX(p^\infty,N)$ such that $\cX(p^\infty,N) \sim \displaystyle{\lim_{\infty\leftarrow r}} \cX(p^r,N)$  in the sense of \cite[Def.~2.4.1]{ScholzeWeinstein};

\item[(ii)] the Hodge-Tate period map $\pi_{\rm HT}\colon \cX(p^\infty,N) \longrightarrow \mathbb{P}^1_{\Q_p}$. 

\end{itemize}

In particular, we have morphisms of adic spaces $\pi_r\colon \cX(p^\infty,N)\to \cX(p^r,N)$, compatible for varying $r \geq 0$, inducing a homeomorphism of the underlying topological spaces $\vert \cX(p^\infty,N) \vert \cong \displaystyle{\lim_{\infty\leftarrow n}} \vert \cX(p^r,N) \vert$.

\subsection{On the pro-Kummer \'etale topology of modular curves.} For every $s\in \N$ we write $\displaystyle{\lim_{\infty \leftarrow r}} \cX(p^r,N)$  (for $r\geq s$) for the pro-finite Kummer \'etale cover of $\cX_0(p^s,N)$ defined by the Kummer \'etale covers $\cX(p^r,N)\to \cX_0(p^s,N)$, for $r\geq s$. The following lemma provides a basis for the pro-Kumemr \'etale topology of $\cX_0(p^s,N)$:

\begin{lemma}\label{lemma:proketbasis} For every $s\in\N$ the site $\cX_0(p^s,N)_\proket$ has a basis consisting of log affinoid perfectoid opens $U$ that are pro-Kummer \'etale over an affinoid perfectoid open of $\displaystyle{\lim_{\infty \leftarrow r}} \cX(p^r,N)$  (for $r\geq s$). In particular, for any such open $U$, $T_p(E)\vert_U$ is a constant sheaf and such basis is closed under fibre products over $\cX_0(p^s,N)$.

\end{lemma}
\begin{proof} Due to  \cite[Prop.~5.3.12]{Diao} the site  $\cX_0(p^s,N)_\proket$ admits a basis consisting of log affinoid perfectoid opens and thanks to \cite[Prop.~5.3.11]{Diao} the category of such bases is closed under fibre products. Given any such  $W$, consider the fibre product  $Z$ of $W$ and $\displaystyle{\lim_{\infty \leftarrow r}} \cX(p^r,N)$  over $\cX_0(p^s,N)$.  By \cite[Cor.~5.3.9]{Diao} such a fibre product exists, it is  pro-Kummer \'etale over $\displaystyle{\lim_{\infty \leftarrow r}} \cX(p^r,N)$ and it is a  pro-finite Kummer \'etale cover of $W$. As a cover of $W$ can be represented as $\displaystyle{\lim_{\infty \leftarrow r}} W_r$ with $W_r:=\cX(p^r,N) \times_{\cX_0(p^s,N)}  W$. In particular, thanks to \cite[Lemma 5.3.8]{Diao}, each $W_r$ is a  finite  \'etale cover of $W$ so that 
$Z=\displaystyle{\lim_{\infty \leftarrow r}} W_r \to W$ is a pro-finite  \'etale cover of $W$ and   $Z$ is log affinoid perfectoid by  \cite[Cor. 5.3.9]{Diao}. This and the fact that $W$ is log affinoid perfectoid implies that $Z$ is also log affinoid perfectoid. Recall  from  \cite[Cor.~III.3.11]{ScholzeTorsion} that $\displaystyle{\lim_{\infty \leftarrow r}} \cX(p^r,N)$  is covered by  perfectoid affinoid open subsets. Taking the fibre product over $\displaystyle{\lim_{\infty \leftarrow r}} \cX(p^r,N)$  of $Z$ with a cover of $\displaystyle{\lim_{\infty \leftarrow r}} \cX(p^r,N)$ by  perfectoid affinoid open subsets, the claim follows.
\end{proof} 

\begin{remark} Endow $\cX(p^\infty,N)$ with the limit log structure. Then, $\cX(p^\infty,N)$ is {\it not} covered by log affinoid perfectoid opens as condition (d) of the definition above is not satisfied. 

This condition is used in \cite[Cor. 5.3.8]{Diao}, an analogue of Abhyankar's lemma, stating  that, for a log affinoid perfectoid, the finite Kummer \'etale site coincides with the finite \'etale site. This was already used in the proof of Lemma \ref{lemma:proketbasis}.
\end{remark}

\subsection{Standard opens.} \label{rmk:standardopens} 
We start by defining certain opens of $\PP^1$, namely let for every $n\ge 1$ $U_0, U_\infty, U_\infty^{(n)}, U_0^{(n)}\subset\PP^1$ be defined as follows. Let $T$ denote a parameter at $0$ on $\PP^1$ then

a) $\displaystyle U_\infty=\{x\in \PP^1 \ | \ \Vert\frac{1}{T}\Vert_x\le 1\}$,  $U_0=\{x\in \PP^1 \ | \ \Vert T-\lambda\Vert_x\le 1, \mbox{ for some }\lambda\in \{0,1,...,p-1\}\}$, 

b) $\displaystyle U_\infty^{(n)}=\PP^1\bigl( \frac{1}{p^nT} \bigr)=\{x\in \PP^1\ | \ \Vert\frac{1}{T}\Vert_x\le \Vert p^n\Vert_x\}$, 

c) $U_0^{(n)}=\cup_{\lambda} U_{0,\lambda}^{(n)}$ with $\lambda=\lambda_0+\lambda_1p+...+\lambda_{n_1}p^{n-1}$, where $\lambda_0,\lambda_1,...,\lambda_{n-1}\in \{0,1,...,p-1\}$ and we have
$\displaystyle U_{0,\lambda}^{(n)}:=\PP^1\bigl( \frac{T-\lambda}{p^n} \bigr)$. 

We remark that for every $n\ge 1$ we have $\PP^1(\Q_p, \Z_p)\subset U_\infty^{(n)}\cup U_0^{(n)}$ and 
moreover the family $\{U_\infty^{(n)}\cup U_0^{(n)}\}_{n\ge 1}$ is a fundamental system of open neighborhoods of $\PP^1(\Q_p,\Z_p)$ in $\PP^1$.

\

We recall from \cite[Thm.~III.3.3.18]{ScholzeTorsion}  that for every rational open subset  $U'\subset U_0$ or $U'\subset U_\infty$ the inverse image $\cU':=\pi_{\rm HT}^{-1}\bigl(U'\bigr)$ is an affinoid perfectoid open subspace of $\cX(p^\infty,N)$. In particular, it is quasi-compact so that it is also the inverse image of an open $\cU'_r$ via $\pi_r$ of $ \cX(p^r,N)$ for some $r$ using the homeomorphism $\vert \cX(p^\infty,N) \vert \cong \displaystyle{\lim_{\infty\leftarrow n}} \vert \cX(p^r,N) \vert$. As the transition maps in the inverse limit are finite and surjective, $\cU'_r$ is in fact the image of $\cU'$ via $\pi_r$ for $r$ large enough. If $U'$ is further invariant for the action of the Iwahori subgroup $\Iw_s\subset \GL_2(\Z_p)$ of matrices which are upper triangular modulo $p^s$, then also $\cU'$ is   $\Iw_s$-invariant as $\pi_{\rm HT}$ is $\Iw_s$-equivariant  and $\cU'$ is the inverse image of   a unique open $\cU'_{0,s}$ of $ \cX_0(p^s,N)$. Indeed, given $\cU'_r \subset \cX(p^r,N)$ for some $r\geq s$ such that its inverse image gives $\cU'$, then $\cU'_r$ is  $\Iw_s$-invariant. As  the morphism $\cX(p^r,N) \to \cX_0(p^s,N)  $ is finite Kummer \'etale and  Galois with group $G_s$ equal to the image of $\Iw_s\subset \GL_2(\Z_p)\to \GL_2(\Z/p^r\Z)$ then $\cU'_{0,s}:=\cU'_r/G_s$ is an open of $ \cX_0(p^s,N) $ with the required properties. Furthermore, $\cU'_{0,s}$ defines the open  $(\cU'_r)_{r\geq s}$ for the pro-Kummer \'etale site of $ \cX_0(p^s,N)$, and hence of $\cX$, given by $\cU'_r:=\cU'_{0,s} \times_{ \cX_0(p^s,N)} \cX(p^r,N)$. By construction $\cU'\sim    \displaystyle{\lim_{\infty\leftarrow r}} \cU'_r$.

In particular, for every $n\ge 1$ we consider the open rational  subspaces  $U_\infty^{(n)}$ of $U_\infty$ and  $U_0^{(n)}$ of~$U_0$ defined above. We remark that $U_\infty^{(n)}$ is invariant under the left action of the subgroup 
${\rm Iw}_n$. Then we denote by $\cX(p^\infty,N)_0^{(n)}:=\pi_{\rm HT}^{-1}\bigl(U_0^{(n)} \bigr) $ and $\cX(p^\infty,N)_\infty^{(n)}:=\pi_{\rm HT}^{-1}\bigl(U_\infty^{(n)} \bigr) $ and recall that they define affinoid perfectoid open subspaces of $ \cX(p^\infty,N)$.

As explained above, they also define opens for the pro-Kummer \'etale site of $ \cX_0(p^n,N)$ and of $\cX$ respectively. Namely,  for $n\ge 1$, $\cX(p^\infty, N)_\infty^{(n)}$, being invariant under ${\rm Iw}_n$, descends to an affinoid open denoted $\cX_0(p^m, N)_\infty^{(n)}$ of $\cX_0(p^m, N)$, for all $m\ge n$. We also have variants $\cX_1(p^m, N)_\infty^{(n)}$, resp.~$\cX(p^m, N)_\infty^{(n)}$ if we descend to $\cX_1(p^m, N)$, resp.~$\cX(p^m, N)$.

In contrast, as $U_0^{(n)}$ is invariant with respect to ${\rm Iw}_1$, the open $\cX(p^\infty, N)_0^{(n)}$ descends to an open affinoid denoted $\cX_0(p^m, N)_0^{(n)}$ of $\cX_0(p^m, N)$, for all $m\ge 1$. In this case we consider the variant $\cX(p^m, N)_0^{(n)}$ open of $\cX(p^m, N)$.

\begin{lemma} For every $h\in \N$ there exists  $n=n(h)\geq 1$ such that for every $r\geq n$ the universal elliptic curve over $\cX_0(p^r,N)_0^{(n)}$ and $\cX_0(p^r,N)_\infty^{(n)}$ resp.~admits a canonical subgroup of order $p^h$.

\end{lemma}
\begin{proof} This follows from \cite[Lemma III.3.3.14]{ScholzeTorsion}, stating that the pre-image via $\pi_{\rm HT}$ of  $\PP^1(\Q_p, \Z_p)$ is, as a topological space, the closure of  the inverse image in $\cX(p^\infty,N)$ of the ordinary locus and the cusps of $\cX$. 

In particular,  $\cX(p^\infty,N)_0^{(n)}$ and $ \cX(p^\infty,N)_\infty^{(n)}$ define a fundamental system of open neighborhoods of the ordinary locus in $\cX(p^\infty,N)$. 

\end{proof}

\begin{remark} Recall that  the ordinary locus in $ \cX_0(p,N)$ has two connected components. Then  $\cX(p,N)_0^{(1)}$ and $\cX(p,N)_\infty^{(1)}$ are neighborhoods of these two components. The first is defined by requiring that the level subgroup is {\it not} the canonical one while the second is the component where the level subgroup {\it coincides} with  the canonical one. Following conventions going back to Robert Coleman we set up the notation so that the first is indexed by $0$ and the second by $\infty$.

\end{remark}

A reason to introduce the open subsets $\cX_0(p^r,N)_\infty^{(n)}$ for $r\ge n$ and $\cX_0(p^r,N)_0^{(n)}$, for any $r \geq 1$, is that they behave nicely under the $U_p$-correspondence, as we will explain bellow.

\subsection{On the Hodge-Tate period map.} Over $\cX(p^\infty,N)$ the sheaf $T_p(E)$ admits a universal trivialization $$T_p(E)=\Z_p a \oplus \Z_p b.$$The map ${\rm dlog}$ defines a surjective map $$T_p E^\vee \otimes_{\Z_p} \cO_{\cX(p^\infty,N)} \longrightarrow \widehat{\omega}_E$$which is used to define the map $\pi_{\rm HT}$: for every log  affinoid perfectoid open $W=\mathrm{Spa}(R,R^+)$ of $\cX(p^\infty,N)$ such that the universal elliptic curve extends to a (generalised) elliptic curve over $\Spec (R^+)$ and  $\omega_E$  is generated as $R^+$-module by one element that we denote $\Omega_W$, we write ${\rm dlog}(a^\vee)= \alpha \Omega_W$, ${\rm dlog}(b^\vee)= \beta \Omega_W$ with $\alpha$, $\beta\in R$ generating the whole ring $R$. Then $\pi_{\rm HT}\vert_W\colon W \to \mathbb{P}^1$ is defined in homogeneous coordinates by $[\alpha;\beta]$. Namely let $W_\infty\subset W$ be the rational open defined by $W(1/\alpha)$ and  let $W_0\subset W$ be the rational open defined by $W(1/\beta)$. Then  $\pi_{\rm HT}\vert_{W_0}\colon W_0\to U_0$ sends the standard coordinate $T$ on the standard affinoid neighorhood $U_0=\mathbb{A}^1$ of $0$ to $\alpha/\beta$ and $\pi_{\rm HT}\vert_{W_\infty}\colon W_\infty\to U_\infty$  sends the standard coordinate $T$ on the standard affinoid neighorhood $U_\infty:=\mathbb{A}^1$ of $\infty$ to $\beta/\alpha$.

For every $n\in\N$  we can refine such morphism to a morphism on $\cX_0(p^n,N)_\proket$ as follows. Considering  the smooth formal model $\mathfrak{X}$ of the modular curve $X$ over $\cO_{\C_p}$ given by moduli theory and the universal generalised elliptic curve $E$ over $\mathfrak{X}$, the invariant differentials of $E$ relative to $\mathfrak{X}$ and the fact that $\cX$ is the adic generic fiber of $\mathfrak{X}$, give an invertible $\cO_{\cX}^+$-module $\omega_E^+$ on $\cX$. Pulling back via the projection map $\cX_0(p^n,N)_\proket \to \cX $ we get a $\cO_{\cX_0(p^n,N)_\proket}^+$-module  for the pro-Kummer \'etale topology that we still, abusively denoted $\omega_E^+$ and passing to $p$-adic completions we finally obtain an invertible $\widehat{\cO}_{\cX_0(p^n,N)}^+$-module $ \widehat{ \omega}_E^+$. Here for simplicity we write $\widehat{\cO}_{\cX_0(p^n,N)}^+$ for $\widehat{\cO}_{\cX_0(p^n,N)_\proket}^+$.

Consider the map ${\rm dlog}$ for the basis of  $\cX_0(p^n,N)_\proket$ given in Lemma \ref{lemma:proketbasis} where $a$ modulo $p^n$ is a generator in $T_ p(E)/ p^n T_p(E)$ of the level subgroup of order $p^n$.  For every log affinoid perfectoid open $U$ as in loc.~cit., write $\widehat{U}=\mathrm{Spa}(R,R^+)$. Then   $T_p(E)^\vee$ is constant on $U$, we have a universal generalised elliptic curve $E$ over $\Spec(R^+)$  and  $\widehat{\cO}_{\cX_0(p^n,N)}^+(U)=R^+$. We then have the map ${\rm dlog}\colon T_p(E)^\vee(U) \otimes_{\Z_p} R^+  \longrightarrow \widehat{ \omega}_E^+(U)$. Gluing we obtain a map of sheaves on $\cX_0(p^s,N)_\proket$:

\begin{equation}\label{eq:dlog} 
{\rm dlog}\colon T_p(E)^\vee \otimes_{\Z_p} \widehat{\cO}_{\cX_0(p^n,N)}^+  \longrightarrow \widehat{ \omega}_E^+\end{equation}

\begin{proposition}\label{proposition:existcanonical} For every $r\in \N$  there exists $m\in\N$ such that 

\begin{itemize}

\item[a.] there exists a canonical subgroup $C_n$ of order $p^n$ on $\cX_0(p^n,N)_\infty^{(m)}$ and $\cX_0(p^n,N)_0^{(m)}$ which on  $\cX_0(p^n,N)_\infty^{(m)}$ coincides with the level $n$ subgroup and which on  $\cX_0(p^n,N)_0^{(m)}$ is disjoint from the level $n$ subgroup;

\item[b.] there exists an invertible $\cO_\cX^+$-module  $\omega_E^{\rm mod}$ on  $\cX_0(p^n,N)_\infty^{(m)}$ and on  $\cX_0(p^n,N)_0^{(m)}$ contained in $\omega_E^+$;

\item[c.] the map ${\rm dlog}$ surjects onto $\omega_E^{\rm mod}\otimes_{\cO_{\cX_1(p^n,N)}^+}\widehat{\cO}_{\cX_1(p^n,N)}^+ $ on $\cX_1(p^n,N)_{\infty,\proket}^{(m)}$ and modulo $p^r$  factors via $C_n^\vee$. The kernel of ${\rm dlog}$ is isomorphic to $\bigl(\omega_E^{\rm mod}\bigr)^{-1}\otimes_{\cO_{\cX_1(p^n,N)}^+}\widehat{\cO}_{\cX_1(p^n,N)}^+ $ (here we omit the Tate twist that usually appears as we work over $\C_p$);

\item[d.]  the map ${\rm dlog}$ surjects onto $\omega_E^{\rm mod}\otimes_{\cO_{\cX(p^n,N)}^+}\widehat{\cO}_{\cX(p^n,N)}^+ $ on $\cX(p^n,N)_{0,\proket}^{(m)}$ and modulo $p^s$ factors via $C_n^\vee$. The kernel of ${\rm dlog}$ is isomorphic to $\bigl(\omega_E^{\rm mod}\bigr)^{-1}\otimes_{\cO_{\cX(p^n,N)}^+}\widehat{\cO}_{\cX(p^n,N)}^+ $.

\end{itemize}

\end{proposition}
\begin{proof} It follows from \cite{AIP} that the result holds true on strict neighborhoods $\cX\bigl(p/\mathrm{Ha}^{p^m}\bigr)$ of the ordinary loci in $\cX$ defined by the points $x$ where  $\vert p \vert_x < \vert \mathrm{Ha}^{p^m} \vert_x$ for $m$ large enough; here $\mathrm{Ha}$ is a (any) local lift of the Hasse invariant.  Thanks to \cite[Lemma III.3.3.8]{ScholzeTorsion} there exists $n\in\N$ such that $\cX(p^r,N)_\infty^{(n)}$ and $\cX(p^r,N)_0^{(n)}$ are contained in   $\cX\bigl(p/\mathrm{Ha}^{p^m}\bigr)$.  The claim follows. 

\end{proof}

\begin{remark} See \cite{CHJ} for similar results in the case of Shimura curves. The notation $\omega_E^{\rm mod}$ is taken from  \cite{PS}.

\end{remark}

From Proposition \ref{proposition:existcanonical} we get an integral version of the Hodge-Tate exact sequence: $$0 \rightarrow \omega_E^{\rm mod,-1}\otimes_{\cO_{\cX_0(p^n,N)}^+}\widehat{\cO}_{\cX_0(p^n,N)_\infty^{(m)}}^+  \longrightarrow T_p(E)^\vee \otimes_{\Z_p} \widehat{\cO}_{\cX_0(p^n,N)_\infty^{(m)}}^+ \longrightarrow  \omega_E^{\rm mod}\otimes_{\cO_{\cX_0(p^n,N)}^+}\widehat{\cO}_{\cX_0(p^n,N)_\infty^{(m)}}^+ \rightarrow 0 .$$This will be useful to compute the cohomology $\mathrm{H}^1\bigl(\cX_1(p^n,N)_{\infty,\proket}^{(m)}, \widehat{\cO}_{\cX_1(p^s,N)_\infty^{(n)}}  \bigr)$. Infact, tensoring the exact sequence with $\omega_E^{\rm mod}$ and taking the long exact sequence in cohomology we obtain a map 

\begin{equation}\label{eq:H0H1}\mathrm{H}^0\bigl(\cX_0(p^n,N)_\infty^{(m)}, \omega_E^{\rm mod,2}\bigr)  \longrightarrow \mathrm{H}^1\bigl(\cX_0(p^s,N)_{\infty,\proket}^{(n)}, \widehat{\cO}_{\cX_1(p^n,N)_\infty^{(m)}}^+  \bigr)\end{equation}and simlarly for $\cX_0(p^n,N)_0^{(m)}$, or their covers $\cX_1(p^n,N)_\infty^{(m)}$, $\cX(p^n,N)_0^{(m)}$.

\subsection{The sheaf $\omega_E^k$.}\label{sec:omegak}

The map ${\rm dlog}$ provides $\omega_E^{\rm mod}/p^r \omega_E^{\rm mod} $, where $r$ is like in Proposition  \ref{proposition:existcanonical}, with a marked section $s$ over $\cX_1(p^n,N)_\infty^{(m)}$ as the image of the tautological generator of $C_n^\vee$. 

Similalry, recall that we have a decomposition $\cX(p^n,N)_0^{(m)}=\amalg_{\lambda\in \Z/p^n\Z} \cX(p^n,N)_{0,\lambda}^{m)}$ where over $\cX(p^n,N)_{0,\lambda}^{m)}$, using the trivialization $T_p(E)/p^n T_p(E)=(\Z/p^n\Z) a \oplus (\Z/p^n\Z) b$ so that $a$ generates the level subgroup $H_n$, we have ${\rm dlog} (a^\vee)=\lambda {\rm dlog} (b^\vee)$. In particular, the canonical subgroup $C_n$ is generated by $b+\lambda a$ and $\omega_E^{\rm mod}/p^r \omega_E^{\rm mod} $ acquires a marked section $s:={\rm dlog} (b^\vee)$.

Suppose we are given an $h$-analytic character $k\colon \Z_p^\ast \to B^\ast$, for $h\leq r$, as in Definition \ref{def:weights}. Using the formalism of VBMS from \S\ref{sec:WkWDk} for $\omega_E^{\rm mod}$ and the section $s$ modulo $p^r$  we get the invertible $\cO_{\cX_1(p^n,N)_\infty^{(m)}}^+\widehat{\otimes} B$-module, resp. $\cO_{\cX(p^n,N)_0^{(m)}}\widehat{\otimes} B$-module  $$\omega_E^k:=\WW_{k}(\omega_E^{\rm mod}, s).$$Since $\omega_E^{\rm mod}$ and the section $s$ modulo $p^r$, are stable under the action of the automorphism group $(\Z/p^n\Z)^\ast$ of    $j_n\colon \cX_1(p^n,N)_\infty^{(m)}\to \cX_0(p^n,N)_\infty^{(m)}$ then $(\Z/p^n\Z)^\ast$ acts on $j_{n,\ast}\bigl(\omega_E^k[1/p]\bigr)$ and, taking  the subsheaf of $j_{n,\ast}\bigl(\omega_E^k[1/p]\bigr)$ on which $\Z_p^\ast$ acts via the character $k$, then $j_{n,\ast}\bigl(\omega_E^k[1/p]\bigr)$ descends to an invertible $\cO_{\cX_0(p^n,N)_\infty^{(m)}}\widehat{\otimes} B[1/p]$-module that we denote $\omega_E^k[1/p] $.

Similalry the Galois group $\Delta_n$ of $j_n\colon \cX(p^n,N)_0^{(m)}\to \cX_0(p^n,N)_0^{(m)}$, which is the Borel subgroup of $\mathrm{GL}_2(\Z/p^n\Z)$  acts compatibly on $(\omega_E^{\rm mod},s)$  and hence it acts on $j_{n,\ast}\bigl(\omega_E^k[1/p]\bigr)$.  In this case we let $\omega_E^k[1/p]$ be the invertible $\cO_{\cX_0(p^n,N)_0^{(m)}}\widehat{\otimes} B[1/p]$-module defined as the subsheaf of $j_{n,\ast}\bigl(\omega_E^k[1/p]\bigr)$ on which the Borel $\mathrm{GL}_2(\Z_p)$ acts via the projection onto the lower right entry $\Z_p^\ast$ composed with  the character $k$.

\subsection{The $U_p$-correspondence.}\label{sec:Up} 

Given the modular curve $\cX_0(p^s,N)$ we have correspondences $T_\ell$, for $\ell$ not dividing $p N$, and $U_p$. They are defined by the analytification $\cX_0(p^s,N, \ell)$, resp. $\cX_0(p^s,\N,p)$ of the modular curve  $X_0(p^s,N,\ell)$, resp. $X_0(p^s,N,p) $ classifying, at least away from the cusps, subgroups $D$ of order $\ell$ of the universal elliptic curve $E$, resp. subgroups of order $p$ of $E$ complementary to the $p$-torsion $H_1$ of the cyclic subgroup $H_s$ of order $p^s$ defined by the level structure. We have two maps $q_1,q_2\colon ,\cX_0(p^s,N, \ell) \to \cX_0(p^s,N)$ defined by the analytification of the maps  $q_1,q_2\colon X_0(p^s,N,\ell)\to X_0(p^s,N) $ where $q_1$ sends the universal object $(E,H_s,\Psi_N,D)$ to $(E,H_s,\Psi_N)$ (the forgetful map) and $q_2$ sends the universal object $(E,H_S,\Psi_N,D)$ to $(E/D, H_s',\Psi_N')$ where $H_s'$ is the image of $H_s$ via the isogeny $E\to E/D$ and $\Psi_N'$ is $\psi_N$ composed with this isogeny. 

 These maps induce morphisms of sites $\cX_0(p^s,N,\ell )_\proket \to \cX_0(p^s,N)_\proket  $, resp. $\cX_0(p^s,N,p )_\proket \to \cX_0(p^s,N)_\proket  $. As $q_1, q_2$ are finite Kummer \'etale the following follows from the discussion in \cite[Cor. 2.6]{EichlerShimura} or \cite[Prop. 4.5.2]{Diao}:

\begin{lemma}\label{lemma:trace}  There is a natural isomorphism of functors $q_{i,\ast}\cong q_{i,!}$. In particular, $q_{i,\ast}$ is exact and we have a natural transformation   ${\rm Tr}_{q_i}\colon  q_{i,\ast} q_i^\ast \to {\rm Id}$, called the trace map.
\end{lemma}

The maps $q_1$ and $q_2$ induce maps of perfectoid spaces. For $\ell\neq p$ the fibre product $\cX(p^\infty,N,\ell ):=\cX(p^\infty,N) \times_{\cX_0(p^s,N)}^{q_i} \cX_0(p^s,N,\ell)$ exists and is independent of the choice of maps $q_1$ or $q_2$ and we then have the two projections $q_1,q_2\colon \cX(p^\infty,N,\ell ) \to \cX(p^s,N,\ell)$. For $\ell=p$ conisder $\cX(p^\infty,N,p ):=\cX(p^\infty,N) \times_{\cX_0(p^s,N)}^{q_i} \cX_0(p^s,N,p)$. Then $\cX(p^\infty,N,p )$ splits completely as $$\cX(p^\infty,N,p )=\amalg_{\lambda=0,\ldots,p-1} \cX(p^\infty,N),$$where over the copy labeled by $\lambda=0,\ldots,p-1$ the isogeny $E\to E':=E/D$ is given by and produces the map of $\Z_p$-modules $$u_\lambda \colon T_p E=\Z_p a \oplus \Z_p b \to T_p(E'):=T_p(E_\lambda):=\Z_p a' \oplus \Z_p b', \qquad a'=a, b'=\frac{b+\lambda a}{p}\in T_p E \otimes \Q_p.$$Using this description, the maps $q_1$ and $q_2$ restricted to the component labeled $\lambda$ define maps $q_{1,\lambda}$ and $q_{2,\lambda}$ where $q_{1,\lambda}$ is the identity map and $q_{2,\lambda}\colon \cX(p^\infty,N) \to \cX(p^\infty,N) $ is an isomorphism such that the pull-back of $T_p E$ is $T_p(E_\lambda)$. 

\

In particular, consider the maps $t_1,t_2\colon  \amalg_{\lambda=0,\ldots,p-1} \PP^1_{\Q_p} \to \PP^1_{\Q_p}$ where on the component labeled by $\lambda$ the map $t_{1,\lambda}$ induced by $t_1$ is the identity while the map $t_2$ is the isomorphism $t_{2,\lambda}\colon \PP^1_{\Q_p} \to \PP^1_{\Q_p}$ 
defined on points by $[\alpha,\beta]\mapsto [\alpha - \lambda \beta,p \beta]$. We then have the following  diagram 

$$\begin{matrix} \cX(p^\infty,N) & \stackrel{q_2}{\longleftarrow} &  \cX(p^\infty,N,p )= \amalg_{\lambda=0,\ldots,p-1} \cX(p^\infty,N) & \stackrel{q_1}{\longrightarrow} & \cX(p^\infty,N) \cr \pi_{\rm HT} \big\downarrow  & &  \amalg_{\lambda=0,\ldots,p-1} \pi_{\rm HT} \big\downarrow & & \pi_{\rm HT} \big\downarrow\cr 
\PP^1_{\Q_p} & \stackrel{t_2}{\longleftarrow} &   \amalg_{\lambda=0,\ldots,p-1}  \PP^1_{\Q_p} & \stackrel{t_2}{\longrightarrow} & \PP^1_{\Q_p}.\cr 

\end{matrix}$$In fact, the squares are commutative. This follows from  the  functoriality of ${\rm dlog}$ with respect to isogenies and by computing $u_\lambda^\vee$:  the map $u_\lambda$ sends $a\mapsto a'$ and $b\mapsto p b'-\lambda a'$ so that on the dual basis $u_\lambda^\vee$ sends $(a')^\vee\mapsto a^\vee-\lambda b^\vee$ and $(b')^\vee\mapsto p b^\vee$.  

\begin{remark} Let us observe that, with the notations above, if we denote by $U_p$ the correspondence on $\cX(p^\infty,N)$ given by $U_p:=q_{2}\circ q_1^{-1}$, if we denote by $\widetilde{U}$ the correspondence on $\PP^1$ defined by $\widetilde{U}:=t_{2}\circ t_1^{-1}$, then we have: $\pi_{\rm HT}\circ U_p=\widetilde{U}\circ \pi_{\rm HT}$. 

\end{remark}

We conclude this section with a lemma   on the dynamic of the operators $t_\lambda$. 

\begin{lemma}\label{lemma:dynamic} a) Let $\lambda=\lambda_0 + p \lambda _1 + \cdots + \lambda_n p^n$ with $\lambda_i\in\{0,\ldots, p-1\}$. Write $t_{\lambda}:=t_{\lambda_n} \circ \cdots \circ t_{\lambda_0}$. Then   $t_{\lambda}\bigl(U_{0,\lambda}^{(n+1)}\bigr)= U_0$, $t_{\lambda}\bigl(\PP^1\backslash U_{0,\lambda}^{(n+1)}\bigr)=U^{(1)}_\infty$. 

b) If $\mu:=\lambda_0+\lambda_1p+...+\lambda_{n-1}p^{n-1}$ with $\lambda_i$ as above for $0\le i\le n-1$ and $t_\mu:=t_{\lambda_{n-1}}\circ...\circ t_{\lambda_0}$, then we have
$t_\mu\bigl(\PP^1\backslash U_{0,\lambda}^{(n+1)}\bigr) \subset U_{\infty}\backslash U_0^{(1)}$ and
$t_\mu(U_\infty^{(1)})\subset U_\infty^{(n+1)}$.

\end{lemma}
\begin{proof}  

 It is enough to prove the statement for $\mathrm{Spa}(K,K^+)$-valued points for an affinoid field $(K,K^+)$. This is determined by a $K$-valued point of $\PP^1_{\Q_p}$, or equivalently a $K^+$-valued point $[\alpha;\beta]$ as $K^+$ is a valuation ring, whose normalized valuation is denoted $v$.  

a) We prove the statement for $\lambda=\lambda_0$ leaving the inductive process to the reader.
If $[\alpha;\beta]$ is a point of $U_\infty\backslash U_0$ then we can assume that $\alpha=1$ and that 
$\beta$ is in the maximal ideal of $K^+$ and $t_\lambda([1;\beta])=[1-\lambda \beta;  p \beta]$ defines a point of $U_\infty^{(1)}$.

If $\beta$ is a unit, we can assume that $\beta=1$ and then $t_\lambda([\alpha;1])=[\alpha-\lambda ;  p ]$. This is a point of  $U_0$ if and only if $\frac{\alpha-\lambda}{p}\in K^+$, i.e., if and only if  $[\alpha;1]$ defines  a point of $U_{0,\lambda}^{(1)}$. Else $\frac{p}{\alpha-\lambda}$ lies in the maximal ideal of $K^+$ and then $t_\lambda([\alpha;1])$ defines a point of $U_\infty\backslash U_0$. 

b) If $[\alpha, \beta]\in \PP^1\backslash U_0^{(n+1)}$, $\alpha, \beta\in K^+$, we may assume that one of $\alpha, \beta$ is $1$.

If $\beta=1$, we have $t_\mu([\alpha, 1])=[\alpha-\mu, p^n]$. Moreover $[\alpha, 1]\notin U_0^{(n+1)}$
implies either that $r:=v(\alpha-\mu)<n$ and in this case $[\alpha-\mu, p^n]\in  U_\infty\backslash U_0^{(1)}$, or that $v(\alpha-\mu)=n$ and $\alpha=\mu+p^n\gamma$, with $\gamma\in K^+$,
$\gamma\notin \F_p ({\rm mod }\  pK^+)$. Then $t_\mu([\alpha, \beta])=[\gamma, 1]\in U_{\infty}\backslash U_0^{(1)}$.

If now $\alpha=1$ and $\beta$ is in the maximal ideal of $K^+$, we have $t_\mu([1, \beta])=[1-\mu\beta, p^n\beta]$.
Since  $1-\mu\beta\in (K^+)^\ast$, we have $t_\mu([1, \beta])=[1, p^n\beta/(1-\mu\beta)]\in U_\infty^{(n)}\subset U_\infty\backslash U_0^{(1)}.$

The other statement is clear.

\end{proof} 

\begin{remark}\label{rmk:uptlambda}
Let us notice that if $x\in U_\infty\backslash U_0^{(1)}$ and $\nu\in \{ 0,1,...,p-1\}$ then 
$t_\nu(x)\in U_\infty^{(1)}$. This observation and Lemma \ref{lemma:dynamic} b) imply the part
$t_\lambda(\PP^1\backslash U_0^{(n+1)})\subset U_\infty^{(1)}$ of Lemma \ref{lemma:dynamic} a).

\end{remark}

\subsection{\'Etale sheaves.}

Let $H\subset \GL_2(\Z_p)$ be a finite index subgroup.  In this section we recall the tensor functor from the category of profinite $H$-representations to the category of sheaves on the pro-Kummer \'etale site of the modular curve $X(H,N)$ defined by $H$, or equivalently of the associated adic space $\cX(H,N)$. We work with the latter.

We fix log gemetric points $\zeta_i$, one  for every connected component $Z_i$ of $\cX(H,N)$. Due to \cite[Prop.~5.1.12]{Diao} the sites $Z_{i,\fket}$ are Galois categories with underlying profinite group, the Kummer \'etale fundamental group $\pi_1^\ket(Z_i,\zeta_i)$.  In particular, the pro-Kummer finite \'etale cover $\bigl(\cX(p^r,N)\bigr)_{r}$ of $\cX(H,N)$ for $r$ big enough, restricted to each $Z_i$ defines a homomorphism  $$\pi_1(Z_i,\zeta_i)\to \displaystyle{\lim_{\infty \leftarrow r}} \mathrm{Aut}\bigl(\cX(p^r,N)/\cX(H,N) \bigr)=H.$$Given a finite representation $L_n$ of $H$, we view it as a representation of  $\pi_1(Z_i,\zeta_i)$, for every $i$, and hence as a local system on each $Z_{i,\ket}$  and, hence, a local system  $\mathbb{L}_n$ on  $\cX(H,N)_\ket$. In fact,  $\mathbb{L}_n$ is a sheaf on $\cX(H,N)_\ket$ such that there exists a finite Kummer \'etale cover $\cX(p^r,N) \to \cX(H,N)$, for $r\gg 0$, on which the sheaf $\mathbb{L}_n$ is constant. 

Given a profinite representation $L$ of $H$, i.e., an inverse limit $L=\displaystyle{\lim_{\infty \leftarrow n}} L_n$ of finite representations $L_n$ for $n\in \N$, we let $\mathbb{L}$ be the inverse limit $\mathbb{L}=\displaystyle{\lim_{\infty \leftarrow n}} \mathbb{L}_n$. It is a sheaf on $\cX(H,N)_\ket$. Notice that using the scheme $X(H,N)$ one gets, as mentioned before, a sheaf on $X(H,N)_\ket$, that we sill denote $\mathbb{L}$. We  have the following  GAGA type of results: 

\begin{theorem}\label{thm:GAGA}  For every $i\in \N$, the maps  $$\mathrm{H}^i\bigl(X(H,N)_\proket,\mathbb{L}  \bigr)\longrightarrow \mathrm{H}^i\bigl(\cX(H,N)_\proket,\mathbb{L}  \bigr)$$are isomorphisms. Analogously the natural map$$\mathrm{H}^i\bigl(\cX(H,N)_\proket,\mathbb{L}  \bigr)\widehat{\otimes} \cO_{\C_p} \longrightarrow   \mathrm{H}^i\bigl(\cX(H,N)_\proket,\mathbb{L}\widehat{\otimes} \widehat{\cO}_{\cX(H,N)}^+\bigr)  $$is an almost isomorphism. 

\end{theorem}
\begin{proof}
The result for each $\mathbb{L}_n$ follows from the discussion in \S \ref{sec:comparecoeff} and from Proposition \ref{prop:GAGA}. 

Consider the natural map $\displaystyle{\lim_\leftarrow}\colon \mathrm{Sh}^\N(Z) \to \mathrm{Sh}(Z)$ from invese systems of sheaves on $Z=X(H,N)_\proket$, and $Z=\cX(H,N)_\proket$ respectively. Using the existence of bases of log affinoid perfectoid opens with the properties recalled in \S \ref{sec:logaffinoid}, it follows from  \cite[lemma 3.18]{ScholzeHodge} that we have $\mathrm{R}^i \displaystyle{\lim_\leftarrow} (\mathbb{L})=0$, both in the algebraic and in the adic setting, and $\mathrm{R}^i \displaystyle{\lim_\leftarrow} \bigl( \mathbb{L}\widehat{\otimes} \widehat{\cO}_{\cX(H,N)}^+\bigr)=0$ for $i\geq 1$. Hence, $\mathrm{H}^j\bigl(Z,\mathbb{L}\bigr)$ and $\mathrm{H}^j\bigl(\cX(H,N)_\proket,\mathbb{L}\widehat{\otimes} \widehat{\cO}_{\cX(H,N)}^+\bigr)$  coincide with the derived functors  $\mathrm{H}^j\bigl(Z,(\mathbb{L}_n)_{n\in \N}\bigr)$, resp.~$ \mathrm{H}^j\bigl(\cX(H,N)_\proket, (\mathbb{L}_n \otimes \cO_{\cX(H,N)_\proket}^+)_{n\in \N}\bigr) $ of  $\displaystyle{\lim_\leftarrow} \mathrm{H}^0(Z,\_) $ introduced by \cite{Ja} on the inverse system $\mathrm{Sh}^\N(Z)$. Due to \cite[Prop. 1.6]{Ja}, these cohomology groups sits in exact sequences
$$ 0 \lra \displaystyle{\lim_\leftarrow}^{(1)} \mathrm{H}^{j-1}\bigl(Z,\mathbb{L}_n\bigr) \lra \mathrm{H}^j\bigl(Z,(\mathbb{L}_n)_{n\in \N}\bigr)  \lra \displaystyle{\lim_\leftarrow}\mathrm{H}^{j}\bigl(Z,\mathbb{L}_n\bigr)  \lra 0$$and similalry for the inverse system $(\mathbb{L}_n \otimes \cO_{\cX(H,N)_\proket}^+)_{n\in \N}$. The maps in the Theorem are compatible with these sequences and the claims follow from the finite case, that is for the sheaves $\mathbb{L}_n$.

\end{proof}

\begin{remark} The sheaf $\mathbb{L}$  has the property that its pull-back to the perfectoid space $\cX(p^\infty,N)$ is constant and coincides, together with the its $H$-action, with $\pi_{\rm HT}^{-1}\bigl(L)$ where we view $L$ as a constant, $H$-equivariant sheaf on $\mathbb{P}^1$ (compare with \cite[\S 2.3]{CaraianiScholze} for the case of $p$-adic automorphic \'etale sheaves).

\end{remark}

\begin{example} Consider on $\cX(p^\infty,N) $  trivilizations $T_p E = \Z_p a \oplus \Z_p b$.
The group $\GL_2(\Z_p)$ acts on the left on the family of  trivilizations: given such a basis $\mathcal{A}:=\{a,b\}$ as above and a matrix $M\in \GL_2(\Z_p)$ we get a new basis $\mathcal{A}':=(a',b'):=(a ,b)  M$.  If we think of a trivialization as an isomorphism $\psi_{\mathcal{A}}\colon T_p E\cong \Z_p^2$ then $\psi_{\mathcal{A}'}$ is $\psi_{\mathcal{A}}$ times left multiplication by
$M$. Thus $T_p(E)$ corresponds to the standard representation $T=\Z_p\oplus \Z_p$ of $\GL_2(\Z_p)$.

This action of $M$ induces a map on dual basis ${^t}(a^\vee,b^\vee)=  M{^t}(a^{',\vee},b^{',\vee})$. Then the trivializations $\psi_{\mathcal{A}^\vee}\colon T_p E^\vee\cong \Z_p^2$ 
and $\psi_{\mathcal{A}^{',\vee}}\colon T_p E^\vee\cong \Z_p^2$ induced by the dual bases are such that $ \psi_{\mathcal{A}^\vee}$ is  $\psi_{\mathcal{A}^{',\vee}}$ times the
right multiplication by $M$. To make the map $\pi_{\rm HT}$ equivariant for the $\GL_2(\Z_p)$-action we take on $\PP^1_{\Q_p}$ the standard action. If $\pi_{\rm HT}  (\psi_{\mathcal{A}^\vee})=[\alpha ;\beta]$ and  $\pi_{\rm HT}  (\psi_{\mathcal{A}^{',\vee}})=[\alpha' ;\beta']$ then ${^t}[\alpha;\beta]=M  \, {^t}[\alpha';\beta']$. 

\end{example}

Consider an $s$-analytic character $k\colon \Z_p^\times\lra B^\ast$ as in Definition \ref{def:weights}. Consider the $\Q_p$-module $D_k(T_0^\vee)[n]$, for $n\geq s$, with action of the Iwahori subroup $\Iw_1$, defined in \S \ref{def:anal}. As $D_k(T_0^\vee)[n]=\bigl(D_k^o(T_0)[n]\bigr)[1/p]$ and $D_k^o(T_0^\vee)[n]$ admits a $\Iw_1$-equivariant filtration with finite graded pieces, we get an associated sheaf $\bD_k(T_0^\vee)[n]$ on the pro-Kummer \'etale site of $\cX_0(p,N)$. Then:

\begin{proposition}\label{prop:GAGAD}  For every $i\in \N$, we have isomorphisms  $$\mathrm{H}^i\bigl(\Gamma_0(p)\cap \Gamma(N),D_k(T_0^\vee)[n] \bigr)\cong \mathrm{H}^i\bigl(\cX_0(p,N)_\proket,\bD_k(T_0^\vee)[n]\widehat{\otimes} \widehat{\cO}_{\cX_0(p,N)}\bigr)  .$$

\end{proposition}

\begin{proof} The first group is identified with  $\mathrm{H}^i\bigl(X_0(p,N)_\ket,\bD_k(T_0^\vee)[n]  \bigr)$  arguing as in \cite[Thm. 3.15]{EichlerShimura} using the filtration $\mathrm{Fil}^\bullet \bD_k(T_0^\vee)[n]$ discussed in \S \ref{sec:anal}. As $\widehat{\cO}_{\cX_0(p,N)}=\widehat{\cO}_{\cX_0(p,N)}^+[1/p]$ and cohomolopgy commutes with direct limits, the conclusion follows from Theorem \ref{thm:GAGA} and inverting $p$.

\end{proof}

We have actions of  Hecke operators on $\mathrm{H}^i\bigl(\cX_0(p,N)_\proket,\bD_k(T_0^\vee)[n]\widehat{\otimes} \widehat{\cO}_{\cX_0(p,N)}\bigr)$ as follows. Let $\ell$ be a prime integer not dividing $N$ and let $q_1,q_2\colon X_0(p,N,\ell)\to X_0(p,N)$ be as in \S\ref{sec:Up}. By \S \ref{sec:HeckeD} we have a map $$\mathcal{U}\colon q_1^\ast\bigl(\bD_k^o(T_0^\vee)[n]\widehat{\otimes} \widehat{\cO}_{\cX_0(p,N)}^+\bigr) \to q_2^\ast\bigl(\bD_k^o(T_0^\vee)[n+1]\bigr)\widehat{\otimes} \widehat{\cO}_{\cX_0(p,N,\ell)}^+.$$Inverting $p$, taking $q_{2,\ast}$ and using the trace map ${\rm Tr}\colon q_{2,\ast} q_2^\ast \to {\rm Id}$ of Lemma \ref{lemma:trace} we get a map 
$$q_{2,\ast} q_1^\ast\bigl(\bD_k(T_0^\vee)[n]\widehat{\otimes} \widehat{\cO}_{\cX_0(p,N)}\bigr) \to \bD_k(T_0^\vee)[n+1]\widehat{\otimes} \widehat{\cO}_{\cX_0(p,N)}.$$We have a   restriction map $D_k(T_0^\vee)[n+1]\to D_k(T_0^\vee)[n] $ which is $\Iw_1$-equivariant and defines a map $\bD_k(T_0^\vee)[n+1] \to \bD_k(T_0^\vee)[n] $. We finally get a map 
$$q_{2,\ast} q_1^\ast\bigl(\bD_k(T_0^\vee)[n]\widehat{\otimes} \widehat{\cO}_{\cX_0(p,N)}\bigr) \to \bD_k(T_0^\vee)[n]\widehat{\otimes} \widehat{\cO}_{\cX_0(p,N)}.$$Taking cohomology groups over $\cX_0(p,N)_\proket$ and using the map $$\mathrm{H}^i\bigl(\cX_0(p,N)_\proket,\bD_k(T_0^\vee)[n]\widehat{\otimes} \widehat{\cO}_{\cX_0(p,N)}\bigr) \to \mathrm{H}^i\bigl(\cX_0(p,N,\ell)_\proket,q_1^\ast\bigl(\bD_k(T_0^\vee)[n]\widehat{\otimes} \widehat{\cO}_{\cX_0(p,N)}\bigr)  \bigr) $$we get Hecke operators 

\begin{equation}\label{eq:HeckeonHi} T_\ell, U_p, U_p^{\rm naive}\colon \mathrm{H}^i\bigl(\cX_0(p,N)_\proket,\bD_k(T_0^\vee)[n]\widehat{\otimes} \widehat{\cO}_{\cX_0(p,N)}\bigr)  \to \mathrm{H}^i\bigl(\cX_0(p,N)_\proket,\bD_k(T_0^\vee)[n]\widehat{\otimes} \widehat{\cO}_{\cX_0(p,N)}\bigr) . \end{equation}

(Here we introduce  the un-normalized operator $U_p^{\rm naive}$ which is $p $ times the standard operator $U_p$ as it preserves integral structures, a fact that will be of  crucial importance later on.)

\subsection{A comparison result on $\cX_0(p^n,N)_\infty^{(m)}$.}  Consider an $s$-analytic weight $k\colon \Z_p^\ast \to B^\ast$ as in definition \ref{def:weights}, integers $ n\geq s,m$, where we recall 
that $s$ is the degree of analycity of $k$  and the associated sheaf $$\mathfrak{D}_{k,\infty}^{o,(m)}[n]:=\bD_k^o(T_0^\vee)[n]\vert_{\cX_0(p^n,N)_\infty^{(m)}} \widehat{\otimes} \widehat{\cO}_{\cX_0(p^n,N)_\infty^{(m)}}^+.$$ The aim of this section is to prove that for $m$ large enough it admits a decreasing filtration $\Fil^h \mathfrak{D}_{k,\infty}^{o,(m)}[n]$ for $h\geq -1$ with the following properties:

\begin{proposition}\label{prop:Wkinfty} We have:

\begin{itemize}

\item[i.] For $n\ge m'\geq m$ we have $\Fil^\bullet  \mathfrak{D}_{k,\infty}^{o,(m)}[n]  \widehat{\otimes}_{ \widehat{\cO}_{\cX_0(p^n,N)_\infty^{(m)}}^+} \widehat{\cO}_{\cX_0(p^n,N)_\infty^{(m')}}^+\cong \Fil^\bullet  \mathfrak{D}_{k,\infty}^{o,(m')}[n] $;

\item[ii.] We have a surjective map $\mathrm{Gr}^{-1} \mathfrak{D}_{k,\infty}^{o,(m)}[n]=\mathfrak{D}_{k,\infty}^{o,(m)}[n]/\Fil^0 \mathfrak{D}_{k,\infty}^{o,(m)}[n]\longrightarrow \omega^{k} \otimes_{\cO_{\cX_0(p^n, N)_\infty^{(m)}}^+} \widehat{\cO}_{\cX_0(p^n, N)_\infty^{(m)}}^+ $;

\item[iii.] We have an isomorphism  $$\mathrm{H}^1\bigl(\cX_0(p^n, N)_{\infty,\proket}^{(m)},  \omega^{k} \otimes_{\cO_{\cX_0(p^s,N)_\infty^{(m)}}^+} \widehat{\cO}_{\cX_0(p^n,
N)_\infty^{(m)}} \bigr)\cong \mathrm{H}^0\bigl(\cX_0(p^n, N)_{\infty}^{(m)},\omega^{k+2}\bigr)[1/p];$$

\item[iv.] The map $$\mathrm{H}^1\bigl(\cX_0(p^n, N)_{\infty,\proket}^{(m)}, \mathfrak{D}_{k,\infty}^{o,(m)}[n] /\Fil^h \mathfrak{D}_{k,\infty}^{o,(m)}[n] \bigr) \to \mathrm{H}^1\bigl(\cX_0(p^n, N)_{\infty,\proket}^{(m)}, \omega^{k} \otimes_{\cO_{\cX_0(p^s,N)_\infty^{(m)}}^+} \widehat{\cO}_{\cX_0(p^n, N)_\infty^{(m)}}^+\bigr),$$induced by the projection onto $\mathrm{Gr}^{-1} \bD_{k,\infty}^{o,(m)}$ and (ii),  has kernel and cokernel annihilated by a power $p^{a h}$ of $p$ (with $a$ depending on $n$);

\item[v] We have $\mathrm{H}^i\bigl(\cX_0(p^n, N)_{\infty,\proket}^{(m)}, \mathfrak{D}_{k,\infty}^{o,(m)}[n] /\Fil^h \mathfrak{D}_{k,\infty}^{o,(m)}[n] \bigr)$ equal to $\mathrm{H}^0\bigl(\cX_0(p^n, N)_{\infty}^{(m)},\mathrm{Gr}^h \mathfrak{D}_{k,\infty}^{o,(m)}[n] \bigr) $ for $i=0$ and it is $0$ for $i\geq 2$.

\end{itemize}

\end{proposition}

\begin{proof}   Recall the surjective map 
$${\rm dlog}\colon T_p(E)^\vee \otimes_{\Z_p} \widehat{\cO}_{\cX_0(p^n,N)_\infty^{(m)}}^+  \longrightarrow  \omega_E^{\rm mod}\widehat{\otimes}  \widehat{\cO}_{\cX_0(p^n,N)_\infty^{(m)}}^+$$ from (\ref{eq:dlog}) and  Proposition \ref{proposition:existcanonical}. It is defined for every $m$ large enough and $\omega_E^{\rm mod} $ is an invertible $\cO_{\cX_0(p^n,N)_\infty^{(m)}}^+$-module.  
Consider the natural projection $j_n\colon \cX(p^n,N)_\infty^{(m)} \to \cX_0(p^n,N)_\infty^{(m)}$; then $j_n$ is Kummer \'etale and Galois with group $\Delta_n$ the subgroup of $\GL_2(\Z/p^n\Z)$ of upper triangular matrices. In particular, to define a sheaf on $\cX_0(p^n,N)_{\infty,\proket}^{(m)}$ is equivalent to define a sheaf on $\cX(p^n,N)_{\infty,\proket}^{(m)}$ with an action of $\Delta_n$ compatible with the action on $\cX(p^n,N)_\infty^{(m)}$. In order to define $\Fil^h \mathfrak{D}_{k,\infty}^{o,(m)}[n]$ we  define a $\Delta_n$-equivariant  filtration on $j_n^\ast \bigl( \mathfrak{D}_{k,\infty}^{o,(m)}[n]\bigr)$.

Over $ \cX(p^n,N)_\infty^{(m)}$  we  have a trivialization $$T_p(E)/p^n T_p(E)=(\Z/p^n\Z) a \oplus (\Z/p^n\Z)b$$such that the level subgroup $H_n$ admits $a$ as generator. In particular,  ${\rm dlog}(a^\vee)$ is a basis for $\omega_E^{\rm mod}/ p^n \omega_E^{\rm mod}$ as $\cO_{\cX(p^n,N)_\infty^{(m)}}^+$-module. Dualizing we get an injective map 
$$({\rm dlog})^\vee\colon \omega_E^{\rm mod,-1} \otimes_{\cO_{\cX_0(p,N)_\infty^{(m)}}^+} \widehat{\cO}_{\cX(p^n,N)_\infty^{(m)}}^+ \lra T_p(E) \otimes_{\Z_p} \widehat{\cO}_{\cX(p^n,N)_\infty^{(m)}}^+ $$with quotient $\mathcal{Q}$ isomorphic to $\omega_E^{\rm mod}\otimes_{\cO_{\cX_0(p,N)_\infty^{(m)}}^+} \widehat{\cO}_{\cX(p^n,N)_\infty^{(m)}}^+$ and such that modulo $p^n$ the section $a$ generates $\omega_E^{\rm mod,-1}/p^n \omega_E^{\rm mod,-1}$ as $\cO_{\cX(p^n,N)_\infty^{(m)}}^+$-module. Then $({\rm dlog})^\vee$ induces, for every $\lambda\in \Z/p^n\Z$, an affine map $$\rho_\lambda\colon \V_0\bigl(T_p(E) , a,b-\lambda a\bigr) \times_{\Z_p} \cX(p^n,N)_\infty^{(m)}  \lra \V_0\bigl(\omega_E^{\rm mod},a\bigr) \times_{\cX_0(p,N)_\infty^{(m)}} \cX(p^n,N)_\infty^{(m)}
 $$on the pro-Kummer \'etale site of $\cX(p^n,N)_\infty^{(m)}$. Notice that $\V_0\bigl(T_p(E) , a,b-\lambda a\bigr) \times_{\Z_p} \cX(p^n,N)_\infty^{(m)}\cong \V_0\bigl(T_p(E)\otimes_{\Z_p} \widehat{\cO}_{\cX(p^n,N)_\infty^{(m)}}^+,a,b-\lambda b\bigr)$ is a principal homogeneous space under the formal vector group $\V'\bigl(\mathcal{Q}\bigr) \subset \V\bigl(\mathcal{Q}\bigr)$ classifying sections of $\mathcal{Q}^\vee$ which are zero modulo $p^n$. We have the invertible $\cO_{\cX_0(p,N)_\infty^{(m)}}^+\widehat{\otimes} B$-module
$\omega^{-k} := \WW_k(\omega_E^{\rm mod,-1}, a\bigr)$. Set $$\mathfrak{W}_{k,\infty,\lambda}^{(m)}:= W_k\bigl(T_p(E),a,b-\lambda a\bigr) \otimes_{\Z_p} \widehat{\cO}_{\cX(p^n,N)_\infty^{(m)}}^+.$$Applying the formalism of VBMS we get the map of sheaves on the pro-Kummer \'etale site of $\cX(p^n,N)_\infty^{(m)}$ $$\rho_\lambda^k\colon \omega^{-k} \otimes_{\cO_{\cX_0(p,N)_\infty^{(m)}}^+} \widehat{\cO}_{\cX(p^n,N)_\infty^{(m)}}^+ \lra \mathfrak{W}_{k,\infty,\lambda}^{(m)} .$$Using that $\V_0\bigl(T_p(E)\otimes_{\Z_p} \widehat{\cO}_{\cX(p^n,N)_\infty^{(m)}}^+,a,b +\lambda a\bigr)$ is a principal homogeneous space under  $\V'\bigl(\mathcal{Q}\bigr)$ we obtain a  $\V'\bigl(\mathcal{Q}\bigr)$-stable increasing filtration  $\Fil_h \mathfrak{W}_{k,\infty,\lambda}^{(m)}$ for $h\geq 0$  with 

\begin{equation}\label{eq:gr} \Fil_0 \mathfrak{W}_{k,\infty,\lambda}^{(m)}=\omega^{-k} \otimes_{\cO_{\cX_0(p,N)_\infty^{(m)}}^+} \widehat{\cO}_{\cX(p^n,N)_\infty^{(m)}}^+,\qquad \mathrm{Gr}_h \mathfrak{W}_{k,\infty}^{(m)}\cong \omega^{-k+2h} \otimes_{\cO_{\cX_0(p,N)_\infty^{(m)}}^+} \widehat{\cO}_{\cX(p^n,N)_\infty^{(m)}}^+.
\end{equation}

See \cite{andreatta_iovita_triple} and \cite[Prop. 5.2]{half}. Taking $\widehat{\cO}_{\cX(p^n,N)_\infty^{(m)}}^+$-duals we get  a sheaf and a decreasing filtration $$\mathfrak{W}_{k,\infty,\lambda}^{(m),\vee}=W_k\bigl(T_p(E),a,b-\lambda b\bigr)^\vee \otimes_{\Z_p} \widehat{\cO}_{\cX(p^n,N)_\infty^{(m}}^+, \qquad \Fil^h \mathfrak{W}_{k,\infty,\lambda}^{(m),\vee}, h\geq -1 $$on the pro-Kummer \'etale site of $\cX(p^n,N)_\infty^{(m)}$. Here $\Fil^h \mathfrak{W}_{k,\infty,\lambda}^{(m),\vee}$ consists of those sections of $\mathfrak{W}_{k,\infty,\lambda}^{(m),\vee}$ which are zero on $\Fil_h \mathfrak{W}_{k,\infty,\lambda}^{(m)}$ (where we set  $\Fil_{-1}\mathfrak{W}_{k,\infty,\lambda}^{(m),\vee}=0$ so that $\Fil^{-1} \mathfrak{W}_{k,\infty,\lambda}^{(m),\vee})= \mathfrak{W}_{k,\infty,\lambda}^{(m),\vee}$). Then
  \begin{equation}\label{eq:Grh} \mathrm{Gr}^h \mathfrak{W}_{k,\infty,\lambda}^{(m),\vee}\cong \omega^{k-2h-2} \otimes_{\cO_{\cX_0(p,N)_\infty^{(m)}}^+} \widehat{\cO}_{\cX(p^n,N)_\infty^{(m)}}^+.\end{equation}

Due to Proposition \ref{prop:AkWkveee} we have a $\Delta_n$-equivariant isomorphism 
 $$j_n^\ast \bigl( \mathfrak{D}_{k,\infty}^{o,(m)}[n]\bigr)\cong  \oplus_{\lambda\in\Z/p^n\Z} \mathfrak{W}_{k,\infty,\lambda}^{(m),\vee}$$and we set $\Fil^h j_n^\ast \bigl( \mathfrak{D}_{k,\infty}^{o,(m)}[n]\bigr)$ to be the filtration corresponding to  $\oplus_{\lambda\in\Z/p^n\Z} \Fil^h\mathfrak{W}_{k,\infty,\lambda}^{(m),\vee}  $. The map $\amalg_ \lambda \rho_\lambda$ defines a  map $\rho_\lambda^k\colon \omega^{-k} \otimes_{\cO_{\cX_0(p,N)_\infty^{(m)}}^+} \widehat{\cO}_{\cX(p^n,N)_\infty^{(m)}}^+ \lra \oplus_\lambda \mathfrak{W}_{k,\infty,\lambda}^{(m)} $ and, hence, a map $$\nu\colon \mathrm{Gr}^{-1} j_n^\ast\bigl(\mathfrak{D}_{k,\infty}^{o,(m)}[n]\bigr) \to \omega^{k} \otimes_{\cO_{\cX_0(p,N)_\infty^{(m)}}^+} \widehat{\cO}_{\cX(p^n,N)_\infty^{(m)}}^+.$$By construction (i) and (ii) hold over $\cX(p^n,N)_\infty^{(m)}$ and to prove those claims we need to show that the given filtration and the map $\nu$ are $\Delta_n$-equivariant. Also given a sheaf $\mathcal{F}$ on $\cX_0(p,N)_\infty^{(m)}$ we have a spectral sequence $$\mathrm{H}^i\bigl(\Delta_n \mathrm{H}^j\bigl(\cX(p^n,N)_{\infty,\proket}^{(m)},j_n^\ast(\mathcal{F})\bigr) \Rightarrow \mathrm{H}^{i+j}\bigl(\cX_0(p,N)_{\infty,\proket}^{(m)},\mathcal{F}\bigr).  $$Since the cohomology groups $\mathrm{H}^j\bigl(\Delta_n,\mathcal{G}\bigr)$ are annihilated by the order   of $\Delta_n$ for $j\geq 1$, to conclude (iv)--(v) it suffices to prove those claims for $\cX(p^n,N)_{\infty}$.

The rest of the proof is a computation using the log affinoid perfectoid  cover by opens $U$ of the adic space $\cX_0(p^n,N)_{\infty,\proket}^{(m)}$ defined by trivializing the full Tate module $\prod_\ell T_\ell(E)$. It is Galois over $\cX_0(p^n,N)_{\infty,\proket}^{(m)}$ with group $G_U$. Let $\widehat{U}:=\mathrm{Spa}(R,R^+)$ be the associated affinoid perfectoid space as in \S \ref{sec:logaffinoid}.  Write $T_p(E)^\vee(U)\otimes R^+=e_0 R^++
e_1R^+$ with $e_0$ mapping to a generator of $\omega_E^{\rm mod}$ and $e_1$  in the kernel of $d\log$ generating $\omega_E^{\rm mod,-1}$. Recall from Proposition \ref{prop:AkWkveee} that $\mathfrak{W}_{k,\infty,\lambda}^{(m),\vee}(U)$ is the dual of $\mathfrak{W}_{k,\infty,\lambda}^{(m)}(U)\cong R^+\otimes B\langle \frac{W_\lambda}{1+p^n Z} \rangle \cdot k(1+p^n Z)$ with increasing filtration defined by $\Fil^h = \oplus_{i=0}^h R^+\otimes B \bigl(\frac{W_\lambda}{1+p^n Z}\bigr)^i  \cdot \cdot k(1+p^n Z) $.

For every $\sigma\in G_U$ we then have $\sigma(e_1)=e_1$ and $\sigma(e_0)=e_0+\xi(\sigma) e_1$. Then, $\sigma(W_\lambda)=W_\lambda +\frac{\xi(\sigma)}{p^n}  (1+p^n Z)  $ and $\sigma(Z)= Z$. If $\xi(\sigma)=\alpha + p^n \beta$ then $\sigma(W_\lambda)=W_{\lambda+\alpha} + \beta (1+p^n Z) $. Thus the incresing filtration on $\oplus_{\lambda\in \Z/p^n\Z}\mathfrak{W}_{k,\infty,\lambda}^{(m)}$ and the diagonal embedding of $R^+\otimes B\to \oplus_{\lambda\in \Z/p^n\Z}\Fil^0 \mathfrak{W}_{k,\infty,\lambda}^{(m)}$ are both stable for the action of $G_U$. This concludes the proof of (i) and (ii).

We pass to Claims (iii)-(v).   As $\omega^{k+2}$ is a locally free $\cO_{\cX_0(p,N)_\infty^{(m)}}^+\widehat{\otimes}B$-module we are left to show that $$\mathrm{H}^1\bigl(\cX_0(p,N)_{\infty,\proket}^{(m)},  \widehat{\cO}_{\cX_0(p,N)_\infty^{(m)}} \bigr)\cong \mathrm{H}^0\bigl(\cX_0(p,N)_{\infty}^{(m)},\omega^{2}\bigr)[1/p].$$So to prove Claim (iii) we are left to show that that the map  (\ref{eq:H0H1}) is an isomorphism. The \'etale cohomology of the structure sheaf on $U$ is trivial as recalled in \S\ref{sec:logaffinoid}.  Thanks to (\ref{eq:Grh}) the sheaves $ \mathfrak{D}_{k,\infty}^{o,(m)}[n] /\Fil^h \mathfrak{D}_{k,\infty}^{o,(m)}[n] \vert_U$ are extensions of the structure sheaf $\widehat{\cO}_U^+$ so that also their cohomology over $U$ is trivial.  Then the cohomology groups in (\ref{eq:H0H1}) and those of $ \mathfrak{D}_{k,\infty}^{o,(m)}[n] /\Fil^h \mathfrak{D}_{k,\infty}^{o,(m)}[n]$ appearing in (iv) and (v) coincide with the continuous cohomology of $G_U$ of the sections of the relevant sheaves over $U$.  This reduces the proof of claims (iii), (iv) and (v)  to a Galois cohomology computation for which we refer to \cite[Thm. 5.4]{half}.

\end{proof}

We have Hecke operators acting on the cohomology of $\mathfrak{D}_{k,\infty}^{o,(m)}[n] $; see (\ref{eq:HeckeonHi}). We assume the hypothesis in the proof of Proposition \ref{prop:Wkinfty}.

\begin{proposition}\label{prop:HeckeWkinfty} The  operator $U_p^{\rm naive}=p U_p$ is defined on $\mathrm{H}^i\bigl(\cX_0(p,N)_{\infty,\proket}^{(m)}, \Fil^h \mathfrak{D}_{k,\infty}^{o,(m)}[n] \bigr)$ for every $h\geq -1$. There exists an operator  $U_{p,h}$ $$U_{p,h}\colon \mathrm{H}^i\bigl(\cX_0(p,N)_{\infty,\proket}^{(m)}, \Fil^h \mathfrak{D}_{k,\infty}^{o,(m)}[n] \bigr) \to  \mathrm{H}^i\bigl(\cX_0(p,N)_{\infty,\proket}^{(m)}, \Fil^h \mathfrak{D}_{k,\infty}^{o,(m)} [n]\bigr) $$ such that  on $\mathrm{H}^i\bigl(\cX_0(p,N)_{\infty,\proket}^{(m)}, \Fil^h \mathfrak{D}_{k,\infty}^{o,(m)}[n] \bigr)$ we have   $U_p^{\rm naive}=p^{h+1} U_{p,h}$  for $i=0$, $1$. 

Moroever, for every positive integer $h$ the cohomology group  $\mathrm{H}^1\bigl(\cX_1(p,N)_{\infty,\proket}^{(n)}, \mathfrak{D}_{k,\infty}^{o,(m)}[n] \bigr)[1/p] $ admits a slope $\leq h$-decomposition with respect to the $U_p$-operator and we have an isomorphism of Hecke modules:

$$\Psi\colon \mathrm{H}^1\bigl(\cX_0(p,N)_{\infty,\proket}^{(m)}, \mathfrak{D}_{k,\infty}^{o,(m)}[n]  \bigr)[1/p]^{(h)} \cong \mathrm{H}^0\bigl(\cX_0(p,N)_{\infty}^{(m)},\omega^{k+2}\bigr)[1/p]^{(h)}.$$

 \end{proposition}

\begin{proof} Recall the construction of the $U_p$-operator in (\ref{eq:HeckeonHi});  all steps are defined integrally on $\mathfrak{D}_{k,\infty}^{o,(m)} [n]$ except for the trace ${\rm Tr}\colon q_{2,\ast} q_2^\ast\sim {\rm Id}$ and all steps are defined for $\Fil^h \mathfrak{D}_{k,\infty}^{o,(m)} [n]$ except possibly for 
$$\mathcal{U}\colon q_1^\ast\bigl(\bD_k^o(T_0^\vee)[n]\widehat{\otimes} \widehat{\cO}_{\cX_0(p,N)}^+\bigr) \to q_2^\ast\bigl(\bD_k^o(T_0^\vee)[n+1]\bigr)\widehat{\otimes} \widehat{\cO}_{\cX_0(p,N,p)}^+.$$We claim that  $\mathcal{U}$ restricts to a map on $q_1^\ast\bigl(\Fil^h \bD_k^o(T_0^\vee)[n]\widehat{\otimes} \widehat{\cO}_{\cX_0(p,N)}^+\bigr) \to q_2^\ast\bigl(\Fil^h \bD_k^o(T_0^\vee)[n+1]\widehat{\otimes} \widehat{\cO}_{\cX_0(p,N,p)}^+\bigr)$ which can be written as $p^{h+1}$ times an operator $\mathcal{U}'$. Both statements can be checked upon passing to sections over a log affinoid perfectoid.   These statements are then proven in \cite[Thm. 5.5]{half}. 

The statement on the existence of slope decomposition and the displayed isomorphism are proven as in \cite[Thm. 5.1]{half} using Proposition \ref{prop:Wkinfty}.

\end{proof}

\subsection{A comparison result on $\cX_0(p^n,N)_0^{(m)}$.} As in the previous section we fix  an $s$-analytic weight $k\colon \Z_p^\ast \to B^\ast$ as in definition \ref{def:weights} and we write $u_k\in B[1/p]$ for the element such that $k(a)=\exp u_k \log a$ for $a\in 1+p^s \Z_p$. Note that $p^{s} u_k\in B$. Fix an integer $n\geq s$  and define the sheaf $\mathfrak{D}_{k,0}^{o,(m)}[n]:=\bD_k^o(T_0^\vee)[n]\vert_{\cX_0(p^n,N)_0^{(m)}} \widehat{\otimes} \widehat{\cO}_{\cX_0(p^n,N)_0^{(m)}}^+$. In this case we have the following:

\begin{proposition}\label{prop:Wk0} For $m$ large enough there exists an increasing  filtration $\Fil_h \mathfrak{D}_{k,0}^{o,(m)}[n]$ for $h\geq 0$ with the following properties:

\begin{itemize}

\item[i.] For $m'\geq m$ we have $\Fil_\bullet  \mathfrak{D}_{k,0}^{o,(m)}[n]  \widehat{\otimes}_{ \widehat{\cO}_{\cX_0(p^n,N)_0^{(m)}}^+} \widehat{\cO}_{\cX_0(p^n,N)_0^{(m')}}^+\cong \Fil_\bullet  \mathfrak{D}_{k,0}^{o,(m')}[n] $;

\item[ii.] For every $h$ the image of the map $$\mathrm{H}^1\bigl(\cX_0(p^n,N)_{0,\proket}^{(m)},\Fil_h \mathfrak{D}_{k,\infty}^{o,(m)}[n] \bigr) \to \mathrm{H}^1\bigl(\cX_0(p^n,N)_{0,\proket}^{(m)}, \mathfrak{D}_{k,\infty}^{o,(m)}[n] \bigr) $$is annihilated by the product $p^{(n+c) h} \left(\begin{matrix} u_k \cr h \end{matrix} \right) $ (with $c$ depending on $n$). 

\end{itemize}

\end{proposition}
\begin{proof}

We follow closely the proof of Proposition \ref{prop:Wkinfty}. We take $m$ large enough so that $E$ admits a canoncial subgroup $C_n$ of level $p^n$ over $\cX_0(p^n,N)_{0}^{(m)}$. Note that the canonical subgroup $C_n$ and the level subgroup $H_n$ are distinct. 

The natural projection $j_n\colon \cX(p^n,N)_0^{(m)} \to \cX_0(p^n,N)_0^{(m)}$ is  Kummer \'etale  with automorphism group $\Delta_n$ isomorphic to the subgroup of $\GL_2(\Z/p^n\Z)$ of  matrices that are upper triangular modulo $p^n$. It is not a Galois cover as  $ \cX(p^n,N)_0^{(m)} $ is not connected. In fact, we have a trivilaization $T_p E/p^n T_p E=(\Z/p^n\Z) a \oplus (\Z/p^n\Z) b$ over $\cX(p^n,N)_0^{(m)}$ where $a$ generates the level subgroup $ H_n$. As remarked in \ref{rmk:standardopens} we have a decomposition $\cX(p^n,N)_{0}^{(m)}=\amalg_{\xi\in \Z/p^n \Z} \cX(p^n,N)_{0,\xi}^{(m)}$ where over $\cX(p^n,N)_{0,\xi}^{(m)}$ we have ${\rm dlog}(a^\vee)=\xi {\rm dlog} (b^\vee)$ modulo $p^n$.  
We recall that to give a sheaf on $\cX_0(p,N)_{0,\proket}^{(m)}$ is equivalent to give a sheaf on $\cX(p^n,N)_{0,\proket}^{(m)}$ endowed with a compatible action of $\Delta_n$. In particular to define $\Fil_h \mathfrak{D}_{k,\infty}^{o,(m)}[n] $  we define the filtration over $\cX(p^n,N)_{0,\proket}^{(m)}$ and we prove that it is stable for the action of $\Delta_n$. The maps ${\rm dlog}$ and ${\rm dlog}^\vee$ define an exact sequence of sheaves on the pro-Kummer \'etale site of $\cX(p^n,N)_0^{(m)}$:

$$Q:=\omega_E^{\rm mod,-1} \otimes_{\cO_{\cX_0(p^n,N)_0^{(m)}}^+} \widehat{\cO}_{\cX(p^n,N)_0^{(m)}}^+ \lra T_p(E) \otimes_{\Z_p} \widehat{\cO}_{\cX(p^n,N)_0^{(m)}}^+ \lra \omega_E^{\rm mod}\otimes_{\cO_{\cX_0(p^n,N)_0^{(m)}}^+} \widehat{\cO}_{\cX(p^n,N)_0^{(m)}}^+.$$ We recall that we have a trivilaization $T_p E/p^n T_p E=(\Z/p^n\Z) a \oplus (\Z/p^n\Z) b$ over $\cX(p^n,N)_0^{(m)}$ where $a$ generates the level subgroup $ H_n$. Over the component $\cX(p^n,N)_{0,\xi}^{(m)}$  the canonical subgroup $H_n$, and hence $Q$ modulo $p^n$, is generated by the section $b+\xi a$ and $a$ maps to a generator of the quotient $\omega_E^{\rm mod}\otimes_{\cO_{\cX_0(p,N)_0^{(m)}}^+} \widehat{\cO}_{\cX(p^n,N)_0^{(m)}}^+$ modulo $p^n$.  In particular,  we get an affine map

$$\rho\colon \V_0\bigl(\omega_E^{\rm mod,-1}, a\bigr)\subset \V_0\bigl(T_p(E)\otimes_{\Z_p} \widehat{\cO}_{\cX(p^n,N)_{0,-\lambda}^{(m)}}^+, a, b-\lambda a\bigr)   \cong \V_0\bigl(T_p(E) , a, b-\lambda a\bigr) \times_{\Z_p} \cX(p^n,N)_{0,-\lambda}^{(m)}   $$on the pro-Kummer \'etale site of $\cX(p^n,N)_{0,-\lambda}^{(n)} $. Considering the underlying sheaves of functions of weight $0$, the map $\rho^\ast$ induces a surjective map of sheaves of rings$$\rho_0^\ast\colon \WW_0\bigl(T_p(E),a,b-\lambda a\bigr) \otimes_{\Z_p} \widehat{\cO}_{\cX(p^n,N)_{0,-\lambda}^{(m)}}^+ \to \widehat{\cO}_{\cX(p^n,N)_{0,-\lambda}^{(m)}}^+.$$We let $\mathcal{I}$ be its kernel. We set $$\mathfrak{W}_{k,0,\lambda}^{(m)}:= \WW_k\bigl(T_p(E),a,b-\lambda a\bigr) \otimes_{\Z_p} \widehat{\cO}_{\cX(p^n,N)_{0,-\lambda}^{(m)}}^+,\qquad \Fil^h \mathfrak{W}_{k,0,\lambda}^{(m)}:= \mathcal{I}^h \mathfrak{W}_{k,0,\lambda}^{(m)} .$$Taking $\widehat{\cO}_{\cX(p^n,N)_\infty^{(m)}}^+$-duals we get  a sheaf and an increasing filtration $$\mathfrak{W}_{k,0,\lambda}^{(m),\vee}=\WW_k\bigl(T_p(E),a,b-\lambda b\bigr)^\vee \hat{\otimes}_{\Z_p} \widehat{\cO}_{\cX(p^n,N)_{0,-\lambda}^{(m}}^+, \qquad \Fil_h \mathfrak{W}_{k,0,\lambda}^{(m),\vee}, h\geq -1 $$on the pro-Kummer \'etale site of $\cX(p^n,N)_{0,-\lambda}^{(m)}$. Here $\Fil_h \mathfrak{W}_{k,0,\lambda}^{(m),\vee}$ consists of those sections vanishing on $\Fil^h \mathfrak{W}_{k,0,\lambda}^{(m)}$.
Due to Proposition \ref{prop:AkWkveee} we have a $\Delta_n$-equivariant isomorphism 
 $$j_n^\ast \bigl( \mathfrak{D}_{k,0}^{o,(m)}[n]\bigr)\cong  \oplus_{\lambda\in\Z/p^n\Z} \mathfrak{W}_{k,0,\lambda}^{(m),\vee}$$and we set $\Fil_h j_n^\ast \bigl( \mathfrak{D}_{k,\infty}^{o,(m)}[n]\bigr)$ to be the filtration given by  $\oplus_{\lambda\in\Z/p^n\Z} \Fil_h\mathfrak{W}_{k,0,\lambda}^{(m),\vee}  $.

In order to get a well defined filtration on $\mathfrak{D}_{k,0}^{o,(m)}[n]$ we need to prove that $\Fil_h j_n^\ast \bigl( \mathfrak{D}_{k,\infty}^{o,(m)}[n]\bigr)$ is $\Delta_n$-equivariant. if this holds, claim (i) is then clear by construction. Arguing as in Proposition \ref{prop:Wkinfty} to prove claim (ii) it suffices to prove it for $\mathrm{H}^1\bigl(\cX(p^n,N)_{0,\proket}^{(m)},  \_ \,\bigr)$. As in loc.~cit.~one reduces the proof of this statement and the statement of the $\Delta_n$-equivariance after passing to a log affinoid perfectoid cover $U$ of $\cX(p^n,N)_{0}^{(m)}$ with group of autmotomorphsism $G_U$ relatively to  $\cX_0(p^n,N)_{0}^{(m)}$ and with $\widehat{U}:=\mathrm{Spa}(R,R^+)$ the associated affinoid perfectoid space.  Write $T_p(E)(U)\otimes R^+=R^+ f_0 + R^+ f_1$ and $T_p(E)^\vee(U)\otimes R^+=e_0 R^++
e_1R^+$ with $e_0=f_0^\vee\equiv b^\vee$  mapping to a generator of $\omega_E^{\rm mod}$ and $e_1=f_1^\vee\equiv a^\vee-\lambda b^\vee$  in the kernel of ${\rm dlog}$ generating $\omega_E^{\rm mod,-1}$. Recall from Proposition \ref{prop:AkWkveee} that $\mathfrak{W}_{k,0,\lambda}^{(m),\vee}(U)$ is the dual of $\mathfrak{W}_{k,0,\lambda}^{(m)}(U)\cong (R^+\otimes B)\langle \frac{W_\lambda}{1+p^n Z} \rangle \cdot k(1+p^n Z)$ where $p^n W_\lambda= Y$ and $X=1+p^n Z$ and we have the universal map $\alpha f_0 + \beta f_1\mapsto 
\alpha Y +\beta X$ (note the roles of $X$ and $Y$ are interchanged compared to loc.~cit.~as in this case we are looking for functions on $T$ that are $0$ on $b$ modulo $p^n$ and $1$ on $a$ modulo $p^n$, or equivalently that are $1$ on $e_1$ modulo $p^n$ and $0$ on $e_0-\lambda e_1$ modulo $p^n$). The decreasing filtration is defined by $\Fil^h = \oplus_{i\geq h}(R^+\otimes B) \bigl(\frac{W_\lambda}{1+p^n Z}\bigr)^i  \cdot k(1+p^n Z) $. 

Given $\sigma\in G_U$ we have $\sigma(e_0)=e_0$ and $\sigma(e_1)=e_1+ \xi(\sigma) e_0$ so that $\sigma(Y)=Y$ and $\sigma(X)=X+\xi(\sigma) Y= X+ p^n \xi(\sigma) W_\lambda$. Here $\xi$ is an $R^+$-valued continuous $1$-cocycle on $G_U$. In particular,  write $k(t)=\exp(u_k \log(t))$ for $t\equiv 1$ modulo $p^n$. Then $$\sigma\bigl(k(X) X^{-i}\bigr)=k\bigl(\sigma(X)\bigr)\sigma(X)^{-i}= \exp\bigl((u_k-i) \log(X (1+ p^n \xi(\sigma) \frac{W_\lambda}{X})\bigr)=$$ $$=k(X) X^{-i} \exp\bigl((u_k-i) \log(1+ p^n \xi(\sigma) \frac{W_\lambda}{X})\bigr)=k(X) X^{-i} \sum_{m=0}^\infty  p^{n m} \left(\begin{matrix} u_k-i \cr m \end{matrix}\right)\bigl(\xi(\sigma) \frac{W_\lambda}{X}\bigr)^m.$$Notice that the term $m=0$ is $(u_k-i) k(X) X^{-i}$. Thus

\begin{equation}\label{eq:sigma} \sigma\bigl(\bigl(\frac{W_\lambda}{1+p^n Z}\bigr)^i  \cdot k(1+p^n Z) \bigr)=\bigl(\frac{W_\lambda}{1+p^n Z}\bigr)^i  \cdot k(1+p^n Z)   \sum_{m}  p^{n m} \left(\begin{matrix} u_k-i \cr m \end{matrix}\right)\bigl(\xi(\sigma) \frac{W_\lambda}{1+p^n Z}\bigr)^m
\end{equation}

This implies first of all that $G_U$ preserves the filtration. Claim (ii) follows from the following lemma.

\end{proof}

\begin{lemma} For every $h$ consider the short exact sequence of $G_U$-modules

$$0 \to \Fil_{h} \mathfrak{W}_{k,0,\lambda}^{(m),\vee}(U) \to \Fil_{2h} \mathfrak{W}_{k,0,\lambda}^{(m),\vee}(U) \to \Fil_{2h}\mathfrak{W}_{k,0,\lambda}^{(m),\vee}(U)/ \Fil_{h}\mathfrak{W}_{k,0,\lambda}^{(m),\vee}(U)\to 0.$$Then, the connecting homomorphism $$\mathrm{H}^0\bigl( G_U, \Fil_{2h}\mathfrak{W}_{k,0,\lambda}^{(m),\vee}(U)/ \Fil_{h}\mathfrak{W}_{k,0,\lambda}^{(m),\vee}(U) \bigr) \to \mathrm{H}^1\bigl(G_U, \Fil_{h} \mathfrak{W}_{k,0,\lambda}^{(m),\vee}(U)\bigr)$$has cokernel annihilated by $p^{(n+c) h}h! \left(\begin{matrix} u_k \cr h \end{matrix} \right) $ for some $c$ depending on $n$. 

\end{lemma} 

\begin{proof} We argue by induction on $h$. We are reduced to prove that for any $h$ the short exat sequence 

$$0 \to \mathrm{Gr}_{h} \mathfrak{W}_{k,0,\lambda}^{(m),\vee}(U) \to \Fil_{h+1} \mathfrak{W}_{k,0,\lambda}^{(m),\vee}(U)/\Fil_{h-1}\mathfrak{W}_{k,0,\lambda}^{(m),\vee} (U) \to \mathrm{Gr}_{h+1}\mathfrak{W}_{k,0,\lambda}^{(m)}(U)\to 0$$the connecting homomorphism$$\mathrm{H}^0\bigl( G_U, \mathrm{Gr}_{h+1}\mathfrak{W}_{k,0,\lambda}^{(m)}(U) \bigr) \to \mathrm{H}^1\bigl(G_U, \mathrm{Gr}_{h} \mathfrak{W}_{k,0,\lambda}^{(m)}(U)\bigr)$$has cokernel annihilated by $p^{(n+c)} \frac{(u_k- h)}{h} $ for some $c$ depending on $n$. Note that 

$$\mathrm{H}^0\bigl( G_U, \mathrm{Gr}_{h+1}\mathfrak{W}_{k,0,\lambda}^{(m)}(U) \bigr)=\bigl((R^+)^{G_U} \otimes B\bigr) \left(\bigl(\frac{W_\lambda}{1+p^n Z}\bigr)^{h+1} \cdot k(1+p^n Z)\right)^\vee
$$
and  
$$\mathrm{H}^1\bigl(G_U, \mathrm{Gr}_{h} \mathfrak{W}_{k,0,\lambda}^{(m)}(U)\bigr)=\bigl((R^+)^{G_U} \otimes B\bigr) \otimes_{(R^+)^{G_U}} \left(\bigl(\frac{W_\lambda}{1+p^n Z}\bigr)^h \cdot k(1+p^n Z)\right)^\vee   \mathrm{H}^1\bigl(G_U, R^+\bigr).
$$
Thanks to (\ref{eq:sigma}) the map sends $\left(\bigl(\frac{W_\lambda}{1+p^n Z}\bigr)^{h+1} \cdot k(1+p^n Z)\right)^\vee $ to the cocyle $G_U\ni \sigma \mapsto (u_k-h) p^n \xi(\sigma) \left(\bigl(\frac{W_\lambda}{1+p^n Z}\bigr)^h \cdot k(1+p^n Z)\right)^\vee  $. The quotient of $ \mathrm{H}^1\bigl(G_U, R^+\bigr)$ by the $(R^+)^{G_U}$-span of  the cocycle $\xi$ is torsion and hence killed by a power $p^c$ of $p$; see \cite[Prop. 5.2]{half}. The conclusion follows.

\end{proof}

\section{The Hodge-Tate Eichler-Shimura map revisited.}

Let $k\colon \Z_p^\ast\lra B^\ast$ be a $B$-valued weight, as in Definition \ref{def:weights}, which is $m'$-analytic for some $m'\in \N$ like in definition \ref{def:weights}, i.e.
there is $u_k\in B[1/p]$ such that $k(t)={\rm exp}\bigl(u_k\log(t)\bigr)$ for all $t\in 1+p^{m'}\Z_p$.  In this section we fix an integer $m\ge m'$ and denote $\bD^o_k(T_0^\vee)[m]$ the integral pro-Kummer \'etale sheaf of distributions on the base-change of $\cX:=\cX_0(p^m, N)$ over ${\rm Spa}(B[1/p], B)$. We will be simply denote this sheaf  by $\bD^o_k$ in this section, and also set $\bD_k:=\bD_k^o\otimes_{\Z_p}\Q_p$. 

Let us fix a slope $b\in \N$ and recall that if $M$ is a $\Q_p$-vector space with a linear endomorphism  $U_p$, we denote $M^{(b)}$ the subvector space of $M$ of elements $x\in M$ such that $P(U_p)(x)=0$ for all polynomials $P(X)\in \Q_p[X]$ whose roots in $\C_p$ have all valuations in $[0,b]\cap \Q$.  Up to  localization of $B$ and for $s$ large enough both ${\rm H}^1\bigl(\cX_\proket, \bD_k \bigr)$ and  ${\rm H}^0\bigl(\cX\langle\frac{p}{\mathrm{Ha}^{p^s}} \rangle, \omega^{k+2}\bigr)$ admit slope $b$ decomposition; here $\mathrm{Ha}$ is a (any) local lift of the Hasse invariant.  Then the main result of this section is 

\begin{theorem}
\label{thm:hodgetate}
For $s$ large enough, there is a canonical $\C_p$-linear, Galois and Hecke equivariant map:
$$
\Psi_{\rm HT}\colon {\rm H}^1\bigl(\cX_\proket, \bD_k(1)  \bigr)^{(b)}\widehat{\otimes}\C_p\lra {\rm H}^0\bigl(\cX\langle\frac{p}{\mathrm{Ha}^{p^s}} \rangle , \omega^{k+2}\bigr)^{(b)}\widehat{\otimes}\C_p.
$$ 
Moreover, if \ $\prod_{i=0}^{b-1}(u_k-i)\in \bigl(B[1/p]  \bigr)^\ast$ then $\Psi_{\rm HT}$ is surjective.

\end{theorem}

The map $\Phi_{\rm HT}$ was defined in \cite{EichlerShimura} using Faltings' sites of $\cX$ and $\cX\langle\frac{p}{\mathrm{Ha}^{p^s}} \rangle$ respectively. To relate it to the language developed in this paper first of all we allow ourself to increase $m$.  In fact, the map  $\cX_0(p^{m_1},N)\to \cX_0(p^{m_2},N)$ for $m_1\geq m_2$ is finite \'etale and one can obtain the map $\Psi_{\rm HT}$ for $m=m_1$ in the Theorem upon taking traces from the map $\Psi_{\rm HT}$ for $m=m_2$; on the LHS the trace is defined using Lemma \ref{lemma:trace}. We also replace $\cX_0(p^m,N)\langle\frac{p}{\mathrm{Ha}^{p^s}} \rangle$ with $\cX_0(p^m,N)_\infty^{(u)}$ for arbitrary large integers $m\geq u$, as the first is contained in the latter for $u$ large enough thanks to \cite[Lemma 3.3.8]{ScholzeTorsion}.

In particular, we choose $m$ large enough such that there is $u\le m$ with the property that  the restriction of $\bD_k$ to the pro-Kummer \'etale site of $\cX_0(p^m, N)_\infty^{(u)}$  has the property that the (pro-Kummer \'etale) sheaf $\bD_k\widehat{\otimes}\widehat{\cO}_{\cX_0(p^m, N)_\infty^{(u)}}$ has the decreasing filtration defined in Proposition \ref{prop:Wkinfty}. We fix such $m$, $u$. Let us denote by $\fD_k^o:=\bD_k^o\widehat{\otimes}\cO_{\cX_\proket}^+$, where $\cO_{\cX_\proket}^+$ is the structure sheaf of the pro-Kummer \'etal site of $\cX$. We recall from \cite{EichlerShimura} that the map $\Psi_{\rm HT}$ is the composition of the following two maps:
$$
{\rm H}^1\bigl(\cX_0(p^m,N)_\proket, \bD_k(1) \bigr)^{(b)}\widehat{\otimes}\C_p\cong \Bigl({\rm H}^1\bigl(\cX_0(p^m,N)_\proket, \fD_k ^o(1)  \bigr)[1/p]\Bigr)^{(b)}\stackrel{\mathcal{R}}{\lra}$$

$$\stackrel{\mathcal{R}}{\lra}\Bigl({\rm H}^1\bigl((\cX_0(p^m,N)_\infty^{(u)})_\proket, \fD_k^o(1)\bigr)[1/p]\Bigr)^{(b)}\stackrel{\Phi}{\lra}  {\rm H}^0\bigl(\cX_0(p^m,N)_\infty^{(u)}, \omega^{k+2}   \bigr)^{(b)}\widehat{\otimes} \C_p,
$$
where $\mathcal{R}$ is the restriction map while the map $\Phi$ is defined in Proposition \ref{prop:HeckeWkinfty} and it is  proved in loc.~cit.~that it is an isomorphism. Therefore, in order to prove Theorem \ref{thm:hodgetate} it is enough to prove

\begin{theorem}
\label{thm:htsurjective}
In the notations above, if \ $\prod_{i=0}^{b-1}(u_k-i)\in \bigl(B[1/p]  \bigr)^\ast$ then the map $\mathcal{R}$ is surjective.

\end{theorem}

To simplify the notation, in the rest of the section we write $\cX$ instead of $\cX_0(p^m,N)$. Before starting the proof of Theorem \ref{thm:htsurjective} we'll describe the dynamic of the $U_p$-operator on the modular curve $\cX$. We think about $U_p$ as correspondences on $\cX$ and we have:

\begin{lemma}
\label{lemma:dynamicup}
 
\noindent
For all $n\ge 1$ we have

i)   $U_p^{n+1}\bigl(\cX\backslash \cX_0^{(n+1)}\bigr)\subset \cX_\infty^{(1)}$ and $U_p^n(\cX\backslash\cX_0^{(n+1)})\subset \cX_\infty\backslash\cX_0^{(1)}$.

ii) $U_p^{n}\bigl(\cX_\infty^{(1)}   \bigr) \subset \cX_\infty^{(n+1)}$.
\end{lemma}

\begin{proof}
This is a direct consequence of Lemma \ref{lemma:dynamic}.
\end{proof}

We also have:

\begin{lemma}
\label{lemma:goodbad}
For every $n\ge 1$ there is a canonical decomposition of correspondences $U_p^n|_{\cX_0^{(n)}}=\bigl(U_p^n\bigr)^{\rm good}\amalg \bigl(U_p^n  \bigr)^{\rm bad}$ such that:

a) $(U_p^n)^{\rm good}\bigl(\cX_0^{(n)}\bigr)\subset \cX_\infty^{(1)}.$

b)   $(U_p^n)^{\rm bad}\bigl(\cX_0^{(n)} \backslash \cX_0^{(n+1)}  \bigr)\subset \cX_\infty\backslash \cX_0^{(1)}$.

\end{lemma}

\begin{proof} In view of remark \ref{rmk:uptlambda} it is enough to define this decomposition for
the correspondence $u:=\cup_{\mu}t_\mu:U_0^{(n)}\lra \mathcal{P}\bigl(\PP^1\bigr)$, where $\mu=\lambda_0+\lambda_1p+...+\lambda_{n-1}p^{n-1}$, $\lambda_i\in \{0,1,...,p-1\}$, for $i=0,1,...,n-1.$
Namely we define $u^{\rm bad}\vert_{U_{0,\mu}^{(n)}}:=t_\mu$ and $u^{\rm good}\vert_{U_{0,\mu}^{(n)}}:=\cup_{\lambda\neq \mu}t_\lambda$.

With these definitions the rest of Lemma \ref{lemma:goodbad} follows from Lemma \ref{lemma:dynamic}.

\end{proof}

We have the following consequence of Lemma \ref{lemma:dynamicup} and Lemma \ref{lemma:goodbad}.
Let us recall our notation $\fD_k^o:=\bD_k^o\widehat{\otimes}\cO^+_\cX$, this is a   sheaf on the pro-Kummer \'etale site of the base-change of $\cX$ to ${\rm Spa}(B[1/p], B)$, base-change which  is not shown in the notations. Also we recall that we use the operator $U_p^{\rm naive}$ induced on cohomomogy by the $U_p$-correspondence and not its normalized version $p^{-1} U_p^{\rm naive}$. With this understanding, to ease the notation, we simply write $U_p$ instead of $U_p^{\rm naive}$.  

\begin{corollary}
\label{cor:polynomialup}
Let $P(T)\in (B\widehat{\otimes}\cO_{\C_p})[T]$ such that $P(T)=TR(T)$ and for every $n\ge 1$ we denote $\bigl(P(U_p)^n   \bigr)^{\rm good}:=(U_p^n)^{\rm good}R(U_p)^n$  and $\bigl(P(U_p)^n\bigr)^{\rm bad}:=\bigl(U_p^n)^{\rm bad}R(U_p)^n$ where:

i) $P(U_p)^{n+1}\colon {\rm H}^1\bigl((\cX_\infty^{(1)})_\proket, \fD_k^o \bigr)\lra    {\rm H}^1\bigl((\cX\backslash \cX_0^{(n+1)})_\proket, \fD_k^o   \bigr)$.

ii) $P(U_p)^{n-1}\colon {\rm H}^1\bigl((\cX_\infty^{(n)})_\proket, \fD_k^o  \bigr)\lra {\rm H}^1\bigl((\cX_\infty^{(1)})_\proket, \fD_k^o  \bigr)$.

iii) $\bigl(P(U_p)^n\bigr)^{\rm good}P(U_p)\colon {\rm H}^1\bigl((\cX_\infty\backslash \cX_0^{(1)})_\proket, \fD_k^o\bigr) \lra {\rm H}^1\bigl( (\cX_0^{(n)})_\proket, \fD_k^o \bigr)$.
\end{corollary} 

(iv)     $\bigl(P(U_p)^n\bigr)^{\rm bad}\colon   {\rm H}^1\bigl((\cX_\infty\backslash \cX_0^{(1)})_\proket, \fD_k^o\bigr) \lra {\rm H}^1\bigl( (\cX_0^{(n)}    \backslash \cX_0^{(n+1)})_\proket, \fD_k^o \bigr)$.

\begin{proof}
In view of the fact that for any polynomial $Q(T)$ with $(B\widehat{\otimes}\cO_{\C_p})$-coeffcients, $Q(U_p)$ maps
${\rm H}^1\bigl((\cX_\infty^{(1)})_\proket, \fD_k^o   \bigr)$ to itself, i) and respectively ii) are immediate consequences of Lemma 
\ref{lemma:dynamicup} a) and b) respectively, while iii) is a consequence of Lemma \ref{lemma:goodbad} a).
\end{proof}

\medskip

Before we start the actual proof of Theorem \ref{thm:htsurjective} it seems natural the recall and gather here the main ingredients in the proof,
 namely the properties of the cohomology of the filtrations of the sheaf $\fD_k^o$ on $\cX_\infty^{(u)}$ and respectively on $\cX_0^{(n)}$, for $u$ as fixed at the beginning of this section and for the moment 
$n\ge 1$ such that the pro-Kummer \'etale sheaf $\fD_k^o$ has the increasing filtration $\fFil_\bullet$ of Proposition 
\ref{prop:Wk0} when restricted to $\cX_0^{(n)}$.

Recall that we have denoted by $\fD^o_k:=\bD^o_k\widehat{\otimes}\widehat{\cO}^+_{\cX}$ for $\widehat{\cO}^+_{\cX}$ the completion of the structure sheaf of the pro-Kummer \'etale site of $\cX$.  We have:

i) The image of the morphism ${\rm H}^1\bigl((\cX^{(n)}_0)_{\proket}, \fFil_h\bigr)\lra {\rm H}^1\bigl((\cX^{(n)}_0)_{\proket}, \fD_k^o  \bigr)$ is annihilated by $p^{(m+c)h}\prod_{i=0}^{h-1}(u_k-i)/h!$, where let us recall $\cX=\cX_0(p^m, N)$ and $c$ is constant with respect to $h$.
In particular if $\Bigl(\prod_{i=0}^{h-1}(u_k-i)/h!\Bigr)\gamma=p^q$ for some $\gamma\in B$, then
$p^{(m+c)h+q}$ annihilates the image of the above map.

\medskip
ii) Recall that the $U_p$ correspondence on $\cX_0^{(n)}$ can be written as the disjoint union of $U_p^{\rm good}$, mapping $\cX_0^{(n)}$ to $\cX_\infty^{(1)}$, and $U_p^{\rm bad}\colon \cX_0^{(n)}\to \cX_0^{(n-1)}$ and corresponding to the $p$-isogeny $E \to E'$ given by modding out by the canonical subgroup. To keep track of where our sheaves are defined we denote $\mathfrak{D}_{k,0}^{o,(s)}$ the restriction of the pro-Kummer \'etale sheaf $\mathfrak{D}^o_k$ to $(\cX_0^{(s)})_{\rm pke}$. Suppose $1\le u\le n$ is such that $\fD_{0,k}^{o,(u-1)}$ has the increasing filtration $\mathfrak{Fil}_\bullet^{(u-1)}$ on $(\cX_0^{(u-1)})_{\rm pke}$. We have:

\begin{proposition}\label{prop:Ubad} The map $U_p^{\rm bad}$ induces a map $U_p^{\rm bad}\colon  \mathfrak{D}_{k,0}^{o,(u)} \to \gamma^\ast \bigl(\mathfrak{D}_{k,0}^{o,(u-1)} \bigr) $ that preserves the filtration. Furthermore for every $h$ there exists an operator $$U_p^{\rm bad,h}\colon \mathfrak{D}_{k,0}^{o,(u)}/ \mathfrak{Fil}_h^{(u)} \to \gamma^\ast \bigl(\mathfrak{D}_{k,0}^{o,(u-1)}/\mathfrak{Fil}_h^{(u-1)} \bigr)$$such that $U_p^{\rm bad}=p^h U_p^{\rm bad,h}$ on $\mathfrak{D}_{k,0}^{o,(u)} /\mathfrak{Fil}_h^{(u)}$. Here $\gamma:\cX_0^{(u)}\subset \cX_0^{(u-1)}$ is the inclusion. 
\end{proposition}
\begin{proof} We use the notation of Proposition \ref{prop:Wk0}. It suffices to prove both statements after passing to a log affinoid perfectoid cover $U$ of $\cX_{0,\lambda}^{(u)}$ for $\lambda=\lambda_0+\lambda_1p+...+\lambda_{n-1}p^{n-1}$, $\lambda_i\in \{0,1,...,p-1\}$, for $i=0,1,...,n-1$, with associated affinoid perfectoid space $\widehat{U}:=\mathrm{Spa}(R,R^+)$. Recall form loc. cit. that $T_p(E)$ is trivialized $T_p(E)(U)=\Z_pa\oplus \Z_p b$ where  the level subgroup is generated by $a$ and the canonical subgroup is generated by the section $b+\lambda a$. 
For every $\lambda$ as above write $\lambda=\lambda_0 + p \lambda'$ and then $U_p^{\rm bad}$ restricts to a map $u_{\lambda_0}\colon \cX_0(p^n,N)_{0,\lambda}^{(m)} \to \cX_0(p^n,N)_{0,\lambda'}^{(m-1)}$. At the level of universal elliptic curves over $U$ it corresponds to the $p$-isogeny defined by $u_{\lambda_0}\colon T_p(E)(U)\to T_p(E')(U)$ with $T_p(E')(U)=\Z_p a \oplus \Z_p b'$ with $b'=\frac{b+\lambda_0 a}{p}$ in $T_p(E)\otimes \Q$. Notice that $\frac{b+\lambda a}{p}= b' + \lambda' a$ defines a generator of the canonical subgroup on $E'$. The map $u_{\lambda_0}$ induces by functoriality the map $\mathfrak{D}_{k,0,\lambda}^{o,(u)} \to u^\ast \bigl(\mathfrak{D}_{k,0,\lambda_0}^{o,(u-1)} \bigr)$. We describe it explicitly.

Write $T_p(E)(U)\otimes R^+=R^+ f_0 + R^+ f_1$ as in Proposition \ref{prop:Wk0} so that $f_0\equiv b+\lambda a$ and $f_1\equiv a$ modulo $p^n$ and $T_p(E)^\vee(U)\otimes R^+=e_0 R^++ e_1R^+$ with $e_0=f_0^\vee\equiv b^\vee$  mapping to a generator of $\omega_E^{\rm mod}$ and $e_1=f_1^\vee\equiv a^\vee-\lambda b^\vee$  in the kernel of ${\rm dlog}$ generating $\omega_E^{\rm mod,-1}$. Then $\mathfrak{W}_{k,0,\lambda}^{(u),\vee}(U)$ is the dual of $\mathfrak{W}_{k,0,\lambda}^{(u)}(U)\cong R^+\otimes B\langle \frac{W_\lambda}{1+p^n Z} \rangle \cdot k(1+p^n Z)$. Similalry write $T_p(E')(U)\otimes R^+=R^+ f_0' + R^+ f_1'$ with $f_0'=\frac{f_0}{p}=b' + \lambda' a$ and $f_1'=f_1$. Then $T_p(E')^\vee(U)\otimes R^+=e_0' R^++
e_1'R^+$ with $e_0'=(f_0')^\vee=p e_0$  mapping to a generator of $\omega_{E'}^{\rm mod}$ and $e_1'=(f_1')^\vee=f_1$  in the kernel of ${\rm dlog}$ for $E'$ generating $\omega_{E'}^{\rm mod,-1}$. In particular,   $\mathfrak{W}_{k,0,\lambda_0}^{(u-1),\vee}(U)$ is the dual of $\mathfrak{W}_{k,0,\lambda_0}^{(u-1)}(U)\cong R^+\otimes B\langle \frac{W_{\lambda_0}'}{1+p^n Z'} \rangle \cdot k(1+p^n Z')$. The map $T_p(E)\to T_p(E')$ induces a map $T_p(E')^\vee \to T_p(E)^\vee$ on the duals and, hence a map 

$$
v_\lambda\colon \mathfrak{W}_{k,0,\lambda}^{(u)}(U)  \to \mathfrak{W}_{k,0,\lambda_0}^{(u-1)}(U),\qquad Z\mapsto Z', W_\lambda \mapsto p W_{\lambda_0}. 
$$
As the decreasing filtrations are defined by  
$$\Fil^h \mathfrak{W}_{k,0,\lambda}^{(u)}(U)= \mathfrak{W}_{k,0,\lambda}^{(u)}(U) \cdot \bigl( \frac{W_\lambda}{1+p^n Z} \bigr)^h,\quad \Fil^h \mathfrak{W}_{k,0,\lambda_0}^{(u-1)}(U)= \mathfrak{W}_{k,0,\lambda_0}^{(u-1)}(U) \cdot \bigl(\frac{W_{\lambda_0}'}{1+p^n Z'} \bigr)^h
$$
we see that $v_\lambda$ respects the filtrations and on $\Fil^h$ can be written as $p^h v_\lambda'$ with $v_\lambda'\colon \Fil^h \mathfrak{W}_{k,0,\lambda}^{(u)}(U)\to \Fil^h \mathfrak{W}_{k,0,\lambda_0}^{(u-1)}(U)$. The claim follows upon taking strong $R^+$-duals.

\end{proof}
\medskip

iii) Given a slope $b\in \N$ there exist integers $d_1,d_2, d:=d_1+d_2$ such that
$p^d$ annihilates ${\rm Ker}\Bigl({\rm H}^1\bigl((\cX_\infty^{(n)})_\proket, \fD_k^o   \bigr)\lra {\rm H}^1\bigl((\cX_\infty^{(n)})_\proket, \fD_k\bigr)\Bigr)$ and $p^de_b$ is integral, where $e_b$ is an idempotent projecting onto the slope $\le b$ of ${\rm H}^1\bigl(\cX_\proket, \fD_k\bigr)$ and of ${\rm H}^1\bigl((\cX_\infty^{(n)})_\proket, \fD_k\bigr)$. This is the main result of Section \S 7.

\bigskip
\noindent
{\bf Proof of Theorem \ref{thm:htsurjective}.}

We now start the proof of Theorem \ref{thm:htsurjective}. We work with a weight $k$ as in Definition \ref{def:weights}, which will be assumed in this proof to be the universal weight of  some (wide) open disk of the weight space. The case of a finite extension of $\Q_p$ is obtained by specialization. In particular $u_k-i $ is not $0$ in $B$, for every $i\in \N$. This will be used in Step 3. 

We recall that we have fixed  a slope $b\in \N$ and consider the module$${\rm H}^1\bigl((\cX_\infty^{(u)})_\proket, \fD_k   \bigr)^{(b)}\cong \Bigl({\rm H}^1\bigl(\cX_\infty^{(u)}, \bD_k\bigr)\widehat{\otimes}\C_p\Bigr)^{(b)}.$$ Let  $Q(T)\in (B\widehat{\otimes}\cO_{\C_p})[T]$ be the polynomial with ${\rm deg}(Q(T))\ge 1$ 
and having the property:   $y\in {\rm H}^1\bigl((\cX_\infty^{(u)})_\proket, \fD_k\bigr)^{(b)}$ if and only if $Q(U_p)y=0$. Such a polynomial exists as by \cite{half} we have an isomorphism ${\rm H}^1\bigl((\cX_\infty^{(u)})_\proket, \fD_k\bigr)^{(b)}\cong {\rm H}^0\bigl(\cX_\infty^{(u)}, \omega^{k+2}\bigr)^{(b)}$ and on ${\rm H}^0\bigl(\cX_\infty^{(u)}, \omega^{k+2}\bigr)$ the operator $U_p$ is compact and has a Fredholm determinant which is an entire power series. $Q(T)$ is obtained from a factor of this Fredholm determinant. We write $Q(T)=P(T)-\alpha$, with $P(T)=TR(T)$
and remark that there is $a\in \Q$ such that $\alpha\in p^a(B\widehat{\otimes}\cO_{\C_p})^\ast$. Then $a\le {\rm deg}(Q(T)\cdot b$. Now we choose integers $h:=2a+2$, $s:=(m+c)h+q$ (see i) above) and $d:=d_1+d_2$ as in iii) above.
Then it is easy to verify that there exist integers $n,r\ge 1$ such that

1) $\displaystyle r\ge \frac{r}{2}+s+d+1+ (n+u+1)a$

\noindent and

2) $nh\ge r$.

We will work with the open affinoids $\cX_\infty^{(u)}:=\cX_0(p^m, N)_\infty^{(u)}, \ \cX_0^{(n)}:=\cX_0(p^m, N)_0^{(n)},\  \cX_0^{n+1)}:=\cX_0(p^m, N)_0^{(n+1)}\subset \cX_0(p^m, N)$ and their pro-Kummer \'etale sites. 

\medskip
\noindent
{\bf Step 1.}    Let $x\in {\rm H}^1\bigl((\cX_\infty^{(u)})_\proket, \fD_k   \bigr)^{(b)}$, then $P(U_p)x=\alpha x$.  As ${\rm H}^1\bigl((\cX_\infty^{(u)})_\proket, \fD_k   \bigr)={\rm H}^1\bigl((\cX_\infty^{(u)})_\proket, \fD^o_k   \bigr)[1/p]$, without loss of generality we may consider $x\in {\rm H}^1\bigl((\cX_\infty^{(u)})_\proket, \fD^o_k   \bigr)^{\rm tf}$ such that $P(U_p)x=\alpha x$ (this   notation was introduced  in Section \S \ref{sec:appendix}.) Then there is a unique $x'\in {\rm H}^1\bigl((\cX_\infty^{(u)})_\proket, \fD^o_k   \bigr)$ with $P(U_p)x'=\alpha x'$ and $(x')^{\rm tf}=p^dx$
and using Corollary \ref{cor:polynomialup} ii)
we have: $P(U_p)^{u}(x')\in {\rm H}^1\bigl((\cX_\infty^{(1)})_\proket, \fD^o_k   \bigr)$
and by using Corollary \ref{cor:polynomialup} i) we have that $P(U_p)^{n+u+1}(x')\in {\rm H}^1\bigl((\cX\backslash\cX_0^{(n+1})_\proket, \fD_k^o   \bigr).$
 We denote by
$\tilde{x}$ the image of $P(U_p)^{n+u+1}(x')$ in ${\rm H}^1\bigl((\cX\backslash\cX_0^{(n+1)})_\proket, \fD^o_k/p^r\fD_k^o   \bigr)$.

We recall that $P(U_p)^{u+1}(x')=P(U_p)\bigl(P(U_p)^u(x') \bigr)\in {\rm H}^1((\cX_\infty\backslash\cX_0^{(1)} )_{\rm prok}, \fD_k^o)$  by Corollary \ref{cor:polynomialup} i)  and so denote $\mathcal{P}\bigl(x\bigr)$ the image of 
$$
\bigl(P(U_p)^{n}\bigr)^{\rm good}\bigl(P(U_p)^{u+1}(x')  \bigr)\in {\rm H}^1\bigl((\cX_0^{(n)})_\proket, \fD_k^o   \bigr)
$$  
in ${\rm H}^1\bigl((\cX_0^{(n)})_\proket, \fD^o_k/p^r\fD_k^o   \bigr)$.

Now let us observe that the opens $(\cX\backslash \cX_0^{(n+1)})$ and $\cX_0^{(n)}$ constitute an open covering of $\cX$ and so we have a Mayer-Vietoris exact sequence:
$$
{\rm H}^1\bigl(\cX_\proket, \fD^o_k/p^r\fD_k^o\bigr)\stackrel{\varphi}{\lra} {\rm H}^1\bigl((\cX\backslash\cX_0^{(n+1)})_\proket, \fD^o_k/p^r\fD_k^o\bigr)\oplus {\rm H}^1\bigl((\cX_0^{(n)})_\proket, \fD^o_k/p^r\fD_k^o\bigr)\stackrel{\psi}{\lra} 
$$
$$
\stackrel{\psi}{\lra} {\rm H}^1\bigl((\cX_0^{(n)}\backslash\cX_0^{(n+1)})_\proket, \fD^o_k/p^r\fD_k^o\bigr),
$$
together with an almost isomorphism $\displaystyle {\rm H}^1\bigl(\cX_\proket, \bD^o_k\bigr)\widehat{\otimes} \cO_{\C_p}\stackrel{\rho}{\lra} {\rm H}^1\bigl(\cX_\proket, \fD^o_k\bigr)$.

We consider the element 
$$(A,B):=\bigl(p^s\tilde{x}, p^s\mathcal{P}(x)\bigr)\in {\rm H}^1\bigl((\cX\backslash\cX_0^{(n+1)})_\proket, \fD^o_k/p^r\fD_k^o\bigr)\oplus {\rm H}^1\bigl((\cX_0^{(n)})_\proket, \fD^o_k/p^r\fD_k^o\bigr)
$$ 
and we first claim: $\psi(A,B)=0$.

Indeed, let us denote by $\mathcal{V}$ the pro-Kummer \'etale site 
$(\cX_0^{(n)}\backslash\cX_0^{(n+1)})_\proket$. Thanks to Corollary \ref{cor:polynomialup} iv) we have
$$
\psi(A,B)=A|_{\mathcal{V}}-B|_{\cV}
=\bigl(P(U_p)^{n}\bigr)^{\rm bad}\bigl(p^sP(U_p)^{u+1}(x')\bigr)\in p^s{\rm H}^1\bigl(\cV, \fD^o_k/p^r\fD_k^o\bigr).
$$
Now let us recall that we have an exact sequence
$$
{\rm H}^1\bigl(\cV, \fFil_h/p^r\fFil_h\bigr)\stackrel{f}{\lra} {\rm H}^1\bigl(\cV, \fD^o_k/p^r\fD_k^o\bigr)\stackrel{g}{\lra} {\rm H}^1\bigl(\cV, \fD^o_k/(\fFil_h + p^r\fD_k^o)\bigr).
$$
By property ii) above and the fact that $nh\ge r$ we have $g\bigl(P(U_p)^{n})^{\rm bad}\bigl(P(U_p)^{u+1}(x')  \bigr)=0$ which implies that there is $\beta\in {\rm H}^1\bigl(\cV, \fFil_h/p^r\fFil_h\bigr)$ such that
$(U_p^n)^{\rm bad}\bigl(U_p^{u+1}(x') \bigr)=f(\beta)$. But by i) above the image of $f$ is annihilated by $p^s$ therefore
$\psi(A,B)=p^sf(\beta)=0$. This proves the claim.

\medskip
\noindent
We continue the proof of the theorem. The Mayer-Vietoris exact sequence implies that there is $y\in 
{\rm H}^1\bigl(\cX_\proket, \bD^o_k\bigr)\widehat{\otimes} \cO_{\C_p}$ such that 
$$
\rho(y)|_{(\cX_\infty^{(u)})_\proket}\equiv p^{s+1}P(U_p)^{n+u+1}x'=p^{s+1}\alpha^{n+u+1}x'\ \Bigl({\rm mod}\ p^r{\rm H}^1\bigl((\cX_\infty^{(u)})_\proket , \fD_k^o \bigr)\Bigr).
$$
 Let $z_0\in {\rm H}^1\bigl(\cX_\proket, \bD^o_k\bigr)\widehat{\otimes} \cO_{\C_p}$ be $z_0:=p^de_b(y)$. Then $z_0\in {\rm H}^1\bigl(\cX_\proket, \bD^o_k\bigr)^{(b)}\widehat{\otimes} \cO_{\C_p}$ and  $\rho(z_0)|_{(\cX_\infty^{(u)})_\proket}\equiv p^{s+d+1}\alpha^{n+u+1}x'\ \Bigl({\rm mod}\ p^r{\rm H}^1\bigl((\cX_\infty^{(u)})_\proket , \fD_k^o \bigr)^{(b)}\Bigr)$ where we write ${\rm H}^1\bigl((\cX_\infty^{(u)})_\proket , \fD_k^o \bigr)^{(b)}$ for the image   $p^d e_b {\rm H}^1\bigl((\cX_\infty^{(u)})_\proket , \fD_k^o \bigr)$.

\medskip
\noindent
{\bf Step 2.}  Let $x_1\in {\rm H}^1\bigl((\cX_\infty^{(u)})_\proket, \fD_k^o  \bigr)^{(b)}$ be such that
$\rho(z_0)|_{(\cX_\infty^{(u)})_\proket}=p^{s+d+1}\alpha^{n+u+1}\bigl(x'-p^{r/2} x_1  \bigr)$.
Such $x_1$ exists indeed because  $x_1\in p^{-s-d-1+r/2}\alpha^{-n-u-1}{\rm H}^1\bigl((\cX_\infty^{(u)})_\proket, \fD_k^o  \bigr)^{(b)}\subset {\rm H}^1\bigl((\cX_\infty^{(u)})_\proket, \fD_k^o  \bigr)^{(b)}$.

Now we apply to $x_1$ the {\bf Step 1} we used on $x$ and obtain $z_1\in  {\rm H}^1\bigl(\cX_\proket, \bD^o_k\bigr)^{(b)}\widehat{\otimes} \cO_{\C_p}$ such that $\rho(z_1)\equiv p^{s+d+1}\alpha^{n+u+1}x_1
\Bigl({\rm mod}\ p^r {\rm H}^1\bigl((\cX_\infty^{(u)})_\proket, \fD_k^o  \bigr)^{(b)}\Bigr)$.

We remark that $\rho(z_0+p^{r/2}z_1)|_{(\cX_\infty^{(u)})_\proket}\equiv p^{s+d+1}\alpha^{n+u+1}x'  \ {\Bigl(\rm mod}\ p^{3r/2}{\rm H}^1\bigl((\cX_\infty^{(u)})_\proket, \fD_k^o  \bigr)^{(b)}\Bigr)$.

Repeating {\bf Step 2} we construct inductively a sequence $(z_m)_{m\ge 0}$ in ${\rm H}^1\bigl(\cX_\proket, \bD^o_k\bigr)^{(b)}\widehat{\otimes} \cO_{\C_p}$ such that for every $m\ge 0$ we have
$$
\rho\bigl(\sum_{j=0}^mp^{jr/2}z_j  \bigr)|_{(\cX_\infty^{(u)})_\proket}\equiv p^{s+d+1}\alpha^{n+u+1}x
\ {\Bigl(\rm mod}\ p^{(2m+1)r/2}{\rm H}^1\bigl((\cX_\infty^{(u)})_\proket, \fD_k^o  \bigr)^{(b)}\Bigr).
$$
If we denote by $$ z:=p^{-s-d-1}\alpha^{-n-u-1}\sum_{j=0}^\infty p^{jr/2}z_j\in
{\rm H}^1\bigl(\cX_\proket, \bD^o_k\bigr)^{(b)}\widehat{\otimes} \C_p,
$$
 then $\rho\bigl(p^{-d}z\bigr)|_{(\cX_\infty^{(u)})_\proket}=p^{-d}\bigl((x')^{\rm tf}\bigr)=x$. To see that the written series converges we recall (see \cite{EichlerShimura}) that ${\rm H}^1\bigl(\cX_\proket, \bD^o_k\bigr)$ is profinite, therefore compact and so ${\rm H}^1\bigl(\cX_\proket, \bD^o_k\bigr)^{(b)}\widehat{\otimes} \C_p$ is complete for the $p$-adic topology.

\medskip
\noindent
{\bf Step 3.}   Using Step (2)  Theorem \ref{thm:htsurjective}  is proven under the condition that $\Bigl(\prod_{i=0}^{h-1}(u_k-i)\Bigr)\in B[1/p]^\ast$. As no $u_k-i$ is $0$ in $B$ due to our assumptions, we have $B\bigl[p^{-1}\prod_{i=b}^{h-1}(u_k-i)^{-1} \bigr]=\cup_{n\in \N} B_n[p^{-1}]$, with $B_n=B[\![p^n/\prod_{i=b}^{h-1}(u_k-i) ]\!]$ satisfying the requirements of Definition \ref{def:weights} for every $n\ge 0$. The Theorem then holds for each $B_n$. Since the map $\Psi_H$ is a map of finite and projective $B[1/p]$-modules, we get that the Theorem holds after inverting $\prod_{i=b}^{h-1}(u_k-i)$. 

Consider the reduction of $\Psi_{\rm HT}$  modulo $u_k-i$, for  $b\leq  i\leq h-1$. It is compatible with the classical $p$-adic Hodge-Tate decomposition, which provides a surjective map
$$\mathrm{H}^1\bigl(\cX_{\overline{K},\proket},  \Symm^{i}(T_0^\vee)\bigr)\otimes \C_p  \lra \mathrm{H}^0\bigl(\cX, \omega^{i+2}\bigr);$$see \cite{EichlerShimura}. The map $\mathrm{H}^1\bigl(\cX_{\overline{K},\proket}, \bD_i^o(T_0^\vee)[n]\bigr)\to \mathrm{H}^1\bigl(\cX_{\overline{K},\proket},  \Symm^{i}(T_0^\vee)\bigr)$ is induced by  the surjective map $\bD_k^o(T_0^\vee)[n]/(u_k-i)\cong \bD_i^o(T_0^\vee)[n] \to \Symm^{i}(T_0^\vee)$ and induces a surjective morphism on the slope $b$-part $\mathrm{H}^1\bigl(\cX_{\overline{K},\proket}, \_\bigr)^{(b)}$ by results of Stevens. The map $\mathrm{H}^0\bigl(\cX, \omega^{i+2}\bigr)\to \mathrm{H}^0\bigl(\cX_\infty^{(u)}, \omega^{i+2}\bigr)\cong \mathrm{H}^0\bigl(\cX_\infty^{(u)}, \omega^{k+2}\bigr)/(u_k-i)$ is the restriction map and is an isomorphism on the slope $\leq b$-part by the classicity result of Coleman. We conclude that $\Psi_{\rm HT}$ is surjective modulo $u_k-i$ for  every $b\leq  i\leq h-1$. This implies that $\Psi_{\rm HT}$ is surjective if $\Bigl(\prod_{i=0}^{b-1}(u_k-i)\Bigr)\in B[1/p]^\ast$ as claimed.

\section{The $B_{\rm dR}$-comparison.}

In this section we consider the modular curve $\cX=\cX(N)$ of full level $N$ and the modular curve $\cX_0(p^n,N)$ over the ring of integers $\cO_K$ of an unramified extension $K$ of $\Q_p$. Let $\widehat{X}$ be the $p$-adic  formal scheme  over $\cO_K$ defined by formally completing the modular curve $X(N)$ over $\cO_K$ along the special fiber $X(N)_0$. Let $E$ be the universal generalised elliptic curve over $X(N)$ and denote its modulo $p$-reduction by  $E_0$. As $X(N)_0$ is smooth, the crystalline cohomology $\mathrm{H}^1_{\rm crys}(E_0/\widehat{X})\bigr) $ is identified with the de Rham cohomology $\mathrm{H}_E:=\mathrm{H}^1_{\rm dR}\bigl(E/\widehat{X}\bigr) $ (as a module with log connection; see below its description).

\subsection{Period sheaves.}
There is a map of sites $w\colon \cX_{\proket} \to \widehat{X}_{\rm et}$ defined by associating to an \'etale map $U\to \widehat{X}$ of formal schemes its adic generic fiber. We set $\cO_{\cX_{\proket}}^{\rm unr, +}:=w^{-1}\bigl(\cO_{\widehat{X}}\bigr)$ (see \cite[\S 2.2]{TT}). It is a subsheaf of $\cO_{\cX_{\proket}}^+$. As $\widehat{X}$  is endowed with the log structure defined by the cusps, this induces a log structure $\widehat{\gamma}\colon  \widehat{\cM} \to \cO_{\widehat{X}}$; more precisely, $ \widehat{\cM}$ is the inverse limit $\displaystyle{\lim_{\infty \leftarrow n}} \cM_n$ where $\cM_n\to \cO_{\widehat{X}}/p^n \cO_{\widehat{X}}$ is the log structure defined by the cusps. This induces a log structure $\gamma^{\rm unr}\colon \cM^{\rm unr, +}:=w^{-1}\bigl( \widehat{\cM} \bigr) \to \cO_{\cX}^{\rm unr, +}$ and   a log structure $\gamma^+\colon \cM^+ \to \cO_{\cX_{\proket}}^+$ which defines  the log structure   $\gamma\colon \cM\to \cO_{\cX_{\proket}}$ of \S \ref{sec:periodsheaves}. In particular, we can refine the pre-log structure   $\gamma^\flat \colon \cM^\flat \to \widehat{\cO}_{\cX_{\proket}^\flat}$ of loc.~cit.~to a pre-log structure  $\gamma^{\flat,+} \colon \cM^{\flat,+} \to \widehat{\cO}_{\cX_{\proket}^\flat}^+$ where $\cM^{\flat,+}$ is the inverse limit $\displaystyle \lim_{\leftarrow } \cM^+$, indexed by $\N$, with transition maps given by raising to the $p$th power. We write $a\mapsto a^\sharp$ for the first projection $\cM^{\flat,+}\to \cM^+$.

Recall that we have the period sheaf $\mathbb{A}_{\rm inf}$ over $\cX_{\proket}$, see \S  \ref{sec:periodsheaves}; here we omit $\cX$ from tha notation.  We define the morphism of multiplicative monoids $\gamma_{\rm inf}\colon \cM^{\flat,+} \to  \mathbb{A}_{\rm inf}$ by composing $\gamma^{\flat,+}$ with the Teichm\"uller lift. Then the map $\vartheta\colon  \mathbb{A}_{\rm inf} \to  \widehat{\cO}_{\cX_{\proket}}^+$ is compatible with pre-log structures, namely $\vartheta \circ \gamma_{\rm inf}$ coincides with the first projection $\cM^{+,\flat} \to \cM^+$ composed with $\gamma^+$ and the natural map $\cO_{\cX_{\proket}}^+ \to  \widehat{\cO}_{\cX_{\proket}}^+$.  The map $\vartheta$ defines a map $$\vartheta_\cX:=1\otimes \vartheta\colon \cO \mathbb{A}_{\rm inf}:=\cO_{\cX_{\proket}}^{\rm unr, +} \otimes_{\Z_p} \mathbb{A}_{\rm inf} \lra \widehat{\cO}_{\cX_{\proket}}^+$$of $\cO_{\cX_{\proket}}^{\rm unr, +}$-algebras. Furthermore , $\vartheta_\cX$ is a map of sheaves compatible with the pre-log structures $\gamma^{\rm unr}\times  \gamma_{\rm inf}\colon  \cM^{\rm unr, +} \times \cM^{\flat,+} \to \cO \mathbb{A}_{\rm inf}$ and $\gamma^+$; namely $\vartheta_\cX \circ \bigl(\gamma^{\rm unr}\times  \gamma_{\rm inf}\bigr) $ coincides with $\gamma^+$ composed with the homomorphism of monoids $ \cM^{\rm unr, +} \times \cM^{\flat, +} \to \cM^+$ provided by the natural map $\cM^{\rm unr, +}\to \cM^{+}$, the projection $\cM^{\flat,+} \to \cM^+ $ given by $a\mapsto a^\sharp$
 and the multplication map $\cM^+\times \cM^+ \to \cM^+$.

Define $\cO \mathbb{A}_{\rm log}$ to be the sheaf on $\cX_{\proket}$ given by the $p$-adic completion of the log-DP enevelope of $\cO \mathbb{A}_{\rm inf}$ with respect to  the product pre-log structure  $\gamma^{\rm unr}\times  \gamma_{\rm inf}$  and with respect to $\vartheta_\cX$; see \cite[Lemma 2.16]{andreatta_iovita_ss} for the definition. More precisely, let $\cM' \subset  (\cM^{\rm unr, +})^{\rm gp} \times (\cM^{\flat,+})^{\rm gp} $ be the sheaf of monoids defined as the  inverse image of $\cM^+\subset (\cM^+)^{\rm gp}$ via the map $(\gamma^{\rm unr})^{\rm gp}\times ( \gamma_{\rm inf})^{\rm gp}\colon (\cM^{\rm unr, +})^{\rm gp} \times (\cM^{\flat,+})^{\rm gp}\to (\cM^+)^{\rm gp}$ associated to $\gamma^{\rm unr}\times  \gamma_{\rm inf}$; here the superscript ${\rm gp}$ is the sheaf of groups associated to a sheaf of monoids. Let $\cO \mathbb{A}_{\rm inf}':= \cO \mathbb{A}_{\rm inf}\otimes_{\Z\bigl[\cM^{\rm unr, +} \times \cM^{\flat,+}\bigr]} \Z\bigl[\cM'\bigr]  $   be the log envelope of $\cO_{\cX}^{\rm unr +} \otimes_{\Z_p} \mathbb{A}_{\rm inf}$, with respect to pre-log structure $\cM^{\rm unr, +} \times \cM^{\flat,+}$ and the map $\vartheta_\cX$.  It is endowed with a tautological pre-log structure $\gamma'\colon \cM'\to \cO \mathbb{A}_{\rm inf}'$.  The map $\vartheta_\cX$ extends uniquely to a map $\vartheta_\cX'\colon \cO \mathbb{A}_{\rm inf}' \lra \widehat{\cO}_{\cX_{\proket}}^+$ compatible with the pre-log structures $\gamma'$ and $\gamma^+$, i.e., such that $\vartheta_\cX' \circ \gamma'$ is $\gamma^+$ composed with the natural morphism $\cM'\to \cM^+$ induced by $(\gamma^{\rm unr})^{\rm gp}\times ( \gamma_{\rm inf})^{\rm gp}$. Let $\mathcal{I}$ be the kernel of $\vartheta_\cX'$.  

Then $\cO\mathbb{A}_{\rm log}$ is the $p$-adic completion of the DP envelope of $\cO \mathbb{A}_{\rm inf}'$ with respect to the ideal $\cI$.    For every positive integer $n$ we also define $\cO\mathbb{A}_{\max,n}^{\log}$ to be the $p$-adic completion of the subsheaf $\displaystyle \cO \mathbb{A}_{\rm inf}'  \bigl[\frac{\mathcal{I}}{p^n}\bigr]$ of $\cO\mathbb{A}_{\rm inf}'[p^{-1}]$. The map $\vartheta_{\cX}'$ extends to a map $\vartheta_{\max,n}\colon \cO\mathbb{A}_{\max,n}^{\log}\to \widehat{\cO}_{\cX_{\proket}}^+$. As  $\mathcal{I}$ admits DP powers in  $\cO \mathbb{A}_{\max,n}^{\log}$ and $\cO\mathbb{A}_{\log}$ is the $p$-adic completion of the  DP envelope of $\cO \mathbb{A}_{\rm inf}'$ with respect to $\mathcal{I}$, we have a natural map $\cO \mathbb{A}_{\log}\to \cO\mathbb{A}_{\max,n}^{\log}$ by the universal property of the DP envelope.

The derivation $d\colon \cO_{\widehat{X}} \to \Omega^{\log}_{{\widehat{X}}/\cO_K}$ defines an $\mathbb{A}_{\rm inf}$-linear connections $$\nabla\colon \cO \mathbb{A}_{\rm log}\lra \cO \mathbb{A}_{\rm log}\otimes\otimes_{\cO^{\rm unr,+}_{\cX_{\proket}}}  \Omega^{\rm unr}_{\cX_{\proket}}$$and $$\nabla\colon \cO \mathbb{A}_{\max,n}^{\log}\lra \cO \mathbb{A}_{\max,n}^{\log}\otimes\otimes_{\cO^{\rm unr,+}_{\cX_{\proket}}}  \Omega^{\rm unr}_{\cX_{\proket}} ,$$where $\Omega^{\rm unr}_{\cX_{\proket}}:=w^{-1}\bigl(\Omega^{\log}_{{\widehat{X}}/\cO_K}\bigr)$ is the inverse image of the logarithmic differentials on $\widehat{X}$.

\begin{remark} On the complement of the cusps, where the log structure is trivial,   the sheaf $\cO \mathbb{A}_{\rm log}$ is  the $p$-adic completion of the DP envelope of $\cO \mathbb{A}_{\rm inf}$ with respect to the kernel of $\vartheta_\cX$; it is the sheaf denoted $\cO \mathbb{A}_{\rm cris}$ in \cite[Def. 2.9]{TT}.

Take an affine open nieghorhhod $\widehat{U}_0\subset \widehat{X}$  of a cusp,  with local coordinate $Y$ at the given cusp. Define $U_0\subset \cX$ to be the associated affinoid with induced log structure. Take  $U=\lim_i U_i$ with $U_i=\bigl(\mathrm{Spa}(R_i,R_i^+),\cM_i\bigr)$ to be a log affinoid perfectoid with initial object $U_0$ and let $(R,R^+)$ be the $p$-adic completion of $\lim_i (R_i,R_i^+)$. By assumption we have a compatible systm of $p^n$-th roots  $Y_n$ of $Y$ so that the system $\underline{Y}:=[Y,Y_1,Y_2,\cdots ]\in \gamma^{\flat,+}\bigl(\cM^{\flat,+}\bigr)(U)$.  As shown in \cite[Lemma 3.25]{andreatta_iovita_ss}, we have $\cO \mathbb{A}_{\rm log}(U)= \mathbb{A}_{\rm cris}(R,R^+) \{\langle w-1 \rangle\}$  the $p$-adic completion of the DP algebra $\mathbb{A}_{\rm cris}(R,R^+) \langle w-1 \rangle$,  where $\displaystyle w=\frac{Y}{[\underline{Y}]}$. Here, $\mathbb{A}_{\rm cris}$ is the $p$-adic completion of the DP envelope of $ \mathbb{A}_{\rm inf}$ with respect to the kernel of $\vartheta$. There is a similar description for  $\cO \mathbb{A}_{\max,n}^{\log}(U)$.
\end{remark}

There are also the geometric de Rham sheaves $\cO\mathbb{B}_{\rm dR, log}^+$ and  $\cO\mathbb{B}_{\rm dR, log}$ defined in \cite[Def. 2.2.10]{DLLZ}, with a map $\vartheta_{\rm dR,log}\colon \cO\mathbb{B}_{\rm dR, log}^+ \to  \widehat{\cO}_{\cX_{\proket}}^+$ and logarithmic connectons $\cO\mathbb{B}_{\rm dR, log}^+\lra \cO\mathbb{B}_{\rm dR, log}^+ \otimes\otimes_{\cO^{\rm unr,+}_{\cX_{\proket}}}  \Omega^{\rm unr}_{\cX_{\proket}}$ (see \cite[\S 2.2]{DLLZ}). As  $\cO\mathbb{B}_{\rm dR, log}^+$ is a $ \cO_{{\cX}_\proket}\otimes  \mathbb{A}_{\rm inf}$-algebra by construction, there is a natural map of sheaves $\cO \mathbb{A}_{\rm inf}\to \cO\mathbb{B}_{\rm dR, log}^+$ whose composite with $\vartheta_{\rm dR,log}$ is $\vartheta_\cX$. 

\begin{lemma}\label{lemma:BdR} The map  $\cO \mathbb{A}_{\rm inf}\to \cO\mathbb{B}_{\rm dR, log}^+$ extends to morphisms $\cO \mathbb{A}_{\log}\to \cO\mathbb{A}_{\max,n}^{\log}\to \cO\mathbb{B}_{\rm dR, log}^+$, for every $n$, such that the composite with $\vartheta_{\rm dR,log}$ is the map  $\vartheta_{\max,n}$ and it is compatible with connections.

\end{lemma} 
\begin{proof} First of all we show how to get a map  $\cO \mathbb{A}_{\rm inf}'\to \cO\mathbb{B}_{\rm dR, log}^+$ of $\cO \mathbb{A}_{\rm inf}$-algebras. 

It suffices to construct this for log affinoid perfectoid objects of $\cX_{\proket}$ arising from toric charts,  as those form a basis  of $\cX_{\proket}$. Take any such object that we denote by $U$. 
It follows from \cite[Lemma 2.3.12]{DLLZ} that there exists a morphism of monoids  $ \beta\colon  \cM^{\flat,+}(U)^\ast\to \cO\mathbb{B}_{\rm dR, log}^+(U)$ such that for every $a \in \cM^{\flat,+}(U)$ we have $\gamma^+(a^\sharp)=\bigl[\gamma^{\flat,+}(a)\bigr]\beta(a)$. Since the map $\cM^{\flat,+}(U)\to \cM^+(U)$ is an isomorphism,  any element in the kernel $H$ of  $(\gamma^{\rm unr})^{\rm gp}(U)\times ( \gamma_{\rm inf})^{\rm gp}(U)\colon (\cM^{\rm unr, +})^{\rm gp}(U) \times (\cM^{\flat,+})^{\rm gp}(U)\to (\cM^+)^{\rm gp}(U)$ is of the form $\bigl(a^\sharp (b^\sharp)^{-1}, a b^{-1}\bigr)$ with $a,b\in \cM^\flat(U)$ such that $a^\sharp, b^\sharp\in \cM^{\rm unr, +}(U)$. Any such element can be sent to $\beta(a)\beta(b)^{-1}$ and this defines a group homomorphism $\beta\colon H\to  \cM^{\flat,+}(U)^\ast$. As $\cM'(U)= H \cdot \bigl(\cM^{\rm unr, +}(U) \times \cM^{\flat, +} (U)\bigr)\subset (\cM^{\rm unr, +})^{\rm gp}(U) \times (\cM^{\flat,+})^{\rm gp}(U)\to (\cM^+)^{\rm gp}(U)$, such map extends to a map of monoids $\beta\colon \cM'(U)\to \cO\mathbb{B}_{\rm dR, log}^+(U)$ whcih is compatible with the map $\gamma^{\rm unr}(U)\times  \gamma_{\rm inf}(U)\colon \cM^{\rm unr, +}(U) \times \cM^{\flat,+}\to \cO \mathbb{A}_{\rm inf}(U) \to \cO\mathbb{B}_{\rm dR, log}^+(U)$. Using this we get a unique morphism   $\cO \mathbb{A}_{\rm inf}'(U) \to \cO\mathbb{B}_{\rm dR, log}^+(U)$ of $\cO \mathbb{A}_{\rm inf}(U)$-algebras coinciding with $\beta$ on $\cM'(U)$.

In this way we get the claimed morphism $\cO \mathbb{A}_{\rm inf}'\to \cO\mathbb{B}_{\rm dR, log}^+$  whose composite with $\vartheta_{\rm dR,log}$ is $\vartheta_\cX'$. In particular, the kernel  $\mathcal{I}$ of $\vartheta_\cX'$ is mapped to the kernel of $\vartheta_{\rm dR,log}$. As $p$ is invertible in $\cO\mathbb{B}_{\rm dR, log}^+$, it also extends to a morphism  $\displaystyle \cO \mathbb{A}_{\rm inf}'  \bigl[\frac{\mathcal{I}}{p^n}\bigr]\to \cO\mathbb{B}_{\rm dR, log}^+$ and $\displaystyle \frac{\mathcal{I}}{p^n}$ maps  to the kernel of  $\vartheta_{\rm dR,log}$. Passing to log affinoid perfectoid objects of $\cX_{\proket}$ one shows that such map extends to the $p$-adic completion $\displaystyle \cO \mathbb{A}_{\max,n}^{\log}$ of $\displaystyle \cO \mathbb{A}_{\rm inf}'  \bigl[\frac{\mathcal{I}}{p^n}\bigr]  \to\cO\mathbb{B}_{\rm dR, log}^+$. The compatibility with   $\vartheta_{\max,n}$ is clear. The compatibility of the connections is also clear as both are defined using the derivation $d\colon \cO_{\widehat{X}} \to \Omega^{\log}_{{\widehat{X}}/\cO_K}$.

\end{proof}

\subsection{Crystalline comparison morphisms}\label{sec:criscomp}

We have a crystalline comparison morphism over $\cX_{\proket}$:  

\begin{equation} \alpha_{\rm log}\colon (T_p (E))^\vee \lra \left(\mathrm{H}^1_{\rm crys}\bigl(E_0/\widehat{X}\bigr) \otimes_{\cO_{\cX}^+ }  \cO \mathbb{A}_{\rm log}\right)^{\nabla'=0},
\end{equation}
where $(T_p (E))^\vee$ is the $\Z_p$-dual of $T_p (E)$ and $\nabla'$ is the natural connection on 
${\rm H}^1_{\rm crys}\bigl(E_0/\widehat{X}\bigr)\otimes \cO\mathbb{A}_{\rm log}$ determined by the connections on the factors.
We also write $$\bigl(\mathrm{H}^1_{\rm crys}(E_0/\widehat{X}) \otimes_{\cO_{\cX}^+ }  -\bigr)   \mbox{ for } \Bigl(w^{-1}\bigl(\mathrm{H}^1_{\rm crys}(E_0/\widehat{X})\bigr) \otimes_{\cO_{\cX}^{\rm unr +} } -\Bigr)$$ to ease the notation. Ast we work over $K$ we have to keep track of the action of ${\rm Gal}\bigl(\overline{\Q}_p/K\bigr)$.  

We describe $\alpha_{\rm log}$ for a log affinoid perfectoid open cover of $\cX$. Consider an \'etale open $U=\mathrm{Spf}(S)\subset \widehat{X}$. Let $W=\mathrm{Spa}(R,R^+)$  be a log affinoid perfectoid cover of $U_{\overline{\Q}_p}$,  the adic geometric generic fiber of $U$. We assume that the universal elliptic curve extends to a (generalised) elliptic curve $\widetilde{E}$ over $\Spec (S)$ and that $T_p(E)$ is trivialized over $W$. Consider the ring $A_{\rm log}(R^+)$ defined by taking the $p$-adic completion of the log  DP enevelope of $S\otimes \mathbb{A}_{\rm inf}(W)$  with respect to the kernel of the map $1\otimes \vartheta(W)\colon S\otimes \mathbb{A}_{\rm inf}(W)\to \widehat{\cO}_\cX^+(W)=R^+$. It naturally maps to $\cO \mathbb{A}_{\rm log}(W)$. We define $\alpha_{\rm cris}(W)$ with values in $\mathrm{H}^1_{\rm crys}\bigl(\widetilde{E}_0/S\bigr) \otimes_S A_{\rm log}(R^+)$ as follows. 

\

{\em Away from the cusps:} \enspace Assume first that $\widetilde{E}$ is an elliptic curve so that $\widetilde{E}[p^n](R^+)=E[p^n](R)$ for every $n\in\ N$ and $T_p (\widetilde{E})(R^+)=T_p(E)(R)$. Equivalently the log structure is trivial and $A_{\rm log}(R^+)=A_{\rm cris}(R^+)$. Then to give $a \in T_p(E^\vee)(R)=T_p \bigl(E\bigr)^\vee(R^+)(1)$ is equivalent to give a map of $p$-divisible groups $\gamma_a\colon \Q_p/\Z_p \to E^\vee[ p^\infty]$ over $R^+$ and $\gamma_a$ defines  a map of covariant Dieudonn\'e modules $${\bf D}_{\rm cris}(\gamma_a)\colon {\bf D}_{\rm cirs}(\Q_p/\Z_p )\bigl(A_{\rm cris}(R^+) \bigr)\to {\bf D}_{\rm cris}\bigl(E^\vee[ p^\infty]\bigr)\bigl(A_{\rm cris}(R^+)\bigr) .$$Note that  $A_{\rm cris}(R^+)={\bf D}_{\rm cris}(\Q_p/\Z_p )\bigl(A_{\rm cris}(R^+)\bigr)$ and ${\bf D}_{\rm cris}\bigl(\widetilde{E}^\vee[ p^\infty]\bigr)\bigl(A_{\rm cris}(R^+)\bigr)=\mathrm{H}^1_{\rm crys}\bigl(E^\vee_0/S\bigr)^\vee\otimes_S A_{\rm cris}(R^+) $, where $E^\vee_0$ is the modulo $p$ reduction of $E^\vee$. Thus ${\bf D}_{\rm cris}(\gamma_a)$ defines a map $${\bf D}_{\rm cris}(\gamma_a)\colon A_{\rm cris}(R^+)\lra \mathrm{H}^1_{\rm crys}\bigl(E^\vee_0/S\bigr)^\vee\otimes_S A_{\rm cris}(R^+)  $$and  setting $$\alpha_{\rm log}(W)(a)=\alpha_{\rm cris}(W)(a):={\bf D}_{\rm cris}(\gamma_a)(1)$$(see \cite{ScholzeWeinstein}{\S 3.5}) and using Weil, respectively Poincar\'e dualities gives the map  $$\alpha_{\rm log}(W): T_p(E)^\vee(R^+)\lra {\rm H}^1_{\rm crys}\bigl(E_0/S)\otimes_SA_{\rm cris}(R^+).$$
 As ${\bf D}_{\rm cris}(\gamma_a)$ is a map of crystals, it is compatible with connections. Since $\nabla(1)=0$ for $1\in A_{\rm cris}(R^+)$, then $\nabla'\Bigl( {\rm Im}\bigl(\alpha_{\rm cris}(W)\bigr)\Bigr)=0$. 

\

{\em Around the cusps:}\enspace Assume next that $U$ does not contain supersingular points. Then the connected part $\widetilde{E}[p^\infty]^0$ of $\widetilde{E}[p^\infty]$ is a $p$-divisble group of multiplicative type. Let  $\widetilde{E}[p^\infty]^{0,\vee}$ be its Cartier dual; it is an \'etale $p$-divisible group over $S$. 
Moreover, $T_p\bigl(\widetilde{E}[p^\infty])(R)$ is a (split over $R^+$ but not split over $S$) extension of $T_p\bigl(\widetilde{E}[p^\infty]^{0,\vee}\bigr)(R^+)$ by $T_p\bigl(\widetilde{E}[p^\infty]^{0}\bigr)(R^+)$. We write $\mathrm{H}^1_{\rm crys}(E_0)$ for the direct sum $\omega_{E/S}^{-1} \cong {\bf D}_{\rm cris}\bigl(\widetilde{E}[ p^\infty]^0\bigr)(S)$ and $\omega_{E/S}\cong {\bf D}_{\rm cris}\bigl(\widetilde{E}[ p^\infty]^{0,\vee}\bigr)$. The connection $\mathrm{H}^1_{\rm crys}(E_0)\to \mathrm{H}^1_{\rm crys}(E_0)\otimes_S \Omega^{1,\log}_{S/\cO_K}$ is the sum of the connections on ${\bf D}_{\rm cris}\bigl(\widetilde{E}[ p^\infty]^0\bigr)(S)$ and ${\bf D}_{\rm cris}\bigl(\widetilde{E}[ p^\infty]^{0,\vee}\bigr)$ plus the $S$-linear isomorphism $\omega_{E/S} \cong \omega_{E/S}^{-1} \otimes_S \Omega^{1,\log}_{S/\cO_{K}}$ provided by the Kodaira-Spencer isomorphism, which we denote ${\rm KS}$.  In other words as a module ${\bf D}_{\rm cris}(\widetilde{E}[p^\infty])$ is isomorphic to the direct sum ${\bf D}_{\rm cris}(\widetilde{E}[p^\infty]^0)\oplus {\bf D}_{\rm cris}(\widetilde{E}[p^\infty]^{0,\vee})\cong \omega_{E/S}\oplus \omega^{-1}_{E/S}$ while the connection 
$\nabla_{\widetilde{E}[p^\infty]}$ is given with respect to this decomposition by $$\left( \begin{array}{cc} \nabla_{\widetilde{E}[p^\infty]^0} &  0 \\ {\rm KS} & \nabla_{\widetilde{E}[p^\infty]^{0,\vee}}\end{array}\right).$$

Then we let $\alpha_{\rm log}(W)$ be the direct sum of the crystalline maps for the $p$-divisible groups $\widetilde{E}[p^\infty]^0$ and $\widetilde{E}[p^\infty]^{0,\vee}$ via this identification for a given splitting $T_p\bigl(\widetilde{E}[p^\infty])(R)=T_p\bigl(\widetilde{E}[p^\infty]^{0}\bigr)(R^+)\oplus T_p\bigl(\widetilde{E}[p^\infty]^{0,\vee}\bigr)(R^+)$.

\

The following lemma implies that these two constructions of $\alpha_{\rm log}$ coincide when they are both defined. In particular, $\alpha_{\log}(W)$ is functorial in the pair $(U,W)$ and descends to a map $\alpha_{\rm log}$. 

\begin{lemma} If $U$ does not contain supersingular points and $\widetilde{E}$ is an elliptic curve, the two constructions coincide.

\end{lemma}

\begin{proof} Notice that  the quotient map $\widetilde{E}^\vee[p^\infty] \to \widetilde{E}^\vee[p^\infty]^{\rm et} $ onto the \'etale part induces a map from  $\mathrm{H}^1_{\rm crys}\bigl(\widetilde{E}_0/S\bigr)$ to the Dieudonn\'e module of $\widetilde{E}^\vee[p^\infty]^{\rm et}(S)$  which splits canonically via the unit root decomposition. As $\widetilde{E}^\vee[p^\infty]^{\rm et}\cong \widetilde{E}[p^\infty]^{0,\vee}$ and similarly $\widetilde{E}^\vee[p^\infty]^{0}\cong \widetilde{E}[p^\infty]^{\rm et,\vee}\cong \widetilde{E}[p^\infty]^0$ we conclude that $\mathrm{H}^1_{\rm crys}\bigl(\widetilde{E}_0/S\bigr)$ splits canoncially, as a module, as the direct sum of $\omega_{E/S}\cong {\bf D}_{\rm cris}\bigl(\widetilde{E}^\vee[ p^\infty]^0\bigr)(S)$ and $\omega_{E/S}^{-1}\cong {\bf D}_{\rm cris}\bigl(\widetilde{E}[ p^\infty]^{0,\vee}\bigr)(S)$, with connection given by the connnection on each factor plus the Kodaira-Spencer isomorphism.  A splitting $T_p\bigl(\widetilde{E}[p^\infty])(R)=T_p\bigl(\widetilde{E}[p^\infty]^{0}\bigr)(R^+)\oplus T_p\bigl(\widetilde{E}[p^\infty]^{0,\vee}\bigr)(R^+)$ uniquely defines a splitting $\widetilde{E}[p^\infty]_{R^+}=\widetilde{E}[p^\infty]^{0,\vee}_{R^+}\oplus \widetilde{E}[p^\infty]^{0}_{R^+}$ of the connected-\'etale exact sequence for $\widetilde{E}[p^\infty]_{R^+}$. 
Since the crystalline comparison map is functorial in the $p$-divisible group, using the given splitting of $p$-divisible groups, the claim follows.

\end{proof}

We also consider  the composite map of sheaves  $$\alpha_{\rm max,n}\colon (T_p (E))^\vee  \lra \left( \mathrm{H}^1_{\rm crys}(E_0/\widehat{X})  \otimes_{\cO_{\cX}^{\rm unr +} }   \cO \mathbb{A}_{\rm max,n}^{\rm log}\right)^{\nabla'=0},$$induced by composing $\alpha_{\rm log}$ with the map of sheaves $\cO \mathbb{A}_{\rm log} \to \cO \mathbb{A}_{\rm max, n}^{\rm log}$ described at the beginning of the section.

\subsection{The sheaves $\mathbf{W}_{k, \rm dR}$.}

We consider strict neighborhoods $\cX\bigl(p/\mathrm{Ha}^{p^r}\bigr)$ of the ordinary locus in $\cX$ where $\mathrm{Ha}$ is a (any) local lift of the Hasse invariant. It follows from
\cite[Lemma III.3.7 \& Lemma 3.14]{ScholzeTorsion} that the neighborhhods $\cX\bigl(p/\mathrm{Ha}^{p^r}\bigr)\otimes_{K} \C_p$ and $\cX_\infty^{(m)})$ of \S \ref{rmk:standardopens} for varying $r$ respectively $m$ are  fundamental systems of open neighborhoods of the ordinay locus of $\cX_{\C_p}$. We then take $r$ and $m$ so that $\cX\bigl(p/\mathrm{Ha}^{p^r}\bigr)\subset \cX_\infty^{(m)})$  and the conclusions of Proposition \ref{proposition:existcanonical} hold. Namely, a canonical subgroup $C_n$ of order $n$ exists and this defines a section \begin{equation}\label{eq:sectionO1} \cX\bigl(p/\mathrm{Ha}^{p^r}\bigr) \to \cX_0(p^n,N)\bigl(p/\mathrm{Ha}^{p^r}\bigr) \end{equation} of the natural forgetful map $\cX_0(p^n,N)\bigl(p/\mathrm{Ha}^{p^r}\bigr)\lra \cX\bigl(p/\mathrm{Ha}^{p^r}\bigr)$. Write  
$$\nu\colon {\cI}g_n\bigl(p/\mathrm{Ha}^{p^r}\bigr)\lra \cX\bigl(p/\mathrm{Ha}^{p^r}\bigr)$$ for the $(\Z/p^n\Z)^\ast$-Galois cover classifying trivializations of $C_n^\vee$.

Let $E$ be the universal elliptic curve over the normalization of $\widehat{X}$ in ${\cI}g_n\bigl(p/\mathrm{Ha}^{p^r}\bigr)$. Its invariant differentials $\omega_E$ and relatve de Rham cohomology $\mathrm{H}_E$ define locally free $\cO_{{\cI}g_n\bigl(p/\mathrm{Ha}^{p^r}\bigr)}^+$-modules with the Hodge filtration $\omega_E\subset \mathrm{H}_E$.
Write $\underline{\delta}$ for the invertible $\cO_{{\cI}g_n\bigl(p/\mathrm{Ha}^{p^r}\bigr)}^+$-module defined by $\underline{\delta}:=\omega_E(\omega_E^{\rm mod})^{-1}$. Note that $\underline{\delta}^{p-1}= \nu^\ast({\rm Hdg})$ where ${\rm Hdg}$ is the ideal of  $\cO_{\cX\bigl(p/\mathrm{Ha}^{p^r}\bigr)}^+$ generated by the local lifts $\mathrm{Ha}$ of the Hasse invariant (due to the blow-up it does not depend on the choice of the local lifts).

From the  tautological section $\tilde{s}$ of $C_n^\vee$ we get a canonical section $s$ of $\omega_E^{\rm mod}/p^n \omega_E^{\rm mod}$ generating it as $\cO_{{\cI}g_n\bigl(p/\mathrm{Ha}^{p^r}\bigr)}^+$-module. 
We denote by $\mathrm{H}_E^\#$ the push-out in the category of $\cO_{{\cI}g_n\bigl(p/\mathrm{Ha}^{p^r}\bigr)}^+$-modules  of the diagram
$$
\begin{array}{cccccccccc}
\underline{\delta}^p\omega_E&\lra&\underline{\delta}^p\mathrm{H}_{E}\\
\cap\\
\omega_{E}^{\rm mod}.
\end{array}
$$ 
We then have the following commutative diagram of sheaves with exact rows:
$$
\begin{array}{ccccccccc}
0&\lra&\underline{\delta}^p\omega_E&\lra&\underline{\delta}^p\mathrm{H}_{E}&\lra&\underline{\delta}^p\omega_{E}^{-1}&\lra&0\\
&&\cap&&\downarrow&&||\\
0&\lra&\omega_{E}^{\rm mod}&\lra&\mathrm{H}_{E}^\#&\lra&\underline{\delta}^p\omega_{E}^{-1}&\lra&0\\
&&\cap&&\cap&&\cap\\
0&\lra&\omega_{E}&\lra&\mathrm{H}_{E}&\lra&\omega_E^{-1}&\lra&0
\end{array}
$$
It follows that $\mathrm{H}_E^\#$ is a locally free $\cO_{{\cI}g_n\bigl(p/\mathrm{Ha}^{p^r}\bigr)}^+$-module of rank two and 
 $(\omega_E^{\rm mod}, s)\subset (\mathrm{H}_E^\#, s)$ is a compatible inclusion of locally free sheaves with marked sections.  

Let $k\colon \Z_p^\ast \to B^\ast$ be an $h$-analytic character for $h\leq n$ where $B$ is a ring as in definition \ref{def:weights}. Using the formalism of the dual VBMS (see \S \ref{sec:WkWDk})  we get sheaves of   $\cO_{{\cI}g_n\bigl(p/\mathrm{Ha}^{p^r}\bigr)}^+\widehat{\otimes} B$-modules $\omega^k \subset \mathbf{W}_{k, \rm dR}$. We have a residual action of the Galois group $(\Z/p^n\Z)^\ast$ of $j_n\colon {\cI}g_n\bigl(p/\mathrm{Ha}^{p^r}\bigr)\to \cX\bigl(p/\mathrm{Ha}^{p^r}\bigr)$ on $(\omega_E^{\rm mod}, s)$ and $ (\mathrm{H}_E^\#, s)$ on which it acts by scalar multplication. We then get sheaves $\omega^k\subset \mathbf{W}_{k, \rm dR}$ of $\cO_{\cX_1\bigl(p/\mathrm{Ha}^{p^r}\bigr)}\widehat{\otimes} B$-modules by taking the subsheaves of $j_{n,\ast}(\omega^k)\subset j_{n,\ast}(\mathbf{W}_{k, \rm dR})$ on which $\Z_p^\ast$ acts via $k$.

\begin{proposition} \label{prop:propWk} The base change of $\omega_E^k[1/p]$ to $\cX\bigl(p/\mathrm{Ha}^{p^r}\bigr)_{\C_p}$ coincides with the restriction of the sheaf $\omega_E^k[1/p]$  defined in \S \ref{sec:omegak} over $\cX_0(p^n,N)_\infty^{(m)}$.

The sheaf $\mathbf{W}_{k, \rm dR}$ has a natural, increasing filtration $\bigl({\rm Fil}_n\mathbf{W}_{k, \rm dR}  \bigr)_{n\ge 0}$ such that $\omega_E^k[1/p]={\rm Fil}_0\mathbf{W}_{k, \rm dR}[1/p] $. 

The Gauss-Manin connection $\nabla\colon \mathrm{H}_E\to \mathrm{H}_E\otimes \Omega^{\log}_{\cX\bigl(p/\mathrm{Ha}^{p^r}\bigr)/K}$ induces a connection $$\nabla_{k}\colon  \mathbf{W}_{k, \rm dR}[1/p] \lra \mathbf{W}_{k, \rm dR}\widehat{\otimes} \Omega^{\log}_{\cX\bigl(p/\mathrm{Ha}^{p^r}\bigr)/K}$$ statisfying Griffiths' transversality, i.e., such that $\nabla_k\bigl({\rm Fil}_n\mathbf{W}_{k, \rm dR}[1/p] \bigr)\subset {\rm Fil}_{n+1} \mathbf{W}_{k, \rm dR}[1/p]$.

The cohomology groups $\mathrm{H}^0\bigl(\cX\bigl(p/\mathrm{Ha}^{p^r}\bigr), \mathbf{W}_{k, \rm dR}\bigr)$ and $\mathrm{H}^0\bigl(\cX\bigl(p/\mathrm{Ha}^{p^r}\bigr), \Fil_n\mathbf{W}_{k, \rm dR}\bigr)$ are endowed with an action of the $U_p$-operator and for every integer $h$ they admit slope $\leq h$ decomposition.  Furthermore we have $p \nabla_k \circ U_p=U_p \circ \nabla_k$ and  for $n\gg 0$ we have

$$\mathrm{H}^0\bigl(\cX\bigl(p/\mathrm{Ha}^{p^r}\bigr), \Fil_n \mathbf{W}_{k, \rm dR}\bigr)^{(h)}=\mathrm{H}^0\bigl(\cX\bigl(p/\mathrm{Ha}^{p^r}\bigr), \mathbf{W}_{k, \rm dR}\bigr)^{(h)}.$$

\end{proposition}
\begin{proof} The first statement follows from the fact that the two constructions coincide on $\cX\bigl(p/\mathrm{Ha}^{p^r}\bigr)_{\C_p}$. The other statements are proven in
 \cite{andreatta_iovita_triple}.

\end{proof}

\subsection{The de Rham comparison map. }
\label{sec:derhamcomp}

Fix an $s$-analytic weight $k\colon \Z_p^\ast \to B^\ast$ as in definition \ref{def:weights}. Let $B_{\rm dR}^+$ and $B_{\rm dR}=B_{\rm dR}^+[t^{-1}]$ be the classical period rings of Fontaine with the canonical topology, so that for example the quotient topology on  $B^+_{\rm dR}/t B^+_{\rm dR}=\C_p$ is the $p$-adic topology on $\C_p$. They are endowed with filtrations such that $\Fil^i B_{\rm dR}=t^i B_{\rm dR}^+$ for every $i\in \Z$.   In this section we use the map $\alpha_{\rm log}$ to get the following result. 

\begin{theorem}\label{thm:deRcompare} We have a Hecke equivariant, $B\widehat{\otimes} B_{\rm dR}^+$-linear, $\mathrm{Gal}(\overline{K}/K)$-equivariant map
$$\rho_k\colon \mathrm{H}^1\bigl(\cX_{\overline{K},\proket}, \bD_k^o(T_0^\vee)[n] \bigr)^{(h)}\widehat{\otimes} B_{\rm dR}^+\lra \mathrm{H}^1_{\rm dR}\bigl(\cX\bigl(p/\mathrm{Ha}^{p^r}\bigr), \mathbf{W}_{k, \rm dR,\bullet}\bigr)^{(h)}\widehat{\otimes} \Fil^{-1}B_{\rm dR},$$where the completed tensor product is taken considering the canonical topology on $B_{\rm dR}^+$. Moreover,  $$\mathrm{H}^1_{\rm dR}\bigl(\cX\bigl(p/\mathrm{Ha}^{p^r}\bigr), \mathbf{W}_{k, \rm dR,\bullet}\bigr)^{(h)}\cong  \frac{\mathrm{H}^0\bigl(\cX\bigl(p/\mathrm{Ha}^{p^r}\bigr)  ,\mathbf{W}_{k+2, \rm dR}[1/p]\bigr)^{(h)}}{\nabla\left(\mathrm{H}^0\bigl(\cX\bigl(p/\mathrm{Ha}^{p^r}\bigr) ,\mathbf{W}_{k, \rm dR}[1/p]\bigr)^{ (h-1)}\right)}.$$Furthermore,  

\begin{itemize}

\item[i.] if $u_k (u_k-1) \cdots (u_k-h+1)$ is invertible in $B[1/p]$ the map $\rho_k$ is surjective;

\item[ii.]  the map $\omega_E^{k+2}\to \mathbf{W}_{k+2, \rm dR}$ induces a surjective map, which is an isomorphism under the hypothesis of (i):  $$\mathrm{H}^0\bigl(\cX\bigl(p/\mathrm{Ha}^{p^r}\bigr) , \omega^{k+2}_E[1/p]\bigr)^{(h)} \lra \mathrm{H}^1_{\rm dR}\bigl(\cX\bigl(p/\mathrm{Ha}^{p^r}\bigr), \mathbf{W}_{k, \rm dR,\bullet}[1/p]\bigr)^{(h)}.$$

\item[iii.] for specializations $B[1/p]\to \Q_p$ so that the composite weight $k_0$ is classical,   $\rho_k$ is compatible with the classical de Rham comparison map

$$\mathrm{H}^1\bigl(\cX_{\overline{K},\proket},  \Symm^{k_0}(T_p(E)^\vee)\bigr)\otimes B_{\rm dR} \cong \mathrm{H}^1_{\rm dR}\bigl(\cX, \Symm^{k_0}(\mathrm{H}_E)\bigr) \otimes_{K} B_{\rm dR}$$via the map induced by taking on the LHS the pro-Kummer \'etale cohomology via the projection $\bD_k^o(T_0^\vee)[n]\to  \bD_{k_0}^o(T_0^\vee)[n] \to \Symm^{k_0}(T_p(E)^\vee)$ and on the RHS the restriction map to the open $\cX\bigl(p/\mathrm{Ha}^{p^r}\bigr)$ $$\mathrm{H}^1_{\rm dR}\bigl(\cX, \Symm^{k_0}(\mathrm{H}_E)\bigr)\lra \mathrm{H}^0\bigl(\cX\bigl(p/\mathrm{Ha}^{p^r}\bigr) , \omega^{k_0+2}_E[1/p]\bigr)/\vartheta^{k_0+1} \mathrm{H}^0\bigl(\cX\bigl(p/\mathrm{Ha}^{p^r}\bigr),  \omega^{-k_0}_E[1/p]\bigr).$$

\end{itemize}
\end{theorem}

The proof of Theorem \ref{thm:deRcompare} will be given in Section \S\ref{sec:proof}.

\medskip
\noindent
We now consider the Galois cohomology for the group $G:=\mathrm{Gal}(\overline{K}/K)$. Recall that $\mathrm{H}^1\left(G,\Fil^{-1}B_{\rm dR}\right)=K [\log \chi]$, where $\chi$ is the cyclotomic character. More precisely we see $\log \chi$ as a $1$-cocycle $\log\chi :G\lra \Z_p\subset \Fil^{-1}B_{\rm dR}$ and we denoted $[\log \chi]$ its cohomology class.  We then obtain from Theorem \ref{thm:deRcompare} the following corollary:

\begin{corollary} We have a Hecke equivariant, $B$-linear map
$$\mathrm{Exp}^\ast_k\colon \mathrm{H}^1\left(G,\mathrm{H}^1\bigl(\cX_{\overline{K},\proket}, \bD_k^o(T_0^\vee)[n](1) \bigr)^{(h)}\right)\lra \mathrm{H}^1_{\rm dR}\bigl(\cX\bigl(p/\mathrm{Ha}^{p^r}\bigr), \mathbf{W}_{k, \rm dR,\bullet}\bigr)^{ (h)},$$called the {\em big dual exponential map}. It has the property that  for every classical weight specialization  $k_0$ it is compatible with the classical dual exponentail map as follows:

a) If $k_0>h-1$, i.e. $k_0$ is a non-crtical weight for the slope $h$, then we have the following commutative diagram with horizontal isomorphisms. Here we denoted by $\exp^\ast_{k_0}$ the Kato dual exponential map associated to weight $k_0$ modular forms.
$$
\begin{array}{ccccccccc}
\Bigl(\mathrm{H}^1\left(G,\mathrm{H}^1\bigl(\cX_{\overline{K},\proket}, \bD_k^o(T_0^\vee)[n](1) \bigr)^{(h)}\right)\Bigr)_{k_0}&\stackrel{\bigl(\mathrm{Exp}^\ast_k\bigr)_{k_0}}{\lra} &\Bigl(\mathrm{H}^1_{\rm dR}\bigl(\cX\bigl(p/\mathrm{Ha}^{p^r}\bigr), \mathbf{W}_{k, \rm dR,\bullet}\bigr)^{ (h)}\Bigr)_{k_0}\\
\downarrow \cong&&\downarrow\cong\\
\mathrm{H}^1\left(G,\mathrm{H}^1\bigl(\cX_{\overline{K},\proket},  \Symm^{k_0}(T_p(E)^\vee)(1)\bigr)^{(h)}\right)&\stackrel{\exp^\ast_{k_0}}{\lra} &\Fil^0\mathrm{H}^1_{\rm dR}\bigl(\cX, \Symm^{k_0}(\mathrm{H}_E)\bigr)^{(h)}
\end{array}
$$

b) If \ $0\le k_0\le h+1$, i.e. $k_0$ is critical with respect to $h$, we only have a commutaive diagram of the form

 $$
\begin{array}{ccccccccc}
\Bigl(\mathrm{H}^1\left(G,\mathrm{H}^1\bigl(\cX_{\overline{K},\proket}, \bD_k^o(T_0^\vee)[n](1) \bigr)^{(h)}\right)\Bigr)_{k_0}&\stackrel{\bigl(\mathrm{Exp}^\ast_k\bigr)_{k_0}}{\lra} &\Bigl(\mathrm{H}^1_{\rm dR}\bigl(\cX\bigl(p/\mathrm{Ha}^{p^r}\bigr), \mathbf{W}_{k, \rm dR,\bullet}\bigr)^{ (h)}\Bigr)_{k_0}\\
\downarrow &&\uparrow\\
\mathrm{H}^1\left(G,\mathrm{H}^1\bigl(\cX_{\overline{K},\proket},  \Symm^{k_0}(T_p(E)^\vee)(1)\bigr)^{(h)}\right)&\stackrel{\exp^\ast_{k_0}}{\lra} &\Fil^0\mathrm{H}^1_{\rm dR}\bigl(\cX, \Symm^{k_0}(\mathrm{H}_E)\bigr)^{(h)}
\end{array}
$$
where the right vertical arrow is induced by restriction.

\end{corollary}

\begin{proof}
Granted the theorem \ref{thm:deRcompare} we have the following natural $B$-linear and $G$-equaivariant maps:
$$
\mathrm{H}^1\bigl(\cX_{\overline{K},\proket}, \bD_k^o(T_0^\vee)[n](1) \bigr)^{(h)}\lra \mathrm{H}^1\bigl(\cX_{\overline{K},\proket}, \bD_k^o(T_0^\vee)[n](1)\bigr)^{(h)}\widehat{\otimes} B_{\rm dR}^+\stackrel{\rho_k}{\lra} 
$$
$$
\stackrel{\rho_k}{\lra}\mathrm{H}^1_{\rm dR}\bigl(\cX\bigl(p/\mathrm{Ha}^{p^r}\bigr), \mathbf{W}_{k, \rm dR,\bullet}\bigr)^{(h)}\widehat{\otimes} \Fil^{-1}B_{\rm dR}(1),
$$
whose composition we denote by $f_k$. Then we obtain the following natural map induced by composing $f_k$ in Galois cohomology with the natural isomorphism:
$$
\mathrm{H}^1\left(G,\mathrm{H}^1\bigl(\cX_{\overline{K},\proket}, \bD_k^o(T_0^\vee)[n](1) \bigr)^{(h)}\right)\lra \mathrm{H}^1\left(G, \mathrm{H}^1_{\rm dR}\bigl(\cX\bigl(p/\mathrm{Ha}^{p^r}\bigr), \mathbf{W}_{k, \rm dR,\bullet}\bigr)^{ (h)}\widehat{\otimes}\Fil^{-1}B_{\rm dR}(1)\right)\cong 
$$

$$
\cong (\mathrm{H}^1_{\rm dR}\bigl(\cX\bigl(p/\mathrm{Ha}^{p^r}\bigr), \mathbf{W}_{k, \rm dR,\bullet}\bigr)^{ (h)}\otimes_K \mathrm{H}^1\bigl(G, \Fil^{-1}B_{\rm dR}(1)\bigr) \cong \mathrm{H}^1_{\rm dR}\bigl(\cX\bigl(p/\mathrm{Ha}^{p^r}\bigr), \mathbf{W}_{k, \rm dR,\bullet}\bigr)^{ (h)}.
$$

\end{proof}

\begin{remark} The main reason we twist by $1$ the pro-Kummer \'etale sheaves  $\bD_k^o(T_0^\vee)[n]$ and $\Symm^{k_0}(T_p(E)^\vee$ in the corollary above is because the Galois representations attached to overconvergent eigenforms, respectively classical ones are quotients of pro-Kumer \'etale cohomology 
of such sheaves (with the twist, that is).
\end{remark}

\subsubsection{A refinement of the map $\alpha_{\rm log}$.}
 
Consider the \'etale cover $j_n\colon {\cI}g_n\bigl(p/\mathrm{Ha}^{p^r}\bigr)\to \cX\bigl(p/\mathrm{Ha}^{p^r}\bigr)$ given by choosing a generator of $C_n^\vee$. It is Galois with groups $\Delta_n\cong (\Z/p^n\Z)^\ast$. Let $D_n:=(T_p(E)/p^n)/C_n$. Then we get an exact sequence $0 \to D_n^\vee \to  T_p(E)^\vee/p^n T_p(E)^\vee \to C_n^\vee \to 0$ with a marked section $s$ of $C_n^\vee$.

\begin{proposition}\label{prop:gammacris} The restriction to ${\cI}g_n\bigl(p/\mathrm{Ha}^{p^r}\bigr)$ of the map $\alpha_{\rm max,n+1}^{\rm log}$  factors via the submodule $\mathrm{H}_E^\sharp \otimes_{\cO_{{\cI}g_n\bigl(p/\mathrm{Ha}^{p^r}\bigr)}^+}  \cO \mathbb{A}_{\rm max,n+1}^{\rm log}\vert_{{\cI}g_n\bigl(p/\mathrm{Ha}^{p^r}\bigr)}$. The induced map $$\beta_{\rm max,n+1}^{\rm log}\colon  (T_p (E))^\vee \otimes_{\Z_p} \cO \mathbb{A}_{\rm max,n+1}^{\rm log}\vert_{{\cI}g_n\bigl(p/\mathrm{Ha}^{p^r}\bigr)}  \lra \mathrm{H}_E^\sharp \otimes_{\cO_{{\cI}g_n\bigl(p/\mathrm{Ha}^{p^r}\bigr)}^+}  \cO \mathbb{A}_{\rm max,n+1}^{\rm log}\vert_{{\cI}g_n\bigl(p/\mathrm{Ha}^{p^r}\bigr)}  $$sends the tautological section $s$ of $C_n^\vee$ to the marked section $s$ of  $\mathrm{H}_E^\sharp/p^n  \mathrm{H}_E^\sharp$ modulo $p^n$ and sends $D_n^\vee$ to $0$ modulo $p^n$. In particular, $j_{n,\ast}(\alpha_{\rm max,n+1}^{\rm log})$  is equivariant for the action of $\Delta_n$.

\end{proposition}
\begin{proof}

Let $\mathcal{J}$ be the kernel of the map $\cO \mathbb{A}_{\rm log}\to \widehat{\cO}_{\cX_{\proket}}^+$.  By construction $\mathcal{J} $ maps to $\frac{\mathcal{I}}{p} \cO \mathbb{A}_{\rm max,n+1}^{\rm log} \subset p^n \cO \mathbb{A}_{\rm max,n+1}^{\rm log}$ via the map $\cO \mathbb{A}_{\rm log}  \lra \cO \mathbb{A}_{\rm max,n+1}^{\rm log}$.

The image of $\alpha_{\rm log}$ modulo $\mathcal{J}$ is contained in $\mathrm{H}_E^\sharp \otimes \widehat{\cO}_{{\cI}g_n\bigl(p/\mathrm{Ha}^{p^r}\bigr)}^+ \subset \mathrm{H}_E \otimes \widehat{\cO}_{{\cI}g_n\bigl(p/\mathrm{Ha}^{p^r}\bigr)}^+$. Also $\mathrm{Ha}^p \cdot \mathrm{H}_E\subset \mathrm{H}_E^\sharp$. As  $p/\mathrm{Ha}^p$ is a section of $\cO_{ \cX\bigl(p/\mathrm{Ha}^{p^r}\bigr)}^+$, we conclude that $p \mathrm{H}_E\subset \mathrm{H}_E^\sharp$. We deduce that the image of $\alpha_{\rm max,n+1}^{\rm log}$ is contained in $$\mathrm{H}_E^\sharp \otimes_{\cO_{{\cI}g_n\bigl(p/\mathrm{Ha}^{p^r}\bigr)}^+}  \cO \mathbb{A}_{\rm max,n+1}^{\rm log}\vert_{{\cI}g_n\bigl(p/\mathrm{Ha}^{p^r}\bigr)}  +  \mathcal{J} \, \mathrm{H}_E \otimes_{\cO_{{\cI}g_n\bigl(p/\mathrm{Ha}^{p^r}\bigr)}^+}  \cO \mathbb{A}_{\rm max,n+1}^{\rm log}\vert_{{\cI}g_n\bigl(p/\mathrm{Ha}^{p^r}\bigr)} \subset  $$ $$ \subset \mathrm{H}_E^\sharp \otimes_{\cO_{{\cI}g_n\bigl(p/\mathrm{Ha}^{p^r}\bigr)}^+}  \cO \mathbb{A}_{\rm max,n+1}^{\rm log} \vert_{{\cI}g_n\bigl(p/\mathrm{Ha}^{p^r}\bigr)} +  p\,\mathrm{H}_E \otimes_{\cO_{{\cI}g_n\bigl(p/\mathrm{Ha}^{p^r}\bigr)}^+}  \cO \mathbb{A}_{\rm max,n+1}^{\rm log}\vert_{{\cI}g_n\bigl(p/\mathrm{Ha}^{p^r}\bigr)}  \subset $$ $$\subset \mathrm{H}_E^\sharp \otimes_{\cO_{{\cI}g_n\bigl(p/\mathrm{Ha}^{p^r}\bigr)}^+}  \cO \mathbb{A}_{\rm max,n+1}^{\rm log}\vert_{{\cI}g_n\bigl(p/\mathrm{Ha}^{p^r}\bigr)} $$as claimed. We aslo know that $\alpha_{\rm cris}$ maps the tautological section $s$ of $C_n^\vee$ to the marked section $s$ of  $\mathrm{H}_E^\sharp/p^n  \mathrm{H}_E^\sharp$ modulo $p^n \cO \mathbb{A}_{\rm cris}+ \mathcal{J}$ and sends $D_n^\vee$ to $0$ modulo $p^n \cO\mathbb{A}_{\rm cris}+\mathcal{J}$. As $\mathcal{J}\subset p^n \cO \mathbb{A}_{\rm max,n+1}$ the conclusion follows. 

The equivariance of $j_{n,\ast}(\beta_{\rm max,n+1}^{\rm log})$  follows from the fact that after composing with the inclusion $\mathrm{H}_E^\sharp \subset j_n^\ast\bigl(\mathrm{H}_E\bigr)$, it coincides with $\alpha_{\rm max,n+1}^{\rm log}$ restricted to ${\cI}g_n\bigl(p/\mathrm{Ha}^{p^r}\bigr)$ and $\alpha_{\rm max,n+1}^{\rm log}$ is defined over $\cX\bigl(p/\mathrm{Ha}^{p^r}\bigr)$.

\end{proof}

\begin{corollary}\label{cor:BDkEk} 
Applying the formalism of VBMS to $\gamma_{\rm cris}$ we get map $$\delta\colon \WW_k^D\bigl((T_p(E))^\vee,D_n^\vee,s\bigr)\vert_{{\cI}g_n\bigl(p/\mathrm{Ha}^{p^r}\bigr)_\proket}  \lra \left( \mathbf{W}_{k, \rm dR} \widehat{\otimes}_{\cO_{{\cI}g_n\bigl(p/\mathrm{Ha}^{p^r}\bigr)}^+}  \cO \mathbb{A}_{\rm max,n+1}\vert_{{\cI}g_n\bigl(p/\mathrm{Ha}^{p^r}\bigr)}\right)^{\nabla=0}$$of sheaves over ${\cI}g_n\bigl(p/\mathrm{Ha}^{p^r}\bigr)_\proket$ such that $j_{n,\ast}(\delta)$ is $\Delta_n$-equivariant.

Composing with the map defined in Corollary \ref{cor:DkWkD} we get a map of sheaves on $\cX \bigl(p/\mathrm{Ha}^{p^r}\bigr)_\proket$:
$$\zeta_k\colon \bD_k^o(T_0^\vee)[n]  \lra  \left(\mathbf{W}_{k, \rm dR} \widehat{\otimes}_{\cO_{\cX\bigl(p/\mathrm{Ha}^{p^r}\bigr)}^+}  \cO \mathbb{A}_{\rm max,n+1}[p^{-1}]\right)^{\nabla=0}$$that is functorial with respect to isogenies $E\to E'$ inducing an isomorphism on canonical subgroups.

\end{corollary}
\begin{proof}

We define $\delta$ for a log affinoid perfectoid open $W$ of ${\cI}g_n\bigl(p/\mathrm{Ha}^{p^r}\bigr)_\proket$ over an open affinoid $U=\mathrm{Spa}(R,R^+)$ of ${\cI}g_n\bigl(p/\mathrm{Ha}^{p^r}\bigr)$. We assume that $T_p(E)(W)$ is trivial and that  $H^\sharp_E(U)$ is a free of rank two $R^+$-module.   It is easy to see that we have isomorphisms $$\V_0^D\bigl((T_p(E))^\vee,D_n^\vee,s\bigr) \widehat{\otimes}_{\Z_p} \cO \mathbb{A}_{\rm max,n+1}(W)\cong \V_0^D\bigl((T_p(E)\otimes_{\Z_p}\cO \mathbb{A}_{\rm max,n+1}(W))^\vee,D_n^\vee\otimes_\Z \cO \mathbb{A}_{\rm max,n+1}(W),s\bigr)$$and that $$\V_0\bigl(\mathrm{H}_E^\sharp,s\bigr)\widehat{\otimes}_{R^+} \cO \mathbb{A}_{\rm max,n+1}(W)\cong \V_0\bigl(\mathrm{H}_E^\sharp\widehat{\otimes}_{R^+} \cO \mathbb{A}_{\rm max,n+1}(W),s\bigr).$$Thanks to Proposition \ref{prop:gammacris} the map $\beta_{\rm max,n+1}^{\rm log}$ induces a map 
$$\V_0\bigl(\mathrm{H}_E^\sharp\widehat{\otimes}_{R^+} \cO \mathbb{A}_{\rm max,n+1}(W),s\bigr) \to \V_0^D\bigl((T_p(E)\otimes_{\Z_p}\cO \mathbb{A}_{\rm max,n+1}(W))^\vee,D_n^\vee\otimes_\Z \cO \mathbb{A}_{\rm max,n+1}(W),s\bigr).$$In conclusion we get a map $$\V_0\bigl(\mathrm{H}_E^\sharp,s\bigr)\widehat{\otimes}_{R^+} \cO \mathbb{A}_{\rm max,n+1}(W) \lra \V_0^D\bigl((T_p(E))^\vee,D_n^\vee,s\bigr) \widehat{\otimes}_{\Z_p} \cO \mathbb{A}_{\rm max,n+1}(W) .$$This induces the claimed map $\delta(W)$. As $j_{n,\ast}(\beta_{\rm max,n+1}^{\rm log})$ is $\Delta_n$-equivariant due to  Proposition \ref{prop:gammacris}, also $\delta$ is.

The formalism of VBMS implies that these maps are functorial with respect to isogenies and base change. The connection on $\mathbf{W}_{k, \rm dR}$ is defined in \cite{andreatta_iovita_triple} using Grothendieck approach: passing to the base change over the first infinitesimal neighborhhod of the diagonal of  $\cX \bigl(p/\mathrm{Ha}^{p^r}\bigr)$ it is realized as an isomorphism between the pull-back via the two projections to $\cX \bigl(p/\mathrm{Ha}^{p^r}\bigr)$ using the functoriality of the VBMS formalism. The fact that $\beta_{\rm max,n+1}^{\rm log}$ has image annihilated by $\nabla$ implies that the image of $\delta(W)$ is also annihilate by the connection.  

\end{proof}

We write $\mathbf{W}_{k, \rm dR,\bullet}$ for the complex $ \mathbf{W}_{k, \rm dR}[p^{-1}] \to \mathbf{W}_{k+2, \rm dR}[p^{-1}]$ where the map is defined by $\nabla$ and
$\Fil_m \mathbf{W}_{k, \rm dR,\bullet}$ for the complex $\Fil_m \mathbf{W}_{k, \rm dR}[p^{-1}] \to \Fil_{m+1} \mathbf{W}_{k+2, \rm dR}[p^{-1}]$ where the map is defined by $\nabla$. Recall that $\mathrm{H}^1\bigl(\cX_{\overline{K},\proket}, \bD_k^o(T_0^\vee)[n] \bigr)$ admits slope decomposition for the $U_p$-operator thanks to Proposition \ref{prop:GAGAD}.

\begin{lemma}\label{lemma:H1factorsFil} For every $h$ there exists $m$ such the map $$\mathrm{H}^1\bigl(\cX_{\overline{K},\proket}, \bD_k^o(T_0^\vee)[n] \bigr)^{(h)}\lra 
\mathrm{H}^1\bigl(\cX\bigl(p/\mathrm{Ha}^{p^r}\bigr)_{\overline{K},\proket}, \mathbf{W}_{k, \rm dR,\bullet} \widehat{\otimes}_{\cO_{\cX(p/\mathrm{Ha}^{p^r})}^+}  \cO \mathbb{A}_{\rm max,n+1}^{\rm log}[p^{-1}]\bigr),$$induced by $\delta$,  factors via $$\mathrm{H}^1\bigl(\cX(p/\mathrm{Ha}^{p^r})_{\overline{K},\proket}, \Fil_m \mathbf{W}_{k, \rm dR,\bullet} \otimes_{\cO_{\cY(p/\mathrm{Ha}^{p^r})}^+}  \cO \mathbb{A}_{\rm max,n+1}^{\rm log}[p^{-1}]\bigr).$$
\end{lemma}
\begin{proof} Consider the complex $\mathbf{W}_{k, \rm dR,\bullet}/\Fil_m\mathbf{W}_{k, \rm dR,\bullet}$. We claim that for $i=0$ and $1$ $$\mathrm{H}^i\bigl(\cX\bigl(p/\mathrm{Ha}^{p^r}\bigr)_{\overline{K},\proket}, \mathbf{W}_{k, \rm dR,\bullet}/\Fil_m\mathbf{W}_{k, \rm dR,\bullet}\widehat{\otimes}_{\cO_{\cX(p/\mathrm{Ha}^{p^r})}^+}  \cO \mathbb{A}_{\rm max,n+1}^{\rm log}[p^{-1}]\bigr)$$admits a slope $h$-decomposition and  the $\leq h$-part is zero, for $m\gg 0$.

It follows from \cite[Lemma 3.33]{andreatta_iovita_triple} that the operator $U_p$ on $$\mathrm{H}^i\bigl(\cX\bigl(p/\mathrm{Ha}^{p^r}\bigr)_{\overline{K},\proket}, \mathbf{W}_{k, \rm dR,\bullet}/\Fil_m\mathbf{W}_{k, \rm dR,\bullet}\widehat{\otimes}_{\cO_{\cX(p/\mathrm{Ha}^{p^r})}^+}  \cO \mathbb{A}_{\rm max,n+1}^{\rm log}\bigr) $$is integrally defined and can be written as $p^{h+1} U'_p$ for some operator $U'_p$ for $m\gg 0$. The proof then follows from \cite{half}, Lemma 5.8 and subsequent claim: one shows that for every polynomial $P$  of slope $\leq h$, $P(U_p)$  is invertible on this space after inverting $p$.
\end{proof}

\subsubsection{Proof of Theorem \ref{thm:deRcompare}} 
\label{sec:proof}

Consider the composite $\Fil_m\mathbf{W}_{k, \rm dR,\bullet} \otimes_{\cO_{\cY(p/\mathrm{Ha}^{p^r})}^+}  \mathbb{A}_{\rm max,n+1}[p^{-1}]\to\Fil_m\mathbf{W}_{k, \rm dR,\bullet}  \otimes_{\cO_{\cY(p/\mathrm{Ha}^{p^r})}^+}  \Fil^0 \cO\mathbb{B}_{\rm dR,log}$ obtained from the morphism  $\cO\mathbb{A}_{\max,n+1}^{\log}\to \cO\mathbb{B}_{\rm dR, log}^+$ of Lemma \ref{lemma:BdR}. 
We recall that the connection $\nabla':\mathbf{W}_{k, \rm dR}\widehat{\otimes}\cO\mathbb{B}_{\rm dR}\lra \mathbf{W}_{k+2,\rm dR}\widehat{\otimes}\cO\mathbb{B}_{\rm dR}$ has the form 
$\nabla'=\nabla_k\widehat{\otimes} 1+1\widehat{\otimes} \nabla_{\rm dR}$, where $\nabla_k$ is the connection on $\mathbf{W}_{k, \rm dR}$ and $\nabla_{\rm dR}$ the one on $\cO\mathbb{B}_{\rm dR}$. Moreover both $\nabla_k$ and $\nabla_{\rm dR}$ satisfy the Griffith-trasveraslity property with respect to the respective filtrations
of $\mathbf{W}_{k, \rm dR}$ and respectively $\cO\mathbb{B}_{\rm dR}$, where let us recall that 
the first sheaf has an increasing filtration ${\rm Fil}_\bullet\mathbf{W}_{k, \rm dR}$ while 
the second has a decreasing filtration ${\rm Fil}^\bullet \cO\mathbb{B}_{\rm dR}$.

For every $s\ge 1$ we consider the composition, which we still denote $\nabla'$:

\noindent
$\displaystyle \Fil_m\mathbf{W}_{k, \rm dR}  \otimes \frac{ \Fil^0 \cO\mathbb{B}_{\rm dR,log}}{\Fil^s \cO\mathbb{B}_{\rm dR,log}}\stackrel{\nabla'}{\lra} \Fil_{m+1}\mathbf{W}_{k+2, \rm dR}  \otimes  \frac{\Fil^0\cO\mathbb{B}_{\rm dR,log}}{\Fil^s \cO\mathbb{B}_{\rm dR,log}}+\Fil_m\mathbf{W}_{k+2,\rm dR}\otimes\frac{\Fil^{-1}\cO\mathbb{B}_{\rm dR}}{\Fil^{s-1}\cO\mathbb{B}_{\rm dR}}\lra \Fil_{m+1}\mathbf{W}_{k+2, \rm dR}  \otimes  \frac{\Fil^{-1}\cO\mathbb{B}_{\rm dR,log}}{\Fil^{s-1} \cO\mathbb{B}_{\rm dR,log}}$

\noindent
 and we have:

\begin{lemma}\label{lemma:H1WkNablaWk} For every $s\ge 1$, the natural map of complexes
$$ \begin{array}{cccccccccc} 
\mathrm{H}^0\Bigl(\cX\bigl(p/\mathrm{Ha}^{p^r}\bigr)_{\overline{K},\proket}, \Fil_m \mathbf{W}_{k, \rm dR}\Bigr)\widehat{\otimes}\otimes \frac{\Fil^0 B_{\rm dR}}{\Fil^s B_{\rm dR}} & \stackrel{\nabla}{\lra} & \mathrm{H}^0\Bigl(\cX\bigl(p/\mathrm{Ha}^{p^r}\bigr)_{\overline{K}_0,\proket}, \Fil_{m+1}\mathbf{W}_{k+2, \rm dR}\Bigr)\widehat{\otimes} \frac{ \Fil^{-1} B_{\rm dR}}{\Fil^{s-1} B_{\rm dR}}  \\ \downarrow & & \downarrow \\ 
\mathrm{H}^0\Bigl(\cX\bigl(p/\mathrm{Ha}^{p^r}\bigr)_{\overline{K},\proket}, \Fil_m\mathbf{W}_{k, \rm dR}  \otimes \frac{ \Fil^0 \cO\mathbb{B}_{\rm dR,log}}{\Fil^s \cO\mathbb{B}_{\rm dR,log}}\Bigr) & \stackrel{\nabla'}{\lra}& \mathrm{H}^0\Bigl(\cX\bigl(p/\mathrm{Ha}^{p^r}\bigr)_{\overline{K},\proket}, \ 
\Fil_{m+1}\mathbf{W}_{k+2, \rm dR}  \otimes  \frac{\Fil^{-1}\cO\mathbb{B}_{\rm dR,log}}{\Fil^{s-1} \cO\mathbb{B}_{\rm dR,log}}\Bigr).  \end{array}$$is an isomorphism. 
 
In the above diagram $B_{\rm dR}$ denotes Fontaine's classical period ring. Furthermore this complex  represents the cohomology $\mathrm{R}\Gamma\bigl(\cX\bigl(p/\mathrm{Ha}^{p^r}\bigr)_{\overline{K},\proket}, \Fil_m\mathbf{W}_{k, \rm dR,\bullet}  \otimes_{\cO_{\cX(p/\mathrm{Ha}^{p^r})}^+}  \frac{\Fil^0\mathbb{B}_{\rm dR}}{\Fil^s \mathbb{B}_{\rm dR}}\bigr) $.\end{lemma}

\begin{proof} Recall that $ \Fil_m\mathbf{W}_{k, \rm dR}[1/p] $ is a locally free $\cO_{\cX\bigl(p/\mathrm{Ha}^{p^r}\bigr)}$-module for every $m$. We prove the result restricting to an affinoid cover $\{\cU_i\}_{i\in I}$ where $\Fil_m$ and $\Fil_{m+1}$ are free. Since $\cX\bigl(p/\mathrm{Ha}^{p^r}\bigr)$ is affinoid, the Chech complex for $\Fil_m\mathbf{W}_{k, \rm dR}[1/p]$ w.r.t. the $\cU_i$'s is exact. As  $\Fil^0 B_{\rm dR}/\Fil^s  B_{\rm dR}$ is an iterated extension of  $\C_p$-vector spaces for every $h$, the Chech complex remains exact aslo after taking $\widehat{\otimes} \Fil^0 B_{\rm dR}/\Fil^s  B_{\rm dR}$. From the result for the $\cU_i$'s we then deduce the Lemma.

We are left to show the claim for each $\cU_i=\mathrm{Spa}(R_i,R_i^+)$. Then   \cite[Lemma 3.3.2]{DLLZ} implies that the group $\mathrm{H}^j\bigl(\cU_{i,\overline{K},\proket},\Fil^{u}\cO\mathbb{B}_{\rm dR,log}/\Fil^{s+u} \cO\mathbb{B}_{\rm dR,log}\bigr)=0$ for $j\geq 1$ and coincides with $\Fil^{u}\left(R_i^+\widehat{\otimes } B_{\rm dR}\right)/\Fil^{s+u} \left(R_i^+\widehat{\otimes } B_{\rm dR} \right)$ for $i=0$, for $u=0, -1$. As the latter coincides with $R_i^+\widehat{\otimes } \bigl(\Fil^u B_{\rm dR}/\Fil^{s+u} B_{\rm dR}\bigr)$ by \cite[Def. 3.1.1]{DLLZ}, the conclusion follows.
\end{proof}

We deduce from Lemma \ref{lemma:H1factorsFil} and Lemma \ref{lemma:H1WkNablaWk} that for every positive integer $s$ we have a natural map

$$\mathrm{H}^1\bigl(\cX_{\overline{K},\proket}, \bD_k^o(T_0^\vee)[n] \bigr)^{(h)}\widehat{\otimes} \bigl(\Fil^0B_{\rm dR}/\Fil^sB_{\rm dR}\bigr)\to\frac{ \mathrm{H}^0\bigl(\cX\bigl(p/\mathrm{Ha}^{p^r}\bigr) , \Fil_{m+1} \mathbf{W}_{k+2, \rm dR}\bigr) \widehat{\otimes} ( \Fil^{-1} B_{\rm dR}/\Fil^{s-1} B_{\rm dR}) }{\nabla_k\left( \mathrm{H}^0\bigl(\cX\bigl(p/\mathrm{Ha}^{p^r}\bigr) , \Fil_m \mathbf{W}_{k, \rm dR}\bigr)\right) \widehat{\otimes} ( \Fil^0 B_{\rm dR}/\Fil^s B_{\rm dR})}.$$
On the other hand, recall from Proposition \ref{prop:propWk} that $\mathrm{H}^0\bigl(\cX\bigl(p/\mathrm{Ha}^{p^r}\bigr) , \Fil_{m} \mathbf{W}_{k, \rm dR}[1/p]\bigr)$ admits a slope $\leq {h-1}$ decomposition and 
$\mathrm{H}^0\bigl(\cX\bigl(p/\mathrm{Ha}^{p^r}\bigr) , \Fil_{m+1} \mathbf{W}_{k+2, \rm dR}[1/p]\bigr)$ admits a slope $\leq h$ and they coincide
with $\mathrm{H}^0\bigl(\cX\bigl(p/\mathrm{Ha}^{p^r}\bigr) ,  \mathbf{W}_{k, \rm dR}[1/p]\bigr)^{(h-1)}$, resp. $\mathrm{H}^0\bigl(\cX\bigl(p/\mathrm{Ha}^{p^r}\bigr)  , \mathbf{W}_{k+2, \rm dR}[1/p]\bigr)^{(h)}$.

The same then holds after $\widehat{\otimes}\Fil^0B_{\rm dR}/\Fil^sB_{\rm dR}$, respectively  $\widehat{\otimes}  \Fil^{-1} B_{\rm dR}/\Fil^{s-1} B_{\rm dR}$ and for their quotient via $\nabla$ (by the five Lemma for slope decompositions cf. \cite[Thm. 5.7]{half}). As $\Fil^0B_{\rm dR}=B_{\rm dR}^+$, $\Fil^{-1} B_{\rm dR}=t^{-1}B_{\rm dR}^+$ and $\Fil^{s+u} B_{\rm dR}=t^{s+u} B_{\rm dR}^+$ for 
$u\in \{-1, 0\}$ we get maps

$$\mathrm{H}^1\bigl(\cX_{\overline{K},\proket}, \bD_k^o(T_0^\vee)[n] \bigr)^{(h)}\widehat{\otimes} \frac{B_{\rm dR}^+}{t^sB_{\rm dR}^+}\lra \mathrm{H}^1_{\rm dR}\bigl(\cX\bigl(p/\mathrm{Ha}^{p^r}\bigr), \mathbf{W}_{k, \rm dR,\bullet}[1/p]\bigr)^{(h)}\widehat{\otimes} \frac{  t^{-1}B_{\rm dR}^+}{t^{s-1} B_{\rm dR}^+}$$with 
$$\mathrm{H}^1_{\rm dR}\bigl(\cX\bigl(p/\mathrm{Ha}^{p^r}\bigr), \mathbf{W}_{k, \rm dR,\bullet}[1/p]\bigr)^{(h)}\cong  \frac{\mathrm{H}^0\bigl(\cX\bigl(p/\mathrm{Ha}^{p^r}\bigr)  ,\mathbf{W}_{k+2, \rm dR}[1/p]\bigr)^{(h)}}{\nabla_k\left(\mathrm{H}^0\bigl(\cX\bigl(p/\mathrm{Ha}^{p^r}\bigr) ,\mathbf{W}_{k, \rm dR}[1/p]\bigr)^{(h-1)}\right)}.$$
As  they are obtained from maps of sheaves on $\cX_\proket$, the equivariance for the $\mathrm{Gal}(\overline{K}/K)$-action is clear. The compatibility with Hecke operators follows from the fact that $\zeta_k$ is comaptible with the map induced by isogenies preserving the canonical subgroup that are used to define the Hecke operators $T_\ell$, for $\ell\not\vert pN$, and the Hecke operator $U_p$. It is comaptible with weight specializations as the map $\zeta_k$ is.  Taking the inverse limits for $s\to \infty$ we get the statement of Theorem \ref{thm:deRcompare}, except for (i) and (ii). Claim (ii) follows from \cite[section \S 3.9]{andreatta_iovita_triple}. Using (ii) we get  a map 

$$\mathrm{H}^1\bigl(\cX_{\overline{K},\proket}, \bD_k^o(T_0^\vee)[n] \bigr)^{(h)}\widehat{\otimes} \frac{B_{\rm dR}^+}{t^sB_{\rm dR}^+}\lra \mathrm{H}^0\bigl(\cX\bigl(p/\mathrm{Ha}^{p^r}\bigr), \omega_E^k\bigr)^{(h)}\widehat{\otimes} \frac{ t^{-1} B_{\rm dR}^+}{t^{s-1} B_{\rm dR}^+},$$
which we'd like to prove is surjective under the hypothesis of i). By devisage  it suffices to prove  surjectivity for $s=1$.  As $t^{-1}B_{\rm dR}^+/B_{\rm dR}^+ \cong \mathbb{C}_p(-1)$, the surjectivity  follows from Theorem \ref{thm:hodgetate}.

\section{Appendix: Integral slope decomposition.}
\label{sec:appendix}

Let us start by formulating the following general property.

We let $R$ be a $p$-torsion free $\Z_p$-algebra and $T$ an $R$-module equipped with an 
$R$-linear operator $v:T\rightarrow T$, let $\alpha\in R$ be an element such that there is $r\in \N$ and $\gamma\in R$  with $\alpha\gamma=p^r$.
We denote by $\rho\colon T\lra T[1/p]:=T\otimes_RR[1/p]$ and denote by
$ T^{\rm tors}:={\rm Ker}(\rho)$,  $T^{\rm tf}:=T/T^{\rm tors}={\rm Im}(\rho)$ and remark that
$v-\alpha$ respects the submodule $T^{\rm tors}$ and therefore induces an $R$-linear map on
$T^{\rm tf}$. 

\begin{definition}
\label{def:ast}
We say that the triple $(T,v,\alpha)$ has property $(*)$ if

1) There is $w\in \N$ such that $p^w\bigl(T^{\rm tors} \bigr)^{v=\alpha}=0.$

2) There is a $\eta\in \N$, which depends only on $\alpha$, such that for every $x\in \bigl(T^{\rm tf}\bigr)^{v=\alpha} $  there is $\tilde{x}\in T^{v=\alpha}$ such that $(\tilde{x})^{\rm tf}=p^\eta x$, where we denoted $(\tilde{x})^{\rm tf}$ the image of $\tilde{x}$ in $T^{\rm tf}$. 
\end{definition}

The main result of this Appendix is the following. Let $B$ denote a ring and let $k\colon \Z_p^\ast\lra B^\ast$ be a $B$-valued weight as in Definition \ref{def:weights}, which is $s$-analytic. Let $\cX_{\infty}^{(u)}$  be the adic subspace of the adic modular curve $\cX:=\cX_0(p^n,N)$ for $n\geq u$ as in  Proposition \ref{proposition:existcanonical}.   Let $\bD_k(T_0^\vee)[n]$  denote the pro-Kummer \'etale sheaf of weight $k$ distributions, for $n\geq s$, over $\cX_\infty^{(u)}$ and $\mathfrak{D}_{k,\infty}^{o,(m)}[n]:=\bD_k(T_0^\vee)[n]\widehat{\otimes}\cO_{(\cX_\infty^{(u)})_{\rm pke}}^+$, where we have denoted $\cO_{(\cX_\infty^{(u)})_{\rm pke}}^+$ the structure sheaf of the pro-Kummer \'etale site of
$\cX_\infty^{(u)}$. We write $R:=\cO^+_{\cX_{\infty}^{(u)}}(\cX_{\infty}^{(u)})\widehat{\otimes}_{\Z_p}B$ and $T:={\rm H}^1\bigl((\cX^{(u)}_\infty )_\proket, \mathfrak{D}_{k,\infty}^{o,(m)}[n]\bigr)$. On $T[1/p]$ we have a $B[1/p]$-linear operator $U_p$ which has finite slope decompositions by Proposition \ref{prop:HeckeWkinfty}.
 
Let $Q(X)\in \bigl(B\widehat{\otimes}\cO_{\C_p}\bigr)[X]$ be the polynomial with the property that  $T[1/p]^{(b)}$, for some $b\in \N$ is the subset of elements $x\in T[1/p]$ such that $Q(U_p)x=0$. Such a polynomial exists as $T[1/p]^{(b)}\cong {\rm H}^0\bigl(\cX_\infty^{(u)}, \omega^{k+2}  \bigr)[1/p]^{(b)}$ by Proposition \ref{prop:HeckeWkinfty}  and on ${\rm H}^0\bigl(\cX_\infty^{(u)}, \omega^{k+2}  \bigr)$ the $U_p$ operator is compact and has a Fredholm determinant. Then $\alpha:=-Q(0)\in p^a(B\widehat{\otimes}\cO_{\C_p})^\ast$ with $a\le b\cdot {\rm deg}(Q(X))$. We write $Q(X)=P(X)-\alpha$, with $P(X)=XR(X)$ and $P(X), R(X)\in \bigl(B\widehat{\otimes}\cO_{\C_p}\bigr)[X]$. We denote $v:=P(U_p)$ and remark that $x\in T[1/p]^{(b)}$ if and only if $v(x)=\alpha x$.
We have

\begin{theorem}
\label{thm:mainappendix}
After localizing $B$ to a new $p$-adically complete ring which we denote by $B'$ and replacing $R$ by $R':=\cO^+_{\cX_{\infty}^{(u)}}(\cX_{\infty}^{(u)})\widehat{\otimes}_{\Z_p}B'$ and $T$ by $T':=T\otimes_RR'$ the triple $(T',v\otimes 1_{R'},\alpha)$ above satisfies property $(\ast)$ of definition \ref{def:ast}.
\end{theorem}

Before we start on the proof of this theorem we need a few lemmas.
We remind the reader that the sheaf $\mathfrak{D}_{k,\infty}^{o,(m)}[n] $ has a decreasing filtration $\bigl(\fFil^\nu\bigr)_{\nu\ge 0}$ by subsheaves with the property
(see Proposition \ref{prop:HeckeWkinfty}):  for all $\nu \ge 0$   and $i\ge 0$ we have

$$
U_p\Bigl({\rm H}^i\bigl((\cX_\infty^{(u)})_\proket,  \fFil^\nu \bigr)\Bigr)\subset p^{\nu+1}{\rm H}^i\bigl((\cX_\infty^{(u)})_\proket, \fFil^\nu\bigr).
$$
Moreover,  
we have $\fFil^\nu/\fFil^{\nu+1}\cong
\omega^{k-2\nu-2}\widehat{\otimes}\cO_{\cX_\infty^{(u)}}(\nu+1)$ by (\ref{eq:gr}).
With these notations we have.

\begin{lemma}
\label{lemma:filtration}
For every $\nu \in \N$ large enough and $i\ge  0$ the triple $\Bigl(T_i^\nu:={\rm H}^i\bigl((\cX_\infty^{(u)})_\proket, \fFil^\nu\bigr), v=P(U_p), \alpha\Bigr)$ satisfies property $(\ast)$ of definition \ref{def:ast}.

\end{lemma}

\begin{proof}
We have the following commutative diagram with exact rows
$$
\begin{array}{llllllllllll}
0&\lra&\bigl(T_i^\nu\bigr)^{\rm tors}&\lra&T_i^\nu&\lra&\bigl(T_i^\nu \bigr)^{\rm, tf}&\lra&0\\
&&\downarrow v-\alpha&&\downarrow v-\alpha&&\downarrow v-\alpha\\
0&\lra&\bigl(T_i^\nu\bigr)^{\rm tors}&\lra&T_i^\nu&\lra&\bigl(T_i^\nu \bigr)^{\rm, tf}&\lra&0\\
\end{array}
$$

We remark that the above property about the behavior of $U_p$ with respect to the cohomology of the filtration and the fact that $v$ is the composition of $U_p$ with an endomorphism of $T_i^\nu$ which commutes with $U_p$, implies that there is $\nu_0$ such that for all $\nu \ge \nu_0$ we have
$(v-\alpha\cdot {\rm Id}_{T^\nu_i})=\alpha U'_i$, where $U'_i:T^\nu_i\lra T^\nu_i$ is an isomorphism. 
Therefore $\bigl(T_i^\nu)\bigr)^{v=\alpha}=T^\nu_i[\alpha]=\bigl(T^\nu_i\bigr)^{\rm tors}[\alpha]$
and $\bigl(T^\nu_i\bigr)^{\rm tf}[\alpha]=0$. Therefore 1) of property $(*)$ follows: $\alpha\Bigl(\bigl(T^\nu_i\bigr)^{\rm tors}\Bigr)^{v=\alpha}=
\alpha (T^\nu_i[\alpha])=0$.
For property 2) let $x\in \Bigl(\bigl(T^\nu_i  \bigr)^{\rm tf}\Bigr)^{v=\alpha}=0$. For every $y\in \bigl(T^\nu_i\bigr)^{v=\alpha}=T^\nu_i[\alpha]$ we have $\alpha y=0=x$.

\end{proof}

\begin{lemma}
\label{lemma:quotient}
For every $\nu\in \N$ the triple $\Bigl(T_\nu:={\rm H}^1\bigl((\cX_{\infty}^{(u)})_\proket, \fD_k^o/\fFil^\nu\bigr), u= P(U_p), \alpha  \Bigr)$ satisfies the property $(\ast)$ of definition \ref{def:ast}.

\end{lemma}

\begin{proof} In order to prove the lemma we'll use induction on $\nu\ge 0$.
For $\nu=0$ we have $\fD_k/\fFil^0\cong 
\omega^k\widehat{\otimes}\cO^+_{\cX_\infty^{(u)}}$, therefore we have: 

1) $p^{1/(p-1)}(T_0)^{\rm tors}=p^{1/(p-1)}{\rm H}^1\bigl(
(\cX^{(u)}_{\infty})_\proket, \widehat{\cO}^+_{\cX_\infty^{(u)}}\bigr)^{\rm tors}\widehat{\otimes}{\rm H}^0(\cX_\infty^{(u)}, \omega^k)=0$ as computed by Faltings; see \cite[Lemma 5.5 \& 5.6]{ScholzeHodge} when the log structure is trivial and their analogues \cite[Lemma 6.1.7 \& 6.1.11]{Diao} in the general case; 

\noindent and 

2) if $x\in (T_0)^{\rm tf}$ is such that $(v-\alpha)x=0$, let $y\in T_0$ be any lift of $x$. Then $(v-\alpha)(y)\in T_0^{\rm tors}$, therefore $z:=p^{1/(p-1)}y\in T_0$ satisfies: $(v-\alpha)(z)=0$ and $z^{\rm tf}=p^{1/(p-1)}x$, where we wrote $z^{\rm tf}$ for the image of $z$ in $T_0^{\rm tf}$. 
Let us observe that we proved more then 1) of property $(\ast)$, namley we showed that there is 
$r\ge 0$ such that $p^rT_0^{\rm tors}=0$. We'll prove the same property, call it $(\ast \ast)$ for all $\nu\ge 0$.

Suppose now that $(\ast \ast)$ is true for $T_\nu$, $\nu\ge 1$ and let us prove it for $T_{\nu+1}$.
We have an exact sequence of pro-Kumer \'etale sheaves on the site $\cV:=(\cX_\infty^{(u)})_\proket$:
$$
0\lra \fFil^\nu/\fFil^{nu+1}\lra \fD_k^o/\fFil^{\nu+1}\lra \fD_k^o/\fFil^\nu\lra 0
$$
therefore a long exact  cohomology sequence:
$$
\cA:={\rm H}^0(\cV, \fD_k/\fFil^\nu)\stackrel{\beta}{\rightarrow}\cB:={\rm H}^1(\cV, \fFil^\nu/\fFil^{\nu+1})\stackrel{\gamma}{\rightarrow}\cC:={\rm H}^1(\cV, \fD/\fFil^{\nu+1})\stackrel{\delta}{\rightarrow}\cD:={\rm H}^1(\cV, \fD/\fFil^\nu)\rightarrow 0
$$

By the induction hypothesis there is $r$ such that  $p^r\cA^{\rm tors}=p^r\cB^{\rm tors}=p^r\cD^{\rm tors}=0$.
We have the commutative diagram with the middle row exact:
$$
\begin{array}{cccccccccccccc} 
&\cA^{\rm tors}&\stackrel{\beta}{\lra}&\cB^{\rm tors}&\stackrel{\gamma}{\lra}&\cC^{\rm tors}&\stackrel{\delta}{\lra}&\cD^{\rm tors}\\
&\downarrow&&\downarrow&&\downarrow&&\downarrow\\
(1)\quad& \cA&\stackrel{\beta}{\lra}&\cB&\stackrel{\gamma}{\lra}&\cC&\stackrel{\delta}{\lra}&\cD&\lra 0\\
&\downarrow&&\downarrow&&\downarrow&&\downarrow\\
&\cA^{\rm tf}&\stackrel{\beta}{\lra}&\cB^{\rm tf}&\stackrel{\gamma}{\lra}&\cC^{\rm tf}&\stackrel{\delta}{\lra}&\cD^{\rm tf}
\end{array}
$$

Let us suppose that there is $s\ge 0$ such that $p^s\Bigl(\cB^{\rm tf}/\beta(\cA^{\rm tf})  \Bigr)^{\rm tors}=0$.
Then we claim that $p^{2r+s}\cC^{\rm tors}=0$.
To see it, let $c\in \cC^{\rm tors}$, then $p^rx=\gamma(y)$, $y\in B$. We denote by $[y]$ the image of
$y$ in $\cB^{\rm tf}/\beta(\cA^{\rm tf})$. As $p^Nx=0$ for some $N\ge 0$ we have that $p^N[y]=0$,
therefore $[y]\in \Bigl(\cB^{\rm tf}/\beta(\cA^{\rm tf}  \Bigr)^{\rm tors}$ and so by the above assumption
$p^s[y]=0$. Let $z\in \cA$ be such that $\beta(z^{\rm tf})=p^sy^{\rm tf}$ in $\cB^{\rm tf}$.
It follows that $\beta(z)-p^sy\in \cB^{\rm tors}$ which implies that $p^{r+s}y=p^r\beta(z)$.
Therefore $p^{s+2r}x=0$.

Let us now prove the remaining claim, namely that there is $s\ge 0$ such that 
$p^s\Bigl(\cB^{\rm tf}/\beta(\cA^{\rm tf}  \Bigr)^{\rm tors}=0$. For this let us recall \cite{half} that we have a commutative diagram with exact rows
$$
\begin{array}{ccccccccccc}
&& {\rm H}^0(\cX_\infty^{(u)}, \omega^{k-2\nu})&\stackrel{\tilde{\beta}}{\lra}&{\rm H}^0(\cX_\infty^{(u)}, \omega^{k-2\nu})&\lra&{\rm Coker}(\tilde{\beta})&\lra&0\\
&&\downarrow i&&\downarrow j&&\downarrow u\\
&&A^{\rm tf}&\stackrel{\beta}{\lra}&\cB^{\rm tf}&\lra&\cB^{\rm tf}/\beta(\cA^{\rm tf})&\lra&0
\end{array}
$$
where $i$ and $j$ are injective with cokernels killed by $p^{\nu/(p-1)}$ and $p^{1/(p-1)}$ respectively.
Moreover, it follows using the explicit basis of the filtration described in Proposition \ref{prop:Wkinfty} and \cite[Prop. 5.2]{half}  that $\tilde{\beta}=\bigl(\prod_{n=0}^\nu (u_k-n)\bigr)\beta'$, with $\beta'$ an isomorphism. As $p^{1/(p-1)}$ kills ${\rm Coker}(j)$, it also kills ${\rm Coker}(u)$. Therefore it is enough to prove the claim for ${\rm Coker}(\tilde{\beta})^{\rm tors}$. We have two possibilities. Either $u_k=n$ for some $0\leq n\leq \nu$ and then  $\cB^{\rm tf}/\beta(\cA^{\rm tf})=\cB^{\rm tf}$ so that the claim is obvious. Else $\prod_{n=0}^\nu (u_k-n)\in (B[1/p])^\ast$ due to our assumption on $B$ in Definition \ref{def:weights}, i.e., there exists $s\in \N$ and $v\in B$ such that $ \prod_{n=0}^\nu (u_k-n) \cdot v=p^s$. Then  $ \tilde{\beta}$ is injective with cokernel annihilated by $p^s$ and the claim is proven also in this case.

So we have proved that the property $(\ast \ast)$ holds for triples $(T_\nu, v=P(U_p), \alpha)$ for all $\nu\ge 0$, which implies 1) of property $(\ast)$.

Let us prove 2) of property $(\ast)$ for $\cC$, supposing that it holds for $\cD$.
We recall our diagram $(1)$ and let $x\in \cC^{\rm tf}$ be such that $(v-\alpha)(x)=0$.
Then $\delta(x)\in \cD^{\rm tf}$ is such that $(v-\alpha)(\delta(x))=0$, therefore by the induction
hypothesis there is $m\ge 0$ and $y\in \cD^{v=\alpha}$ such that $y^{\rm tf}=p^m\delta(x)$.
Let $z\in \cC$ be such that $\delta(z)=y$ and let $\tilde{x}\in \cC$ be such that $\tilde{x}^{\rm tf}=x$. Then there is $q\in \cB$ such that $\gamma(p^{m}q)-z+p^m\tilde{x}\in \cC^{\rm tors}$ and so $p^mz-p^{m+r}\gamma(q)=p^{m+r}\tilde{x}$. Let $t=p^mz-p^{m+r}\gamma(q)\in \cC$. It has the property that
$(v-\alpha)(t)=(v-\alpha)p^{m+r}(\tilde{x})$ and so $(v-\alpha)(t)^{\rm tf}=0$, i.e.
$(v-\alpha)(t)\in \cC^{\rm tors}$. Therefore $(v-\alpha)(p^rt)=0$ and $(p^rt)^{\rm tf}=p^{m+2r}x.$
\end{proof}

\begin{proof} ({\em of Theorem \ref{thm:mainappendix}}) \enspace

Let $\nu\in \N$ be large enough so that Lemma \ref{lemma:filtration} is satisfied for the triple 
$\Bigl({\rm H}^1\bigl((\cX_\infty^{(u)})_\proket, \fFil^\nu\bigr), v=P(U_p), \alpha \Bigr)$
and consider the exact sequence of sheaves on the pro-Kummer \'etale site  $\cX_{\infty}^{(u)})_\proket$:
$$
0\lra \fFil^\nu\lra \fD_k\lra \fD_k/\fFil^\nu\lra 0,
$$
where we use the notations introduced before Lemma \ref{lemma:quotient}.
It induces the long exact sequence of  pro-Kummer \'etale cohomology groups which, in order to simplify notations we write ${\rm H}^\bullet(F)$ instead of ${\rm H}^\bullet\bigl((\cX_\infty^{(u)})_\proket, F\bigr)$, where $F$ is a sheaf on $(\cX_\infty^{(u)})_\proket$.

$$
\begin{array}{ccccccccc}
\cA={\rm H}^0(\fD_k/\fFil^\nu)&\stackrel{\beta}{\rightarrow}&\cB={\rm H}^1(\fFil^\nu)&\stackrel{\gamma}{\rightarrow} &\cC={\rm H}^1(\fD_k)&\stackrel{\delta}{\rightarrow} &\cD={\rm H}^1(\fD_k/\fFil^\nu)&\stackrel{\epsilon}{\rightarrow}&\cE={\rm H}^2(\fFil^\nu).
\end{array}
$$
We'd like to show the $v$-module $\cC$ satisfies the property $(\ast)$.

1) Let $x\in \cC^{v=\alpha}$ such that $x$ is a $p$-power torsion element. Then $\delta(x)\in 
\cD^{v=\alpha}$ is a $p$-power torsion element and let $s=s(\cD)\in \N$ be such that $p^s\delta(x)=0$.
Then let $y\in \cB$ be such that $\gamma(y)=p^sx$. We also have
$\gamma\bigl((v-\alpha)y\bigr)=(v-\alpha)\gamma(y)=
(v-\alpha)(p^sx)=0$. Therefore let $z\in \cA$ be such that $\beta(z)=(v-\alpha)(y)$.

At this point we recall the following result from \cite{serre}, Proposition 12, and \cite{coleman}, Propositions A.4.2 \& A.4.3: as 
$\cA[1/p]$ is a finitely generated, free $R[1/p]$-module and $U_p$ is completely continuous on it, 
after localizing $B$ to $B'$ and replacing $R$ by $R'$, there are $d=d(\cA,\alpha), e=e(\cA,\alpha)\in \N$ and
$e_\alpha=p^dP_\alpha(U_p)$, with $p^dP_\alpha(T)$ a series with integral coefficients such that
for every $z\in \cA$ we denote by $z_\alpha:=e_\alpha z\in \cA$ and by $z^\perp:=p^dz-z_\alpha\in \cA$. Then $(v-\alpha)(z_\alpha)=0$.
Moreover there is a $w_\alpha\in \cA$ with $(v-\alpha)(w_\alpha)=p^ez_\alpha^\perp$. 

Let now $e,d$ be as above then $p^{d+e}z=p^ez_\alpha+(v-\alpha)(w_\alpha)$.
Therefore we have:
$$
(v-\alpha)\bigl(p^{d+e}y-\beta(w_\alpha)   \bigr)=p^e\beta(z_\alpha),\ \mbox{and so we have}  \ (v-\alpha)^{2}\bigl(p^{d+e}y-\beta(w_\alpha ) \bigr)=0.
$$

If we set $m:=p^{d+e}y-\beta(w_\alpha)$ we have $(v-\alpha)^{2}(m)=0$ and $\gamma(m)=p^{d+e}\gamma(y)=p^{d+e+s}x$. As $(v-\alpha)^{2}m=0$ we have $\alpha^{2}m=0$ therefore
$\alpha^{2}p^{d+e+s}x=0$. This concludes 1) of property $(\ast)$ for $\cC$.

2) Let $x\in \bigl(\cC^{\rm tf}\bigr)^{v=\alpha}$. Then there is $r:=r(\cD,\alpha)$ and $y\in\cD^{v=\alpha}$ such that $y^{\rm tf}=p^r\delta(x)$, where we denoted $y^{\rm tf}$ 
the image of $y$ in $\cD^{\rm tf}$. The image $\epsilon(y)\in \cE^{v=\alpha}$ is annihilated by $\alpha$
and therefore there is $z\in \cC$ such that $\delta(z)=\alpha y$. There is $w\in \cD$ such that $(v-\alpha)(z)=\gamma(w)$ and let  $q\in \cB$ be such that $\alpha w=(v-\alpha)(q)$. Therefore $(v-\alpha)(\alpha z-\gamma(q))=0$.
Let $\tilde{x}:=\alpha z-\gamma(q)\in \cC^{v=\alpha}$. Then $\delta(\tilde{x})=\delta(\alpha z)=\alpha^2 y$.
The image $\alpha^2 y^{\rm tf}$ of $\alpha^2 y$ in $(\cD[1/p])^{v=\alpha}$ is $\alpha^2 p^r \delta(x)$. But $\delta$ induces an isomorphism $\displaystyle \bigl(\cC[1/p]  \bigr)^{v=\alpha}\stackrel{\delta}{\cong} \bigl(\cD[1/p]  \bigr)^{v=\alpha}$ (see \cite{half}). Therefore the image $(\tilde{x})^{\rm tf}$ of $\tilde{x}$ in $\bigl(\cC[1/p]  \bigr)^{v=\alpha}$ is $\alpha^2p^r x$ which 
proves the claim.

\end{proof}


\begin{thebibliography}{99}


\bibitem[AI2]{andreatta_iovita_triple} F.~Andreatta, A.~Iovita, {\em Triple product $p$-Adic $L$-functions associated to a triple of finite slope $p$-adic families of modular forms.} Preprint (2018).

\bibitem[AI1]{andreatta_iovita_ss} F.~Andreatta, A.~Iovita, {\em Semisatble Sheaves and Comparison Isomorphisms in the Semistable Case.}   Rend. Semin. Mat. Univ. Padova \textbf{128} (2012), 131--285.

\bibitem[AIP]{AIP} F.~Andreatta, A.~Iovita, V.~Pilloni, {\em $p$-adic families of Siegel modular cuspforms.} Ann. of Math. \textbf{181} (2015),  623--697. 

\bibitem[ICM]{ICM_AIP} F.~Andreatta, A.~Iovita, V.~Pilloni, {\em  $p$-adic variation of automorphic sheaves}. Proceedings of the International Congress of Mathematicians - Rio de Janeiro 2018. Vol. II. Invited lectures, 249--276, World Sci. Publ., Hackensack, NJ, 2018.


\bibitem[AIS1]{EichlerShimura} F.~Andreatta, A.~Iovita, G.~Stevens, {\em  Overconvergent Eichler-Shimura isomorphisms}, Journal of the Institute of Mathematics of Jussieu \textbf{14} (2015), pp. 221-274.

\bibitem[AIS2]{half} F.~Andreatta, A.~Iovita, G.~Stevens, {\em  A $0.5$ overconvergent Eichler-Shimura isomorphism}, Annales math\'ematiques du Qu\'ebec \textbf{40} (2016), pp. 121-148.




 
\bibitem[BG]{BarreraShan} D.~Barrera, S. Gao, {\em Overconvergent Eichler-Shimura isomorphisms for quaternionic modular forms over $\Q$}, International Journal of Number Theory {\bf 13} (2017), 2487-2504.

\bibitem[BMS]{bhatt_morrow_scholze} B. Bhatt, M. Morrow, P. Scholze, {\em Integral $p$-adic Hodge theory}, ArXiv:1602.03148v3, (2019).

\bibitem[BS]{bhatt_scholze} B. Bhatt, P. Scholze, {\em Prisms and Prismatic cohomology}, preprint, (2019).


\bibitem[CS]{CaraianiScholze}
A.~Caraiani, P.~Scholze,  {\em On the generic part of the cohomology of
compact unitary Shimura varieties}, Ann. of Math. \textbf{186} (2017), 649–766.


\bibitem[CK]{cesnavicius_koshikawa} K. Cesnavicius, T. Koshikawa, {\em The $A_{\rm inf}$-cohomology in the semi-stable case}, Comp. Math. 155, No. 11, 2039-2128, (2019). 

\bibitem[CHJ]{CHJ} P.~Chojecki, D.~Hansen, C.~Johansson, {\em Overconvergent modular forms and perfectoid Shimura curves} 
Documenta Math. {\bf 22}, 191-262.


\bibitem[Co]{coleman} R. Coleman, {\em $p$-Adic Banach spaces and families of modular forms.} Invent. Math. \textbf{127} (1997),  417--479.

\bibitem[CDN]{colmez_dospinescu_niziol} P. Colmez, G. Dospinescu, W.Niziol, {\em Intregral \'etale cohomology of Drinfeld symmetric spaces}, preprint, (2019).

\bibitem[DLLZ1]{Diao}
H.~Diao, K.-W.~Lan, R.~Liu, X.~Zhu: {\em Logarithmic adic spaces: some foundational results},  arXiv:1912.09836, Preprint (2019).

\bibitem[DLLZ2]{DLLZ}
H.~Diao, K.-W.~Lan, R.~Liu, X.~Zhu: {\em Logarithmic  Riemann-Hilbert correspondences for rigid varieties}, arXiv:1803.05786, Preprint (2019).



\bibitem[Hu]{Huber}
R.~Huber, {\em \'Etale Cohomology of Rigid Analytic Varieties and Adic Spaces}, Aspects of Mathematics \textbf{30}.


\bibitem[Il]{Illusie} L.~Illusie, {\em An overview of the work
of K. Fujiwara, K. Kato and C. Nakamura on
logarithmic \'etale cohomology}, Ast\'erisque {\bf 279} (2002),
271-322.

\bibitem[Ja]{Ja} U.~Jannsen,  \emph{Continuous \'etale cohomology}. Math. Ann. \textbf{280} (1988), 207--245. 

\bibitem[PS]{PS} V.~Pilloni, B.~Stroh, \emph{Cohomologie coh\'erente et repr\'esentations Galoisiennes.} Annales math\'ematiques du Qu\'ebec \textbf{40} (2016),  167--202. 


\bibitem[Se]{serre} J.-P. Serre, {\em Endomorphismes compl\`etement continus des espaces de Banach $p$-adiques.}  Inst. Hautes \'Etudes Sci. Publ. Math. \textbf{12} (1962), 69--85.


\bibitem[S1]{ScholzeHodge}
P.~Scholze, {\em $p$-adic Hodge thery for rigid analytic varieties},  Forum Math. $\Pi$ \textbf{1} (2013), 77 pp.

\bibitem[S2]{ScholzeTorsion}
P.~Scholze, {\em On torsion in the cohomology of locally symmetric varieties}, Ann. of Math. \textbf{182} (2015), 945–1066.



\bibitem[SW]{ScholzeWeinstein}
P.~Scholze, J.~Weinstein, {\em Moduli of $p$-divisible groups}, Camb. J. Math. \textbf{1} (2013),  145–237.

\bibitem[TT]{TT} F. Tan, J. Tong {\em Crystalline comparison isomorphisms in $p$-adic Hodge theory: the absolutely unramified case}. Algebra \& Number Theory \textbf{13} (2019), 1509--1581.




\end{thebibliography}
\end{document}